\documentclass[reqno,11pt,a4paper]{amsart}

\usepackage{amsmath,amsthm,amssymb,mathtools}
\usepackage{multicol,multirow,hhline,comment,graphicx,xcolor,url}
\usepackage{mathrsfs}	
\usepackage[margin=1.4in]{geometry}
\usepackage[foot]{amsaddr}
\usepackage{cellspace}		
\usepackage{hyperref}

\usepackage{enumitem} 
\setlist[enumerate,1]{leftmargin=1cm}

\usepackage{accents,stackengine,scalerel} 

\theoremstyle{plain}

\newtheorem{theorem}{Theorem}[section]

\newtheorem{proposition}[theorem]{Proposition}
\newtheorem{corollary}[theorem]{Corollary}
\newtheorem{lemma}[theorem]{Lemma}

\theoremstyle{definition}
\newtheorem{definition}[theorem]{Definition}

\theoremstyle{remark}
\newtheorem{remark}[theorem]{Remark}

\makeatletter
 \DeclareRobustCommand{\checkarg}{\@ifnextchar[{\@witharg}{}}
 \DeclareRobustCommand{\@witharg}[1][]{\ensuremath{(#1)}}
 
 \DeclareRobustCommand{\scaleGen}[1]{\@ifnextchar[{\@scalewithargs{#1}}{\odot^{}_{#1}}}
 \def\@scalewithargs#1[#2][#3]{#2 \odot^{}_{#1} #3}
 
 \DeclareRobustCommand{\blankBinOp}{\@ifnextchar[{\@blankbinopwithargs}{}}
 \def\@blankbinopwithargs[#1][#2]{#1 #2}
\makeatother

\usepackage{xparse}

\makeatletter
\NewDocumentCommand{\myminus}{m}{%
  $\m@th#1{-}$%
}

\makeatother

\def\IPspace{\mathcal{I}}

\def\HIPspace{\IPspace_H}
\def\IPH{\HIPspace}

\def\dH{d_H}
\def\dHs{d'_H}

\def\IPA{\mathcal{I}_{\alpha}}
\def\dIA{d_{\alpha}}

\def\Exc{\mathcal{E}}

\def\mBxcA{\nu_{\BESQ}^{(-2\alpha)}}	

\def\cC{\mathcal{C}}	

\def\HA{\mathcal{N}^{\textnormal{\,sp}}}

\def\SHA{\Sigma(\HA)}

\def\HfinA{\HA_{\textnormal{fin}}}
\def\SHfinA{\Sigma\big(\HfinA\big)}

\def\umCladeA{\nu^{(\alpha)}_{\perp\textnormal{cld}}}	

\def\uF{\bF_{\perp}}	

\def\cN{\mathcal{N}}

\def\cNRE{\mathcal{N}\big([0,\infty)\times\Exc\big)}

\def\cS{\mathcal{S}}

\def\len{\textnormal{len}}			
\def\life{\zeta}					
\def\dis{\textnormal{dis}}			
\def\IPmag#1{\left\|\vphantom{I}#1\right\|}		

\def\skewer{\textsc{skewer}}		
\def\skewerP{\widebar{\skewer}}		
\newcommand{\xiA}{\xi}

\def\cutoffL#1#2{\textsc{cutoff}^{\leq #1}\ensuremath{\left(#2\right)}}
\def\cutoffG#1#2{\textsc{cutoff}^{\geq #1}\ensuremath{\left(#2\right)}}
\def\cutoffLB#1#2{\textsc{cutoff}^{\leq #1}\ensuremath{\big(#2\big)}}
\def\cutoffGB#1#2{\textsc{cutoff}^{\geq #1}\ensuremath{\big(#2\big)}}

\def\Dirac#1{\delta\left( #1 \right)}
\def\DiracBig#1{\delta\big( #1 \big)}	

\def\reverse{\mathcal{R}}							
\def\reverseI{\mathcal{R}_{\textnormal{IP}}}		

\def\scaleB{\scaleGen{\textnormal{sp}}}		
\def\scaleH{\scaleGen{\textnormal{cld}}}		

\def\scaleHA{\odot^{\alpha}_{\textnormal{cld}}}		

\def\ShiftRestrict#1#2{#1\big|^{\leftarrow}_{#2}} 
\def\shiftrestrict#1#2{#1|^{\leftarrow}_{#2}}

\def\Restrict#1#2{#1\big|_{#2}}
\def\restrict#1#2{#1|_{#2}}

\def\Concat{ \mathop{ \raisebox{-2pt}{\Huge$\star$} } }
\def\ConcatIL{ \mbox{\huge $\star$} }
\def\concat{\star}

\def\bN{\mathbf{N}}			
\def\bF{\mathbf{F}}			
\def\bX{\mathbf{X}}			
\def\bn{\mathbf{n}}

\newcommand{\IPLT}{\mathscr{D}}

\def\bff{\mathbf{f}}		

\def\BR{\mathbb{R}}				
\def\BN{\mathbb{N}}				
\def\Leb{\textnormal{Leb}}		

\def\to{\rightarrow}
\def\downto{\downarrow}

\def\cf{\mathbf{1}}				
\def\Pr{\mathbf{P}}				
\def\BPr{\mathbb{P}}			
\def\EV{\mathbf{E}}				
\def\cF{\mathcal{F}}			
\def\cFI{\cF_{\IPspace}}		

\def\cA{\mathcal{A}}	

\def\distribfont#1{\texttt{\upshape #1}}

\def\ExpDist{\distribfont{Exponential}\checkarg}
\def\GammaDist{\distribfont{Gamma}\checkarg}

\def\BetaDist{\distribfont{Beta}\checkarg}

\def\PoiDir{\distribfont{PD}\checkarg}
\def\PoiDirAT{\ensuremath{\PoiDir(\alpha,\theta)}}

\def\PRM{\distribfont{PRM}\checkarg}
\def\PRMLBA{\ensuremath{\distribfont{PRM}\big(\Leb\otimes\mBxcA\big)}}

\def\Stable{\distribfont{Stable}\checkarg}

\def\BESQ{\distribfont{BESQ}\checkarg}

\def\PDIP{\distribfont{PDIP}\checkarg}
\def\PDIPAT{\ensuremath{\PDIP(\alpha,\theta)}}

\stackMath
\newcommand{\cev}[1]{\ThisStyle{
	\stackengine{0ex}{\SavedStyle#1}{\SavedStyle
		\scaleobj{0.63}{\leftharpoonup}}{O}{c}{F}{\useanchorwidth}{S}}}

\makeatletter
\let\save@mathaccent\mathaccent
\newcommand*\if@single[3]{%
  \setbox0\hbox{${\mathaccent"0362{#1}}^H$}%
  \setbox2\hbox{${\mathaccent"0362{\kern0pt#1}}^H$}%
  \ifdim\ht0=\ht2 #3\else #2\fi
  }
\newcommand*\rel@kern[1]{\kern#1\dimexpr\macc@kerna}
\newcommand{\widebar}{}
\DeclareRobustCommand*\widebar[1]{\@ifnextchar^{\wide@bar{#1}{0}}{\wide@bar{#1}{1}}}
\newcommand*\wide@bar[2]{\if@single{#1}{\wide@bar@{#1}{#2}{1}}{\wide@bar@{#1}{#2}{2}}}
\newcommand*\wide@bar@[3]{%
  \begingroup
  \def\mathaccent##1##2{%
    \let\mathaccent\save@mathaccent
    \if#32 \let\macc@nucleus\first@char \fi
    \setbox\z@\hbox{$\macc@style{\macc@nucleus}_{}$}%
    \setbox\tw@\hbox{$\macc@style{\macc@nucleus}{}_{}$}%
    \dimen@\wd\tw@
    \advance\dimen@-\wd\z@
    \divide\dimen@ 3
    \@tempdima\wd\tw@
    \advance\@tempdima-\scriptspace
    \divide\@tempdima 10
    \advance\dimen@-\@tempdima
    \ifdim\dimen@>\z@ \dimen@0pt\fi
    \rel@kern{0.6}\kern-\dimen@
    \if#31
      \overline{\rel@kern{-0.6}\kern\dimen@\macc@nucleus\rel@kern{0.4}\kern\dimen@}%
      \advance\dimen@0.4\dimexpr\macc@kerna
      \let\final@kern#2%
      \ifdim\dimen@<\z@ \let\final@kern1\fi
      \if\final@kern1 \kern-\dimen@\fi
    \else
      \overline{\rel@kern{-0.6}\kern\dimen@#1}%
    \fi
  }%
  \macc@depth\@ne
  \let\math@bgroup\@empty \let\math@egroup\macc@set@skewchar
  \mathsurround\z@ \frozen@everymath{\mathgroup\macc@group\relax}%
  \macc@set@skewchar\relax
  \let\mathaccentV\macc@nested@a
  \if#31
    \macc@nested@a\relax111{#1}%
  \else
    \def\gobble@till@marker##1\endmarker{}%
    \futurelet\first@char\gobble@till@marker#1\endmarker
    \ifcat\noexpand\first@char A\else
      \def\first@char{}%
    \fi
    \macc@nested@a\relax111{\first@char}%
  \fi
  \endgroup
}
\makeatother

\newcommand{\sskewer}{\textsc{sSkewer}}
	
\numberwithin{equation}{section}
\numberwithin{figure}{section}
\numberwithin{table}{section}

\newcommand\PoiIPP[1]{self-similar \ensuremath{(#1)} interval partition evolution}
\newcommand\CPoiIPP[1]{Self-similar \ensuremath{(#1)} interval partition evolution}
\newcommand\PoiIPPAT{\PoiIPP{\alpha,\theta}}

\newcommand\SSIPE[1]{\distribfont{SSIPE}\ensuremath{(#1)}}
\newcommand\SSIPEAT{\ensuremath{\distribfont{SSIPE}(\alpha,\theta)}}
\newcommand\SSIPEA{\distribfont{SSIPE}\ensuremath{(\alpha,0)}}

\newcommand\SSIPEzAT[1]{\ensuremath{\distribfont{SSIPE}_{#1}(\alpha,\theta)}}
\newcommand\SSIPEzA[1]{\ensuremath{\distribfont{SSIPE}_{#1}(\alpha,0)}}

\newcommand\PDIPEAT{\distribfont{PDIPE}\ensuremath{(\alpha,\theta)}}

\newcommand{\fskewer}{\textsc{fSkewer}}
\def\umCladeAbar{\overline{\nu}^{(\alpha)}_{\perp\textnormal{cld}}}	

\newcommand{\IPAo}{\mathcal{I}_{\alpha,1}}
\newcommand{\IPHo}{\mathcal{I}_{H,1}}
\newcommand{\IPHos}{\overline{\mathcal{I}}_{H,1}}
\newcommand{\IPHs}{\mathcal{I}_{H}^*}

\begin{document}

\ \vspace{-22pt}

\title[Two-parameter interval partition diffusions]{Diffusions on a space of interval partitions:\\ The two-parameter model}

\author[N.~Forman]{Noah Forman$^1$}
\address{$^1$ Department of Mathematics \& Statistics\\ McMaster University\\ Hamilton, ON L8S 4K1 \\ Canada}
\email{noah.forman@gmail.com}
\author[D.~Rizzolo]{Douglas Rizzolo$^2$}
\address{$^2$ Department of Mathematics\\ University of Delaware\\ Newark, DE 19716\\ USA}
\email{drizzolo@udel.edu}
\author[Q.~Shi]{Quan Shi$^3$}
\address{$^3$ Academy of Mathematics and Systems Science\\ Chinese Academy of Sciences\\ Beijing 100190\\ China}
\email{quan.shi@amss.ac.cn}
\author[M.~Winkel]{Matthias Winkel$^4$}
\address{$^4$ Department of Statistics\\ University of Oxford\\ Oxford, OX1 3LB\\ UK}
\email{winkel@stats.ox.ac.uk}

\begin{abstract}
We introduce and study interval partition diffusions with Poisson--Dirichlet$(\alpha,\theta)$ stationary distribution for parameters $\alpha\in(0,1)$ and $\theta\ge 0$. This extends previous work on the cases $(\alpha,0)$ and $(\alpha,\alpha)$ and builds on our recent work on measure-valued diffusions. Our methods for dealing with general $\theta\ge 0$ allow us to strengthen previous work on the special cases to include initial interval partitions with dust. 
In contrast to the measure-valued setting, we can show that this extended process is a Feller process improving on the Hunt property established in that setting. 
These processes can be viewed as diffusions on the boundary of a branching graph of integer compositions. Indeed, by studying their infinitesimal generator on 
suitable quasi-symmetric functions, we relate them to diffusions obtained as scaling limits of composition-valued up-down chains. 
\end{abstract}         

\thanks{This research is partially supported by NSF grants {DMS-1204840, DMS-1444084, DMS-1855568}, UW-RRF grant A112251, EPSRC grant EP/K029797/1, NSERC RGPIN-2020-06907, SNSF grant P2ZHP2\_171955, NSFC grant 12288201.}

\maketitle

\ \vspace{-28pt}

\section{Introduction and main results}\label{sec:intro}

In this paper, we introduce and study a two-parameter family of interval partition diffusions such that, for each choice of parameters, $\alpha\in(0,1)$ and
$\theta\ge 0$, the stationary distribution is the corresponding two-parameter Poisson--Dirichlet interval partition, ${\tt PDIP}(\alpha,\theta)$. The members of this two-parameter family arise as the unique regenerative ordering of the coordinates of the Poisson--Dirichlet distribution, ${\tt PD}(\alpha,\theta)$, on the Kingman simplex. Gnedin and Pitman introduced \PDIPAT\ through an underlying family of regenerative composition structures \cite{GnedPitm05}.

The cases $\theta=0$ and $\theta=\alpha$ of our construction were introduced in \cite{Paper1-2}. A full two-parameter family of measure-valued Poisson--Dirichlet diffusions arising from a variant of our construction was introduced in \cite{FVAT}. Motivated in part by an earlier version of the present paper and by 
scaling limit results conjectured in \cite{RogeWink20}, a two-parameter family was obtained as a scaling limit in \cite{RivRiz20}. Encouraged by a referee, we establish in this version of the present paper that the two families coincide.


There is a long history of interest in dynamics preserving Poisson--Dirichlet and related distributions.  Of particular relevance for us, Ethier and Kurtz's \cite{EthiKurt81} infinitely-many-neutral-alleles diffusion model has stationary distribution $\PoiDir(0,\theta)$.  Recently, using techniques developed by Borodin and Olshanski \cite{BoroOlsh09} to study diffusive limits of random walks on partitions, Petrov \cite{Petrov09} constructed a two-parameter family of diffusions on the Kingman simplex with \PoiDirAT\ stationary distributions. This extends Ethier and Kurtz's model. These Ethier--Kurtz--Petrov (\texttt{EKP}$(\alpha,\theta)$) diffusions have also been studied e.g.\ in \cite{CdBERS17,Ethier14,FengSun10,FengSun19,RuggWalk09,RuggWalkFava13}.

\subsection{Kernels with a branching property and immigration} In \cite{Paper1-1}, we constructed branching interval partition diffusions from the level sets of marked L\'evy processes. In 
\cite{Paper1-2}, we computed their semigroups and adapted Shiga's \cite{Shiga1990} construction of Fleming--Viot superprocesses by normalization/time-change  to obtain $(\alpha,0)$- (and $(\alpha,\alpha)$-)interval partition diffusions. While our toolkit derives almost exclusively from marked L\'evy processes, the semigroup of transition kernels is more easily generalized to the two-parameter setting. 

\begin{definition}\label{def:IP_1}
	An \emph{interval partition} is a set $\beta$ of disjoint, open subintervals of some interval $[0,M]$, that cover $[0,M]$ up to a Lebesgue-null set. We write $\IPmag{\beta}$ to denote $M$. We refer to the elements of $\beta$ as its \emph{blocks}. The Lebesgue measure of a block is called its \emph{mass}.  The set of all interval partitions is denoted by $\HIPspace$. 
\end{definition}

Fix $\alpha\in(0,1)$ and $\theta\ge 0$. We define kernels $(\kappa_y^{\alpha,\theta},\,y\ge 0)$ on $\HIPspace$. These kernels possess a \emph{branching property} under which the state at time $y$ can be seen as the concatenation of a family of independent interval partitions indexed by the blocks of the initial interval partition, along with one additional independent interval partition representing \em immigration\em; see Figure \ref{fig:semigroup}. This generalizes the cases $\theta=0$ without immigration, and $\theta=\alpha$ with a specific immigration parameter, of \cite{Paper1-2}. In the general case, this is based on the Poisson--Dirichlet$(\alpha,\theta)$ interval partitions, ${\tt PDIP}(\alpha,\theta)$, of \cite{GnedPitm05}, which we recall in Section \ref{sec:IPspace}.
\begin{figure}[t]
 \centering
 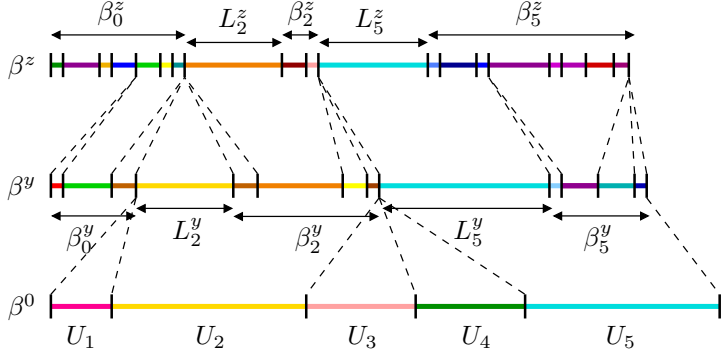
 \caption{Illustration of the transition kernel $\kappa_y^{\alpha,\theta}$. The initial state $\beta^0$ has five blocks: $U_1,\ldots,U_5$. Some contribute $\varnothing$ for time $y$, here
   $U_1$, $U_3$ and $U_4$. The others each contribute a leftmost block $L_j^y$ and a partition $\beta_j^y=G_j\bar{\beta}_j$ of further blocks. If we add ``immigration'' $\beta_0^y=G^y\bar{\beta}_0$, then $\beta^y=\beta_0^y\concat(0,L_2^y)\concat\beta_2^y\concat(0,L_5^y)\concat\beta_5^y$. The semigroup property requires consistency of a transition from 0 to $z$ 
   and composition of the transitions from $0$ to $y$ and from $y$ to $z$.\label{fig:semigroup}\vspace{-0.1cm}}
\end{figure}
 
To be more precise, let us begin by formalizing this concatenation. We call the family $(\gamma_U)_{U\in\beta}$ in $\mathcal{I}_H$ indexed by $\beta\in\mathcal{I}_H$ \emph{summable} if $\sum_{U\in\beta}\IPmag{\gamma_U} < \infty$. We then define
$S(U) := \sum_{U^\prime=(u^\prime,v^\prime)\in\beta\colon u^\prime<u}\IPmag{\gamma_{U^\prime}}$ for 
$U = (u,v) \in \beta$, and the \emph{concatenation}
\begin{equation}\label{eq:IP:concat_def}
 \Concat_{U\in\beta}\gamma_U := \{(S(U)+x,S(U)+y)\colon\ U\in\beta,\ (x,y)\in \gamma_U\}.\vspace{-0.2cm}
\end{equation} 
We also write $\gamma\concat\gamma^\prime$ to concatenate two interval partitions. 
For $c\ge0$ let $c\gamma$ denote the interval partition obtained by multiplying each block in $\gamma$ by $c$.

Now, fix $b>0$, $r>0$. We consider an independent triple $\big(G,\bar\beta,L_{b,r}^{(\alpha)}\big)$ where $G\sim \GammaDist[\alpha,r]$, $\bar\beta\sim \PDIP[\alpha,\alpha]$, and $L_{b,r}^{(\alpha)}$ is $(0,\infty)$-valued with Laplace transform
\begin{equation}
  \EV\left[e^{-\lambda L_{b,r}^{(\alpha)}}\right] = \left(\frac{r+\lambda}{r}\right)^{\alpha}\frac{e^{br^2/(r+\lambda)}-1}{e^{br}-1}.
  \label{LMBintro}
\end{equation}
Then we define the distribution $\mu_{b,r}^{(\alpha)}$ of a random interval partition as 
\begin{equation}
 \mu_{b,r}^{(\alpha)}=e^{-br}\delta_\varnothing+(1\!-\!e^{-br})\Pr\left((0,L_{b,r}^{(\alpha)})\concat G\bar{\beta}\in\cdot\,\right).
\label{blockdesc}
\end{equation}
The idea is that for $r>0$, under $\mu_{b,r}^{(\alpha)}$, we associate with a block of size $b$ either no descendants (the empty interval partition $\varnothing$) with probability $e^{-br}$; or the descendants comprise one block of size $L_{b,r}^{(\alpha)}$ followed by the blocks of $G\bar\beta$.

\begin{definition}[$\kappa^{\alpha,\theta}_y(\beta,\,\cdot\,)$, $y>0$, $\beta\in\HIPspace$]\label{def:kernel:sp}
 Fix $\alpha\in (0,1)$, $\theta\ge 0$ and let $\beta\in\IPH$ and $y>0$. Then $\kappa^{\alpha,\theta}_y(\beta,\,\cdot\,)$ is defined to be the distribution of
$G^y\bar{\beta}_0\concat\Concat_{U\in\beta}\beta_U^y$ for independent $G^y\sim\GammaDist[\theta,1/2y]$, $\bar{\beta}_0\sim\PDIP(\alpha,\theta)$, and $\beta_U^y\sim \mu_{{\rm Leb}(U),1/2y}^{(\alpha)}$, $U\in\beta$, where ``Leb'' denotes Lebesgue measure.
\end{definition}

It is straightforward to see that only finitely many blocks $U\in\beta$ have descendants under $\mu_{{\rm Leb}(U),1/2y}^{(\alpha)}$ and hence that 
$\kappa_y^{\alpha,\theta}(\beta,\,\cdot\,)$ is well-defined. It is not obvious from this definition that these kernels form a transition semigroup. 

\begin{proposition}	\label{prop:sp}
	Fix $\alpha\in (0,1)$ and $\theta\ge 0$.
	\begin{enumerate}[label=(\roman*),ref=(\roman*)]
		\item \emph{Diffusion properties}:  \label{item:SSIP:diffusion}
		The family $(\kappa^{\alpha,\theta}_y,\, y\ge 0)$ forms the transition semigroup of a path-continuous $\HIPspace$-valued Hunt process. \smallskip
		\item \emph{Self-similarity}:  \label{item:SSIP:SS}
		If $(\beta^y,\,y\geq 0)$ is an instance of this Hunt process then so is $(c \beta^{y/c},\,y\geq 0)$ for any $c>0$.\smallskip
		\item \emph{Total mass}:  \label{item:SSIP:mass}
		The associated total mass process, $(\|\beta^y\|,\,y\geq 0)$, is a $\BESQ(2\theta)$, i.e.\ a $2\theta$-dimensional squared Bessel process.\smallskip
		\item \emph{Pseudo-stationarity}:  \label{item:SSIP:pseudostat}
		Consider an independent pair of $\bar\beta\sim\PDIP[\alpha,\theta]$ and $Z = (Z(y),y\ge 0)\sim\BESQ_{\mu}(2\theta)$, with an arbitrary initial distribution $\mu$. If the aforementioned Hunt process $(\beta^y,\,y\geq 0)$ has initial distribution $\beta^0\stackrel{d}{=}Z(0) \bar\beta$ then for each fixed $y\ge 0$ we have $\beta^y \stackrel{d}{=} Z(y) \bar\beta$.
	\end{enumerate}
\end{proposition}

In the language of \cite{Lamperti72}, these processes are 1-self-similar.

\subsection{Starting from dust}\label{sec:main_results}
One drawback of the state space $\HIPspace$, equipped with the Hausdorff metric $\dH$ between complements of the form $C(\beta):=[0,\|\beta\|]\setminus\bigcup_{U\in\beta}U$,
is that it is not complete. E.g., $\beta_n=\{(0,1/n),(1/n,2/n),\ldots,((n-1)/n,1)\}$, $n\ge 1$, is a Cauchy sequence that does not converge in $\HIPspace$. Indeed, $C(\beta_n)$ approaches $[0,1]$ under the Hausdorff metric $d_H$. To address this, we now extend the evolutions of Proposition \ref{prop:sp} to a completion of $(\HIPspace,d_H\circ C)$. The elements of this completion may be thought of as ``generalized interval partitions'' in which blocks $U\in\beta$ must still be disjoint open subsets of $[0,M]$, but $C(\beta,M):=[0,M]\setminus\bigcup_{\beta\in U}U$ need not be Lebesgue-null. We think of this residual, left-out mass $C(\beta,M)$ as ``dust'' \cite{Haas04}. 

More precisely, let $\mathcal{K}$ be the space of compact subsets of $[0,\infty)$ that contain $0$, equipped with the Hausdorff metric $d_H$. 
We associate with any $K\in\mathcal{K}$ the family $\beta(K)$ of open intervals formed by the connected components of $[0,\max K]\setminus K$.
Define $\IPHs$ to be the image of $\mathcal{K}$ under the map $\beta^*(K) := (\beta(K),\max K)$. The map $\beta^*$ is a bijection from $\mathcal{K}$ to $\IPHs$. Then $\IPHs$ is a set of ordered pairs $(\beta,M)$ in which $M\in [0,\infty)$ and $\beta$ is a family of disjoint open subintervals of $[0,M]$. This space can be viewed as an extension of $\IPH$ via the inclusion map $\beta \mapsto (\beta,\|\beta\|)$ from $\IPH$ to $\IPHs$. We refer to elements of $\IPHs$ as \emph{generalized interval partitions}.

Now $C\colon \IPHs\to \mathcal{K}$ given by $C(\beta,M) = [0,M]\setminus\bigcup_{U\in\beta}U$, is the inverse of the map $\beta^*\colon\mathcal{K}\rightarrow\IPHs$.
Equip both $\IPH$ and $\IPHs$ with the metric $d_H\circ C$. 
It is well-known that $(\mathcal{K},d_H)$ is locally compact and separable. 
Then it is easy to see that 
$(\mathcal{K},d_H)$, or equivalently $(\IPHs,d_H\circ C)$, is a metric completion of $(\mathcal{I}_H,d_H\circ C)$ in that every $K\in\mathcal{K}$ is a $d_H$-limit of some  
$C(\beta_n)$, $\beta_n\in \mathcal{I}_H$. 


Consider the distribution $\widetilde{\mu}_{0,r}^{(\alpha)}$ of $G_0\bar{\beta}_0$ for independent $G_0\sim\mathtt{Exponential}(r)$ and $\bar{\beta}_0\sim\PDIP(\alpha,0)$.
Since $L_{b,r}^{(\alpha)}$ tends to $\mathtt{Gamma}(\alpha,r)$ in distribution as $b\downarrow 0$, it will follow from Proposition \ref{prop:PDIP} that $\widetilde{\mu}_{0,r}^{(\alpha)}$ is the weak limit of $\mu_{b,r}^{(\alpha)}\big(\,\cdot\,\big|\,\HIPspace\setminus\{\varnothing\}\big)$. 

\begin{definition}[$\kappa^{\alpha,\theta}_y(\beta^*,\,\cdot\,)$, $y>0$, $\beta^*\in\IPHs$]\label{def:kernel:sp2}
	Fix $\alpha\in (0,1)$, $\theta\ge 0$, $\beta^*=(\beta,M)\in\IPHs$ and $y>0$. 
	Set 
	\begin{equation}\label{betay}
	\beta^y := G_0^y\bar{\beta}_0\concat\Concat_{U\in\beta \cup\{\{R_i\},\, i\le J_y\}}\beta_U^y	
	\end{equation}
	 for independent $G_0^y\sim\GammaDist[\theta,1/2y]$, $\bar{\beta}_0\sim\PDIP(\alpha,\theta)$, and $\beta_U^y\sim \mu_{{\rm Leb}(U),1/2y}^{(\alpha)}$,
           $U\in\beta$, as well as, in the case when $m={\rm Leb}(C(\beta^*))>0$, further independent variables $J_y\sim\mathtt{Poisson}(\alpha m/y)$, and $R_i\sim\mathtt{Unif}(C(\beta^*))$ and $\beta_{\{R_i\}}^y\sim\widetilde{\mu}_{0,1/2y}^{(\alpha)}$, $i\ge 1$.
	 Here the concatenation is according to the order of these disjoint sets from left to right. 
	 Then $\kappa^{\alpha,\theta}_y(\beta^*,\,\cdot\,)$ is defined to be the distribution of $(\beta^y, \|\beta^y\|)$. 
\end{definition}

Note that Definition \ref{def:kernel:sp2} specifies kernels $\kappa_y^{\alpha,\theta}$, $y>0$, on $\IPHs$. However, $\kappa^{\alpha,\theta}_y(\beta^*,\,\cdot)$ is, by construction, supported on 
$\HIPspace$ for all $\beta^*\in\IPHs$, in the sense that $\beta^y$ as specified in \eqref{betay} is almost surely in $\HIPspace$. If furthermore, $\beta^*=(\beta,\|\beta\|)\in\IPHs$ is associated with 
$\beta\in\HIPspace$, then the push-forward of $\kappa^{\alpha,\theta}_y(\beta^*,\,\cdot\,)$ under the natural projection from 
$\{(\gamma,M)\in\IPHs\colon\gamma\in\HIPspace,\,M=\|\gamma\|\}$ to $\HIPspace$ is $\kappa_y^{\alpha,\theta}(\beta,\,\cdot\,)$ as defined in Definition \ref{def:kernel:sp}.  This justifies using the same notation for both kernels. As a consequence of the following theorem, we can further use the push-forward of 
$\kappa^{\alpha,\theta}_y(\beta^*,\,\cdot\,)$ for general $\beta^*\in\IPHs$ as an entrance law for the Hunt process of Proposition \ref{prop:sp}.

\begin{theorem}	\label{thm:sp}
	Fix $\alpha\in (0,1)$ and $\theta\ge 0$. Then the family $(\kappa^{\alpha,\theta}_y,\, y\ge 0)$ forms the transition semigroup of a path-continuous self-similar Feller process on $(\IPHs,d_H\circ C)$.
\end{theorem}

\begin{definition}\label{def:SSIPE*}
	We refer to the $\HIPspace$-valued Hunt processes of Proposition \ref{prop:sp} as \emph{self-similar $(\alpha,\theta)$-interval partition evolutions}, or  ${\tt SSIPE}(\alpha,\theta)$. 
           We write  ${\tt SSIPE}_\beta(\alpha,\theta)$ for the distribution of an  ${\tt SSIPE}(\alpha,\theta)$ starting from $\beta\in\HIPspace$.
          The corresponding $\IPHs$-valued Feller processes of Theorem \ref{thm:sp} are denoted by  ${\tt SSIPE}_{\beta^*}(\alpha,\theta)$, $\beta^*\in\IPHs$.
\end{definition}

Theorem \ref{thm:sp} establishes, in particular, the extension of the Hunt processes of \cite{Paper1-2} to Feller processes on a completion of the original state space. Also generally in the two-parameter setting, this Feller property opens up analytical techniques not previously available. This has already been exploited in \cite{Paper1-3} to connect to ${\tt EKP}(\alpha,\theta)$-diffusions, and this is also the key for us, in this paper, to connect our construction with its understanding of path properties and genealogical structure to the interval partition evolutions constructed in \cite{RivRiz20}.

\subsection{Unit-mass processes}\label{sec:unitmass} The set of all interval partitions of $[0,1]$ is denoted by $\IPHo := \{\beta\in\IPH \colon \|\beta\|=1\}$. While $\IPHs$ includes a completion of $\mathcal{I}_{H,1}$, it is more natural to consider the set $\IPHos$ of collections of disjoint open subintervals of $[0,1]$ equipped with the metric $d_H\circ C_1$, where
  $C_1(\beta)=[0,1]\setminus\bigcup_{U\in\beta}U$, again including the case of dust by allowing $C_1(\beta)$ to have positive Lebesgue measure. We refer to elements of $\IPHos$ as \emph{generalized partitions of $[0,1]$}. Then $\gamma\mapsto (\gamma,1)$ isometrically embeds $\IPHos$ in $\IPHs$. 
The space $(\IPHos,d_H\circ C_1)$ is a metric completion of $\mathcal{I}_{H,1}$.

\begin{definition}[De-Poissonization, \PDIPEAT]\label{def:depoiss}
	Let $\mu$ be a probability measure on $\IPHos$ and $\boldsymbol{\beta}:= (\beta^y, y\ge 0)$ an $\mathtt{SSIPE}(\alpha,\theta)$ with initial distribution $\mu$. We define the time-change
	\begin{equation}\label{eq:tau-beta}
		\tau_{\boldsymbol{\beta}}(u):= \inf \left\{ y\ge 0\colon \int_0^y \|\beta^z\|^{-1} d z>u \right\}, \qquad u\ge 0.
	\end{equation}
	The map from $\boldsymbol{\beta}$ to the process
	\[
	\bar{\beta}^u:= \big\| \beta^{\tau_{\boldsymbol{\beta}}(u)} \big\|^{-1}  \beta^{\tau_{\boldsymbol{\beta}}(u)},\quad u\ge 0,
	\]
	is called \emph{de-Poissonization}. The resulting process takes values in $\IPHos$ and is called a \emph{Poisson--Dirichlet$(\alpha, \theta)$ interval partition evolution}, or a \PDIPEAT. 
\end{definition}

It follows from Proposition \ref{prop:sp}\ref{item:SSIP:mass}, from the fact that $\kappa_y^{\alpha,\theta}(\beta^*,\,\cdot\,)$ is concentrated on $\mathcal{I}_H$ for all $y>0$ and $\beta^*\in\mathcal{I}_H^*$, and from well-known properties of squared Bessel processes, e.g.\ in \cite[p.\ 314-5]{GoinYor03}, that $\tau_{\boldsymbol{\beta}}$ is a.s.\ well-defined on all $u\ge0$ and continuous and strictly increasing for all $\theta \ge 0$, with 
\begin{equation}\label{eq:tau-beta-prop}
	\lim_{u\uparrow \infty} \tau_{\boldsymbol{\beta}}(u) = \inf \{y>0 \colon \beta^y = \varnothing \} < \infty \quad \text{if }\theta < 1. 
\end{equation}


\begin{theorem}\label{thm:dP-IP}
	\PDIPEAT\ is a path-continuous Feller process on $(\IPHos,d_H\circ C_1)$ with stationary law \PDIPAT.
\end{theorem}



Indeed, in Theorem \ref{thm:twopargen}, we identify the generator of ${\tt PDIPE}(\alpha,\theta)$.

We point out that the corresponding extension of Fleming--Viot processes of \cite{FVAT} to include dust can be considered, but it is not continuous in the initial state, hence not Feller. 
Recall that the Ethier--Kurtz--Petrov (${\tt EKP}(\alpha,\theta)$) diffusion is stationary with the law ${\tt PD}(\alpha,\theta)$. This law is supported on the Kingman simplex: the set of non-increasing sequences in $[0,1]$ that sum to exactly $1$. However, this process can enter continuously from any state in the closure of the simplex: the set of sequences with sum at most 1. Petrov left open the question of whether the ${\tt EKP}(\alpha,\theta)$ diffusion ever leaves the Kingman simplex (at exceptional times) to visit points in the closure. This question was answered in the negative by Ethier \cite{Ethier14} using analytic methods based on the existence of densities. Our construction allows us to answer the analogous question for ${\tt PDIPE}(\alpha,\theta)$ probabilistically. Together with the results of a companion paper \cite{Paper1-3}, Theorem \ref{thm:dP-IP} gives a probabilistic approach to the main theorem of \cite{Ethier14}.

Specifically, denote by ${\tt RANKED}\colon\IPHos\rightarrow\overline{\nabla}_\infty$ the map that associates with $\beta\in\IPHos$ the 
decreasing sequence of interval lengths in $\beta$, where
\begin{equation*}
 \overline{\nabla}_\infty := \left\{ (x_i)_{i\ge1}\in [0,\infty)^{\mathbb{N}} \colon \sum\nolimits_{i\ge1} x_i\le 1\right\}
\end{equation*}
is equipped with the metric $d_{\infty}((x_i)_{i\ge1},(y_i)_{i\ge1}) = \sup_{i\ge1}|x_i-y_i|$.

\begin{corollary}\label{cor:RANKED} Mapping ${\tt PDIPE}_\beta(\alpha,\theta)$, $\beta\in\IPHos$, under ${\tt RANKED}$ yields a 
  $\overline{\nabla}_\infty$-valued Feller process.
\end{corollary}

We use this in \cite[Theorem 4.2]{Paper1-3} to identify this ranked process as ${\tt EKP}(\alpha,\theta)$, up to linear time-change, by calculating relevant parts of the infinitesimal generator of ${\tt PDIPE}(\alpha,\theta)$.

\subsection{Infinitesimal generators}\label{sec:intro:gen}

Rivera-Lopez and Rizzolo \cite{RivRiz20} recently established a two-parameter Poisson--Dirichlet interval partition diffusion as a scaling limit of up-down Markov chains on integer compositions. Specifically, denote by $\mathcal{C}=\bigcup_{n\ge 0}\mathcal{C}_n$ the space of all integer compositions of any $n\ge 0$, including the unique composition 
$\emptyset\in\mathcal{C}_0$ of 0. For a composition $\sigma=(\sigma_1,\ldots,\sigma_k)\in\mathcal{C}_n$ of $|\sigma|=n$ denote by $\ell(\sigma)=k$ its \em length\em. The diffusion of \cite{RivRiz20} on (a state space isometric to) 
$\overline{\mathcal{I}}_{H,1}$ is characterized by its generator $\mathcal{A}_{\alpha,\theta}$. Specifically, let we $m_\emptyset^\circ=1$, and for 
compositions $\sigma\in\mathcal{C}\setminus\{\emptyset\}$ of $n=|\sigma|\ge 1$ of length $\ell=\ell(\sigma)\in[n]$, we define 
\begin{equation}\label{eq:msigmacirc}
m_\sigma^\circ(\beta)=\sum_{\overset{\scriptstyle U_1,\ldots,U_\ell\in\beta}{\rm strictly\,increasing}}\prod_{j=1}^\ell({\rm Leb}(U_j))^{\sigma_j},\qquad\beta\in\mathcal{I}_{H,1},
\end{equation}
and we extend continuously to $\overline{\mathcal{I}}_{H,1}$. For these functions, the generator acts as   
\[
\mathcal{A}_{\alpha,\theta}m_\sigma^\circ=-|\sigma|(|\sigma|-1+\theta)m_\sigma^\circ+\sum_{j\colon\sigma_j\ge 2}\sigma_j(\sigma_j-1-\alpha)m_{\sigma-\square_j}^\circ+\sum_{j\colon\sigma_j=1}\eta_jm_{\sigma\ominus\square_j}^\circ,
\]
where we write $\sigma-\square_j=(\sigma_1,\ldots,\sigma_{j-1},\sigma_j-1,\sigma_{j+1},\ldots,\sigma_k)$ when $\sigma_j\ge 2$ and 
$\sigma\ominus\square_j=(\sigma_1,\ldots,\sigma_{j-1},\sigma_{j+1},\ldots,\sigma_k)$ when $\sigma_j=1$, and where $\eta_1=\theta$ and $\eta_j=\alpha$ for $j\ge 2$.
It was shown in \cite{RivRiz20} that the linear span of $\{m_\sigma^\circ,\sigma\in\mathcal{C}\}$ is a core of $\mathcal{A}_{\alpha,\theta}$.

\begin{theorem}\label{thm:twopargen} The process ${\tt PDIPE}(\alpha,\theta)$ of Theorem \ref{thm:dP-IP} has generator $2\mathcal{A}_{\alpha,\theta}$.
\end{theorem}

In particular, this establishes the generator of the instances ${\tt PDIPE}(\alpha,0)$ and ${\tt PDIPE}(\alpha,\alpha)$ 
constructed previously in \cite{Paper1-2}.  Since \cite[Theorem 1.4]{RivRiz20} establishes that the operator $\mathcal{A}_{\alpha,\theta}$ as an operator on the linear span of $\{m_\sigma^\circ,\,\sigma\in\mathcal{C}\}$ is closable and generates a conservative Feller diffusion, this indeed identifies the process of \cite{RivRiz20} with ${\tt PDIPE}(\alpha,\theta)$.

\begin{corollary}\label{cor:rivriz} The two-parameter family of processes Rivera-Lopez and Rizzolo constructed in \cite{RivRiz20} is the same up to linear time-change as  
${\tt PDIPE}(\alpha,\theta)$, $\alpha\in(0,1)$, $\theta\ge 0$, as defined in Definition \ref{def:depoiss}.
\end{corollary}

\begin{remark}
Results on infinitesimal generators in the present paper are confined to the statements in this subsection, their application in Section \ref{sec:updown} and their proofs in Section \ref{sec:infgen}. While \cite[Section 4]{Paper1-3} builds on Sections \ref{sec:intro}--\ref{sec:generalized} of the present paper, particularly on Corollary \ref{cor:RANKED} and its proof in Section \ref{sec:Feller}, the generator results here should be seen as a refinement of the partial generator calculations of \cite[Sections 1-3]{Paper1-3} to a core of its domain. 
\end{remark}

\subsection{Limits of up-down chains on the graph of compositions}\label{sec:updown}
The processes \PDIPEAT\ can be viewed as diffusions on the boundary of a weighted branching graph and are the scaling limits of natural up-down Markov chains on these branching graphs. Branching graphs and their boundaries have received substantial focus in the literature. 
Particular attention has been given to the Young graph and the Kingman graph, which are the respective settings of \cite{BoroOlsh09, Fulman2009} and \cite{ Fulman2009,Petrov09}. Both graphs have the set of Young diagrams as the vertex set, with an edge between two diagrams if one can be obtained by removing a box from the other. The two graphs only differ in the weights attached to each edge.  A general framework for studying the types of dynamics on branching graphs considered in \cite{BoroOlsh09, Petrov09} is developed in \cite{Petrov13}, while \cite{Fulman2009} gives a general approach to up-down chains that includes up-down chains on branching graphs.

Compositions give an ordered analogue of Young diagrams. There is a natural branching graph structure on the set of compositions, where there is an edge between two compositions if one can be obtained by adding a box to the other (either creating a new component or increasing the size of an existing component). Connections between the graph of compositions, the Kingman graph, and the Young graph have been explored in \cite{GnOl05}, resulting in the study of the graph of zigzag diagrams.  Although dynamics on the Young graph and Kingman graph have been well explored, the analogous dynamics on the graphs of compositions and zigzag diagrams have not.

The diffusions we construct can be thought of as taking place on the boundary of the branching graph of integer compositions, which was identified by Gnedin \cite{Gnedin97}.  In an earlier version of this paper we conjectured that these diffusions are the scaling limits of up-down chains on the graph of compositions whose up-transitions correspond to seating probabilities in the ordered Chinese Restaurant Process \cite{PitmWink09} and whose down-transitions come from the graph's edge weights; see the discussion in \cite{RogeWink20}. 
This is a natural ordered analogue of the up-down chains used by Petrov in \cite{Petrov09}, where the up-transitions correspond to a ranked Chinese Restaurant Process.  
This conjecture 
was supported by the fact that the diffusions we construct have the correct stationary distributions (see Theorem \ref{thm:dP-IP}) and, as we show in a companion paper \cite{Paper1-3}, the processes of ranked block sizes of Corollary \ref{cor:RANKED} evolve according to \texttt{EKP}$(\alpha,\theta)$ diffusions. 

In fact, this conjecture is made more precise by \cite{RivRiz20}. Inspired by the present paper, \cite{RivRiz20} proves the existence of the scaling limits of these chains on compositions 
and finds the generator for said processes. Corollary \ref{cor:rivriz} here is now the final piece of the jigsaw that proves this conjecture.

\subsection{Construction from marked L\'evy processes} As mentioned above, our toolkit and indeed part of our motivation derives from connections to marked L\'evy processes that go back to \cite{Paper1-1} in the case $\theta=0$. While we postpone details to later sections, we would like to point out some parallels to Ray--Knight theorems that are particularly pertinent when starting from dust on the one hand and in connection with general $\theta>0$ on the other hand. 

Roughly speaking, we need a c\`adl\`ag process $(X_t,\,t\in[0,T])$ with only positive jumps, which we call \emph{scaffolding}. Each jump $\Delta X_t$ is marked by a continuous excursion $f_t$ with excursion length $\zeta(f_t)=\Delta X_t$. We write $N=\sum_{t\in[0,T]\colon\Delta X_t>0}\delta(t,f_t)$ for the point measure of marked jumps. We can interpret each excursion as representing the width of a \emph{spindle} from bottom to top and the spindle is aligned with the scaffolding process, so that it attaches to the jump, as in Figure~\ref{fig:scaff_spind_skewer_1}. 
At each level $y$, the horizontal line intersects certain spindles $f_t$ and the concatenation of their widths $f_t(y-X_{t-})$ at level $y$ gives rise to an interval partition, called the \emph{skewer} at level $y$, denoted by $\skewer(y,N,X)$. With suitably chosen scaffolding and spindles, the skewer process with level $y$ increasing is the desired 
$\SSIPE{\alpha,\theta}$. 

\begin{figure}
 \centering
  \includegraphics[width = 4.8in,height=2.7in]{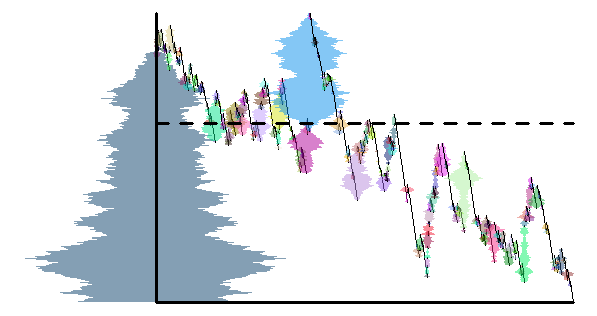}\\
  \hspace{.24in}\includegraphics[width = 2.2in,height=.15in]{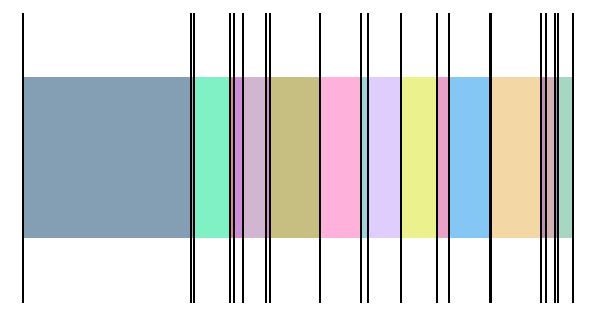}\hspace{.5in}
  \caption{A simulation of the stable L\'evy process with jumps marked by squared Bessel bridges, here depicted by symmetric spindle shapes
     inscribed into the jumps so that interval partitions are obtained by piercing the spindles in the jumps that cross level $y\ge 0$, as if on a skewer. Specifically, this 
     L\'evy process has an initial jump given by the lifetime of a ${\tt BESQ}_a(-2\alpha)$ process, and it is stopped when the L\'evy process returns to level $0$
     so that the interval partition evolution is starting from a single block $\{(0,a)\}$ at level 0.\label{fig:scaff_spind_skewer_1}}
\end{figure}

Also recall the Ray--Knight theorems for Brownian motion that identify, for suitable stopping times, the total local time up to this stopping time, as a process indexed by level, as certain squared Bessel processes. More precisely, the first Ray--Knight theorem considers the first hitting time of $-x$, levels $-x+y$, $y\in[0,x]$, and
finds a ${\tt BESQ}_0(2)$ on that time interval. The second Ray--Knight theorem considers the first time the local time at level 0 exceeds $z$, levels $y\in[0,\infty)$ and finds a ${\tt BESQ}_z(0)$. See e.g.\ \cite[Section XI.2]{RevuzYor}. Generalisations for Brownian motion perturbed at its past or future infimum \cite{LeGallYor1986,CarmPetiYor94} yield general ${\tt BESQ}_0(\delta)$. In our construction we will extract an interval partition for each level from a marked stable L\'evy process $\mathbf{X}$, in which jumps
have been marked by the marking kernel $s\mapsto{\tt BESQ}_{0,0}^s(4+2\alpha)$, i.e.\ every jump of height $\Delta\mathbf{X}_t=s$ is marked by an independent 
squared Bessel bridge, of dimension parameter $4+2\alpha$, of length $s$ from 0 to 0.
This allows us to obtain Ray--Knight theorems analogous to \cite[Theorem 1.8]{Paper1-2}, which covers cases $\theta=0$ and $\theta=\alpha$ without dust.

\begin{theorem}\label{thm:rayknight}
	Let $\bX$ be a spectrally positive stable L\'evy process of index $1+\alpha$, with Laplace exponent $\psi(\lambda) = \frac{\lambda^{1+\alpha}}{2^{\alpha}\Gamma(1+\alpha)}$. 
	Mark each jump $(t,\Delta\bX_t)$ of $\bX$, via the kernel $s\mapsto \BESQ^{s}_{0,0}(4+2\alpha)$, with a excursion $f_t$ and write $\bN=\sum_{t\ge 0\colon \Delta\bX_t > 0}\delta(t,f_t)$. Set $\underline{\bX} = (\underline{\bX}_t = \inf_{r\le t} \bX_r,\, t\ge 0)$.    
	\begin{enumerate}
		\item Let $z\ge 0$. Define
		\[
		\beta^y := \skewer(y,\bN|_{[0,T_{-z/2\alpha}]},\bX- \underline{\bX}),\qquad  y> 0.
		\]
		Set $\beta^0_*=(\varnothing,z)$ and $\beta^y_* = \big(\beta^y, \|\beta^y\|\big)$, $y>0$. Then the process $(\beta^y_*,\, y\ge 0)$ 
		is an ${\tt SSIPE}_{(\varnothing,z)}(\alpha,0)$, i.e.\ an ${\tt SSIPE}(\alpha,0)$ starting from dust of mass $z$, with total mass process $\BESQ_z(0)$. 
		\item Let $x>0$. 
		Set $\bX^{(\theta)}:= \bX-(1-\frac{\alpha}{\theta})\underline{\bX}$ and 
		$T^{(\theta)}_{-x} = \inf\{t\ge 0\colon \bX^{(\theta)}_t < -x\}$. 
		Then the process 
		\[
		\beta^y := \skewer\left(y,\bN|_{[0,T^{(\theta)}_{-x}]}, x +\bX^{(\theta)} \right),\qquad  y\in [0,x]
		\]
		is an ${\tt SSIPE}_\varnothing(\alpha,\theta)$ restricted to the time interval $[0,x]$.
		Its total mass process $\big(\|\beta^y\|,\,y\in [0,x]\big)$ is a $\BESQ_0(2\theta)$ on that time interval. 
	\end{enumerate}
\end{theorem}

\subsection{Diversity property}
The partitions that arise as states in our diffusions possess an interesting property.

\begin{definition}\label{def:diversity_property} 
	For $0\!<\!\alpha \!<\!1$, we say that an interval partition $\beta\!\in\!\HIPspace$ has the \emph{($\alpha$-)diversity property}, or that $\beta$ is an 
	\emph{interval partition with ($\alpha$-)diversity}, if the following limit exists for every $t\in [0,\IPmag{\beta}]$:
	\begin{equation}
		\IPLT^{(\alpha)}_{\beta}(t) := \Gamma(1-\alpha) \lim_{h\downarrow 0}h^\alpha \#\{(a,b)\in \beta\colon\ |b-a|>h,\ b\leq t\}.\label{eq:IPLT}
	\end{equation}
	We denote by $\IPA\subset\HIPspace$ the set of interval partitions with $\alpha$-diversity. 
\end{definition}

\begin{proposition}\label{prop:entrlaw:intro}
	Let $\beta^*\in\IPHs$ and $(\beta^y, M^y)$ an ${\tt SSIPE}_{\beta^*}(\alpha,\theta)$. Then almost surely we have
	$\beta^y\in\IPA$ for all $y>0$. 
	If $\beta^0\in \IPA$, then $(\beta^y, y\ge 0)$ is a path-continuous Hunt process on the space $\IPA$ endowed with a metric $\dIA$ that we specify in Section \ref{sec:IPspace}.
\end{proposition}

Diversity is a continuum analogue to the number of components of a composition \cite{CSP}. The \PDIPAT\ stationary laws of our diffusions are supported on $\IPA$; e.g.\ this follows from the regenerative property noted in \cite{GnedPitm05} and the total diversity property noted in \cite{CSP}. Diversities arise, for example, in spinal projections of continuum random trees (CRTs): we can decompose any continuum tree around a path from the root to a leaf, called a ``spine,'' and represent the mass of each subtree branching off the spine as a block in an interval partition. 
Examples include the Brownian CRT and other stable CRTs \cite{HPW}, as well as  
the two-parameter family of binary trees studied in \cite{PitmWink09}. Diversities in spinal projections then describe distances along the spine in the tree. 
By virtue of this connection the $\big(\frac12,\frac12\big)$- and $\big(\frac12,0\big)$-interval partition evolutions, previously studied in \cite{Paper1-2}, were used in \cite{Paper4} to construct a continuum-tree-valued process stationary with the law of the Brownian CRT. By generalizing \cite{Paper1-2} in the present work, we open the door to subsequent generalizations of the construction of the continuum-tree-valued process in \cite{Paper4}. In the setting of algebraic trees, in which distances along a spine are omitted and only the branching structure is retained, the Brownian tree diffusion of \cite{LohrMytnWint18} was extended in \cite{NussWint20} to the $(\alpha,1-\alpha)$ case.



\subsection{Organization of the paper}

We recall notions and results on related processes in Section \ref{sec:prelim}. In Section \ref{sec:IPAT}, we construct and study ${\tt SSIPE}(\alpha,\theta)$ as a two-parameter family of $\HIPspace$-valued Hunt process and prove Proposition \ref{prop:sp} and Theorem \ref{thm:rayknight}(ii). In Section \ref{sec:generalized}, we construct and study ${\tt SSIPE}(\alpha,\theta)$ as $\IPHs$-valued Feller processes and establish Theorems \ref{thm:sp}, \ref{thm:dP-IP} and \ref{thm:rayknight}(i), as well as Corollary \ref{cor:RANKED} and Proposition \ref{prop:entrlaw:intro}. In Section \ref{sec:infgen}, we turn to the infinitesimal generator and establish Theorem \ref{thm:twopargen}.

\section{Preliminaries}\label{sec:prelim}

The $\mBxcA$ interval partition evolutions of \cite{Paper1-1}, called type-1 evolutions in \cite{Paper1-2} are, in the language of the present work, \PoiIPP{\alpha,0}s, or \SSIPEA. In Section \ref{sec:IPspace} we recall the state space of interval partitions for these evolutions introduced in \cite{Paper1-0}. In Section \ref{sec:skewer} we recall from \cite{Paper1-1,Paper1-2} the construction of \SSIPEA. Finally, in Section \ref{sec:theta_1}, we recall from \cite{FVAT} relevant methods needed to extend to all $\theta\ge0$.

\subsection{The state space: interval partitions with diversity}\label{sec:IPspace}

The state spaces $(\IPA,\dIA)$ and $(\HIPspace,\dHs)$ for our evolutions were introduced in \cite{Paper1-0}. We review key definitions and results from that study here.

We adopt the notation $[n] := \{1,2,\ldots,n\}$.
For $\beta,\gamma\in \IPH$, a \emph{correspondence} from $\beta$ to $\gamma$ is a finite sequence of ordered pairs of intervals $(U_1,V_1),\ldots,(U_n,V_n) \in \beta\times\gamma$, $n\geq 0$, where the sequences $(U_j)_{j\in [n]}$ and $(V_j)_{j\in [n]}$ are each strictly increasing in the left-to-right ordering of the interval partitions.
The $\alpha$-\emph{distortion} of a correspondence $(U_j,V_j)_{j\in [n]}$ from $\beta$ to $\gamma$ in $\IPA$, denoted by $\dis_\alpha(\beta,\gamma,(U_j,V_j)_{j\in [n]})$, is defined to be the maximum of the following four quantities:
 \begin{enumerate}[label=(\roman*), ref=(\roman*)]
  \item $\sum_{j\in [n]}|\Leb(U_j)-\Leb(V_j)| + \IPmag{\beta} - \sum_{j\in [n]}\Leb(U_j)$, \label{item:IP_m:mass_1}
  \item $\sum_{j\in [n]}|\Leb(U_j)-\Leb(V_j)| + \IPmag{\gamma} - \sum_{j\in [n]}\Leb(V_j)$, \label{item:IP_m:mass_2}
  \item $\sup_{j\in [n]}|\IPLT^{(\alpha)}_{\beta}(U_j) - \IPLT^{(\alpha)}_{\gamma}(V_j)|$,
  \item $|\IPLT^{(\alpha)}_{\beta}(\infty) - \IPLT^{(\alpha)}_{\gamma}(\infty)|$.
 \end{enumerate}
 Similarly, the \emph{Hausdorff distortion} of a correspondence $(U_j,V_j)_{j\in [n]}$ between $\beta,\gamma\in\HIPspace$, denoted by $\dis_H(\beta,\gamma,(U_j,V_j)_{j\in[n]})$, is defined to be the maximum of (i)--(ii).

 For $\beta,\gamma\in\HIPspace$ we define
 \begin{equation}\label{eq:IPH:metric_def}
  \dHs(\beta,\gamma) := \inf_{n\ge 0,\,(U_j,V_j)_{j\in [n]}}\dis_H\big(\beta,\gamma,(U_j,V_j)_{j\in [n]}\big),
 \end{equation}
 where the infimum is over all correspondences from $\beta$ to $\gamma$. For $\beta,\gamma\in\IPA$ we similarly define
 \begin{equation}\label{eq:IP:metric_def}
  d_\alpha(\beta,\gamma) := \inf_{n\ge 0,\,(U_j,V_j)_{j\in [n]}}\dis_\alpha\big(\beta,\gamma,(U_j,V_j)_{j\in [n]}\big).
 \end{equation}
The name ``Hausdorff distortion'' refers to a relationship between $\dHs$ and the Hausdorff metric on compact sets \cite[Theorem 2.3]{Paper1-0}. Indeed, we showed that $d_H^\prime$ generates the same topology on $\mathcal{I}_H$ as the metric $d_H\circ C$ given by $C(\beta)=[0,\|\beta\|]\setminus\bigcup_{U\in\beta}U$ and
\[
d_H(K_1,K_2)=\inf\{\varepsilon>0\colon K_1\subseteq K_2^\varepsilon\ \mbox{and}\ K_2\subseteq K_1^\varepsilon\},
\]
for compact $K_1,K_2\subset[0,\infty)$, where $K^\varepsilon=\{x\in[0,\infty)\colon |x-y|\le\varepsilon\mbox{ for some }y\in K\}$ is the $\varepsilon$-thickening of $K$.


\begin{lemma}[{\cite[Theorems~2.3 and 2.4]{Paper1-0}}]\label{prop:IPH}
 $(\IPH,\dHs)$ is a Polish metric space, while $(\IPA,d_\alpha)$ is a Lusin space.
\end{lemma}


An interval partition $\beta\in\HIPspace$ can be reversed and scaled by $c>0$ as 
\begin{equation}\label{eq:IP:xforms}
 \reverse_{\textnormal{IP}}(\beta) := \big\{(\IPmag{\beta}-b,\IPmag{\beta}-a)\colon (a,b)\!\in\!\beta\big\}, \quad 
  c\beta := \big\{(ca,cb)\colon (a,b)\!\in\! \beta\big\}.
\end{equation}

For $\alpha\in(0,1)$ and $\theta\ge 0$, let $W_j\sim{\tt Beta}(\theta+j\alpha,1-\alpha)$, $j\ge 1$, independent, and set $P_n=W_1\cdots W_{n-1}(1 - W_n)$, $n \ge 1$. Then the 
decreasing rearrangement of $(P_n,n \ge 1)$ is known to be $\PoiDir(\alpha,\theta)$-distributed. To construct an interval partition with block lengths $P_n$, $n\ge 1$, start from $\{(0,P_1)\}$
and proceed inductively by inserting an interval of length $P_{n+1}$ to the right of any of the first $n$ blocks with probability $\alpha/(n\alpha+\theta)$ and into left-most position with remaining probability $\theta/(n\alpha+\theta)$. The distribution of the limiting interval partition is called $\PDIP(\alpha,\theta)$.    

The $\reverseI$-reversal of ${\tt PDIP}(\alpha,\theta)$ was studied in \cite{GnedPitm05} and \cite{PitmWink09} under the name ``regenerative $(\alpha,\theta)$-interval partition'' in view of a multiplicative regenerative property and connections to ranges of subordinators and exponential subordinators. The above construction captures the limiting proportions of customers at tables in the ordered $(\alpha,\theta)$-Chinese restaurant process \cite{PitmWink09}.
As a classical example, the excursion intervals of a standard Brownian bridge form a $\PDIP\big(\frac12,\frac12\big)$. If we instead consider excursion intervals in $(0,1)$ of unconditioned Brownian motion, including the incomplete final excursion, then the $\reverseI$-reversal of this set (to move the incomplete excursion interval to the start) is a \PDIP[\frac12,0]. 

Other examples are related to squared Bessel processes. Specifically, for $r\!\in\! \mathbb{R}$, $z\!\ge\! 0$ and $B$ a standard one-dimensional Brownian motion, it is well-known that there exists a unique strong solution to the equation 
\[
Z(t) = z + r t + 2 \int_0^t \sqrt{|Z(s)|} d B(s), 
\]
which is called an \emph{$r$-dimensional squared Bessel process} starting from $z$ and denoted by $\BESQ_{z}(r)$. When $r\le 0$, the boundary point 0 is not an entrance boundary for $(0,\infty)$, while exit at 0 (we will then force absorption) happens almost surely. For $r=d\in\BN$, the squared norm of a $d$-dimensional Brownian motion is a \BESQ[d]. 
The \BESQ[0] process, also known as the Feller diffusion, is a continuous-state branching process that arises as a scaling limit of critical Galton--Watson processes. From this point of view, we can interpret \BESQ[r] with $r>0$ as a Feller diffusion with immigration, and the case $r < 0$ as a Feller diffusion with emigration at rate $|r|$. See \cite{PitmYor82,GoinYor03,Pal13}.
If we replace Brownian motion by \distribfont{BESQ}$(2-2\alpha)$ in the forgoing example, $\alpha\in (0,1)$, then we get $\PDIP(\alpha,\alpha)$ and $\PDIP[\alpha,0]$. See e.g.\ \cite[Examples 3--4 and Sections 8.3--8.4]{GnedPitm05}.

\begin{proposition}\label{prop:PDIP}
 Fix $\alpha\in (0,1)$. Let $S\sim\ExpDist(\lambda)$ and, independently, let $Y$ be a \Stable[\alpha] subordinator with Laplace exponent $q^\alpha$. Set 
  $T := \inf\{s > 0:$ $Y(s) > S\}$ and
  $$\beta := \big\{(Y(t-),Y(t)) \colon t\in [0,T),\, Y(t-)\neq Y(t)\big\}.$$
  Then $Y(T-)$, $S-Y(T-)$, and $(1/Y(T-)) \beta$ are jointly independent, with respective distributions $\GammaDist[\alpha,\lambda]$, $\GammaDist[1\!-\!\alpha,\lambda]$, and $\PDIP[\alpha,\alpha]$. Also,
  \begin{equation}\label{eq:PDIPa0_eg}
   \frac{1}{S} \Big( \big\{(0,S-Y(T-))\big\} \concat \beta \Big) \sim \PDIP[\alpha,0].
  \end{equation}
\end{proposition}

This result is well-known in the folklore around the Poisson--Dirichlet distribution, but for completeness we prove it here.


\begin{proof}
 Let $(Z(t),\,t\ge 0) \sim \distribfont{BESQ}(2-2\alpha)$, as mentioned in the example above the proposition. Then, e.g.\ from \cite[Lemma 3.7]{PermPitmYor92}, the last zero of $Z$ in $[0,1]$ is $G\sim \BetaDist[\alpha,1-\alpha]$, independent of $(Z(uG)/G,u\in[0,1])$, which is a $\distribfont{BESQ}(2-2\alpha)$ bridge. By the scaling invariance of ${\tt BESQ}$, if $S\sim\ExpDist[\lambda]$ is independent of $Z$ then $Z' = SZ(\cdot/S)$ has the same distribution as $Z$ and is also independent of $S$. Let $Y$ denote the level 0 inverse local time process of $Z'$, so $Y$ is a \Stable[\alpha] subordinator independent of $S$ \cite{CSP}, and let $T$ be as in the statement of the proposition. Then
 $$Y(T-) = SG \sim \GammaDist[\alpha,\lambda] \mbox{ and }S - Y(T-) = S(1-G) \sim \GammaDist[1\!-\!\alpha,\lambda]\!,$$
and they are independent, by Beta-Gamma algebra. Moreover, both variables are independent of the bridge mentioned above. The excursion intervals of that bridge comprise the interval partition $\beta$ of the proposition; thus, the claimed \PDIP[\alpha,\alpha] law of $(1/Y(T-))\beta$, as well as the \PDIP[\alpha,0] in \eqref{eq:PDIPa0_eg}, follow from those laws arising in connection with excursion intervals, as noted above the proposition.
\end{proof}

We will also require the following identity.

\begin{lemma}\label{lem:IP:concat_bound}
 For any $\beta_1, \beta_2, \gamma_1, \gamma_2 \in \IPA$, 
 \begin{equation}\label{eq:IP:concat_bound}
	|\dIA(\beta_2, \gamma_2 ) - \dIA(\beta_1, \gamma_1 )| \le \dIA(\beta_1\concat\beta_2, \gamma_1\concat \gamma_2 ) \le \dIA(\beta_1, \gamma_1 ) + \dIA(\beta_2, \gamma_2 ).
 \end{equation}
\end{lemma}
\begin{proof}
 The second inequality is straightforward. We prove the first. Throughout this proof, for $\beta\in\IPA$ and $x\in\BR$ we write $\beta+x$ to denote $\{(a+x,b+x)\colon (a,b)\in \beta\}$.
 
 Consider an arbitrary correspondence $C = (U_j,V_j)_{j\in [n]}$ from $\beta_1\concat\beta_2$ to $\gamma_1\concat\gamma_2$. Let $k,m\in [n]$ denote, respectively, the greatest index for which both $U_k\in \beta_1$ and $V_k\in\gamma_1$, and the least index for which both $U_{m+1}\in \beta_2+\|\beta_1\|$ and $V_{m+1}\in \gamma_2+\|\gamma_1\|$. Then $C_1 = (U_j,V_j)_{j\in [k]}$ and $C_2 = (U_{j+m}\!-\!\|\beta_1\|,V_{j+m}\!-\!\|\gamma_1\|)_{j\in [n-m]}$ are, respectively, correspondences from $\beta_1$ to $\gamma_1$ and from $\beta_2$ to $\gamma_2$.
 
 If $m > k$ then either every pair $(U_j,V_j),\ j\in [k+1,m]$, satisfies $U_j\in\beta_1$ while $V_j\in\gamma_2+\|\gamma_1\|$ or every pair has $U_j\in\beta_2+\|\beta_1\|$ while $V_j\in\gamma_1$ (there cannot be ``mismatched'' pairs of both kinds, as this would violate the ordering property of correspondences). 
 
 We leave it to the reader to confirm that
 \begin{equation*}
 \begin{split}
  \dis_\alpha(\beta_1,\gamma_1,C_1) &\le \dis_\alpha(\beta_1\concat\beta_2,\gamma_1\concat\gamma_2,C) + \dis_\alpha(\beta_2,\gamma_2,C_2)\\
  \text{and}\ \ \dis_\alpha(\beta_2,\gamma_2,C_2) &\le \dis_\alpha(\beta_1\concat\beta_2,\gamma_1\concat\gamma_2,C) + \dis_\alpha(\beta_1,\gamma_1,C_1).
 \end{split}
 \end{equation*}
 
 As we could obtain this bound beginning from an arbitrary correspondence between the concatenated partitions, this proves the lemma.
\end{proof}

\subsection{Scaffolding, spindles, and \SSIPEA}\label{sec:skewer}

Let $\Exc$ denote the set of non-negative, real-valued excursions that are continuous except, possibly, at their birth and death times, when they may have c\`adl\`ag jumps:
\begin{equation}
 \Exc := \left\{f\colon \BR\to [0,\infty)\ \middle| \begin{array}{c}
    \displaystyle \exists\ z\in(0,\infty)\textrm{ s.t.\ }\restrict{f}{(-\infty,0)\cup [z,\infty)} = 0,\ f(0) = f(0+),\\[0.2cm]
    \displaystyle f\text{ positive and continuous on }(0,z),\ f(z-)\text{ exists}
  \end{array}\right\}\!.\label{eq:cts_exc_space_def}
\end{equation}
We define the \emph{lifetime} and \emph{amplitude}, $\life,A\colon \Exc \to (0,\infty)$ via
\begin{equation}
 \life(f) = \sup\{s\geq 0\colon f(s) > 0\} \quad\mbox{and}\quad
 	A(f) = \sup\{f(s),\, s\in [0,\life(f)]\}.
\end{equation}



 For the purpose of the following, for $m>0$ let $H^m\colon \Exc\to [0,\infty]$ denote the first hitting time of level $m$. 
As in \cite[Section 2.3]{Paper1-1}, we write $\mBxcA$ to denote a $\sigma$-finite Pitman--Yor excursion measure \cite{PitmYor82} on $\Exc$ associated with \BESQ[-2\alpha]. In particular, under the probability measure $\mBxcA(\,\cdot \mid A(f) > m)$, the restricted canonical process $(f(y),\,y\in [0,H^m(f)])$  is a $\BESQ_0(4+2\alpha)$ stopped upon first hitting $m$, independent of $(f(H^m(f) + y),\,y\ge0)\sim\BESQ_m(-2\alpha)$. As in \cite{Paper1-1}, we choose to scale this measure so that
\begin{equation}\label{eq:mBxcA}
\begin{split}
 \mBxcA\{f\in\Exc\colon A(f) > m\} &= \frac{2\alpha(\alpha+1)}{\Gamma(1-\alpha)}m^{-1-\alpha} \text{\ \ for all }m>0,\\
 \mBxcA\{f\in\Exc\colon \life(f) > z\} &= \frac{\alpha}{2^{\alpha}\Gamma(1-\alpha)\Gamma(1+\alpha)}z^{-1-\alpha} \text{\ \ for all }z>0.
\end{split}
\end{equation}

\begin{figure}
 \centering
 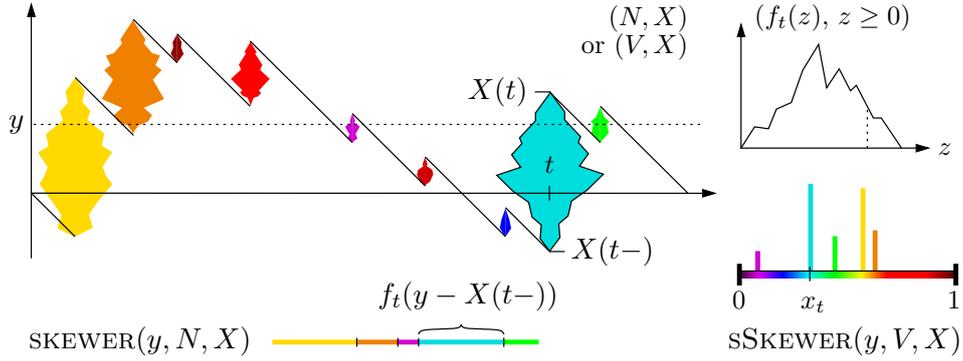
 \caption{A discrete scaffolding and spindles. Left: The slanted black lines comprise the graph of the scaffolding $X$, shaded blobs decorating jumps describe spindles: points $(t,f_t)$ of $N$. The skewer extracts intervals from level $y$, with color only used to illustrate the relation between blocks and spindles. Upper right: Graph of one spindle. Lower right: The (measure-valued) ``superskewer'' studied in \cite{FVAT}; heights of vertical lines represent masses of atoms along a colour spectrum on [0,1]; the superskewer arises from a richer point measure $V$ of spindles with ``type labels,'' $(t,f_t,x_t)$, $x_t\in[0,1]$; here, the colors represent these type labels. Not drawn to scale.\label{fig:scaf_marks}}
\end{figure}

Now, consider a Poisson random measure $\bN$  on $[0,\infty)\times\Exc$ with intensity $\Leb\otimes\mBxcA$, denoted by \PRM[\Leb\otimes\mBxcA]. We wish to pair this with a \emph{scaffolding function} $\bX$. The idea is that each point $(t,f_t)$ of $\bN$ coincides with a jump $\Delta\bX(t) = \bX(t) - \bX(t-) = \life(f_t)$. We then view the \emph{spindle} $f_t$ as describing an evolving mass that is born at level $\bX(t-)$, evolves continuously, and then dies at level $\bX(t)$. The mass of this spindle at level $y$ is then $f_t(y-\bX(t-))$. See Figure \ref{fig:scaf_marks}.

 For $(\cS,d_{\cS})$ a Borel subset of a complete and separable metric space, we denote by $\cN(\cS)$ the space of boundedly finite measures on that space. 
For $N\in\cNRE$, we define the \emph{length} of $N$ to be
\begin{equation}
 \len(N) := \inf\Big\{t>0\colon N\big([t,\infty)\times\Exc\big) = 0\Big\} \in [0,\infty].\label{eq:PP:len_def}
\end{equation}
When the following limit exists for $t\in [0,\len(N)]\cap [0,\infty)$, we further define 
\begin{equation}
 \xiA_N(t) := \lim_{z\downto 0}\left(\int_{[0,t]\times\{g\in\Exc\colon\zeta(g) > z\}}\!\!\life(f)N(du,df) - \frac{(1+\alpha)t}{(2z)^{\alpha}\Gamma(1-\alpha)\Gamma(1+\alpha)}\right).\label{eq:scaff_def}   
\end{equation}
We also write $\xiA(N) := \big( \xiA_N(t),\ t\in [0,\len(N)]\cap [0,\infty) \big)$ and we denote the length of the domain interval for this process by $\len(\xiA(N)) := \len(N)$. We call $\xiA(N)$ the \emph{scaffolding associated with $N$}.

\begin{proposition}[Proposition 2.12 of \cite{Paper1-1}]\label{prop:stable_JCCP}
 For $\bN$ a \PRMLBA\ on $[0,\infty)\times \Exc$, the convergence in \eqref{eq:scaff_def} holds a.s.\ uniformly in $t$ on any bounded interval. 
 Moreover, the scaffolding $\xiA(\bN)$ is a spectrally positive stable L\'evy process of index $1+\alpha$, with L\'evy measure $\Pi(dz)= \mBxcA\{\life \in dz\}$ on $(0,\infty)$ and Laplace exponent $\psi$ given by 
\begin{equation}
\begin{split}
 \Pi(dz) & = \frac{\alpha(\alpha+1)}{2^{\alpha}\Gamma(1-\alpha)\Gamma(1+\alpha)}z^{-\alpha-2}dz \quad 
 \text{and} \quad \psi(\lambda) = \frac{\lambda^{1+\alpha}}{2^{\alpha}\Gamma(1+\alpha)}.
\end{split}\label{eq:scaff:laplace}
\end{equation}
\end{proposition}


\begin{definition}[Point measure of spindles]\label{def:PP_spindles}
 We write $\HA\subset\cNRE$ to denote the set of all counting measures $N$ on $[0,\infty)\times \Exc$ with the following properties:
 \begin{enumerate}[label=(\roman*), ref=(\roman*)]
  \item $N\big( \{t\}\times\Exc \big)\leq 1$ for every $t\in [0,\infty)$,
  \item $N\big( [0,t]\times\{f\in\Exc\colon \life(f) > z\} \big) < \infty$ for every $t,z > 0$, 
  \item the convergence in \eqref{eq:scaff_def} holds uniformly in $t$ on bounded intervals.
 \end{enumerate}
 We define $\HfinA := \{N\in\HA\colon \len(N) <\infty\}$. We call the members of $\HA$ \emph{point measures of spindles}. We denote the $\sigma$-algebras on these spaces generated by evaluation maps by $\SHA$ and $\SHfinA$.
\end{definition}

Let $(N_a)_{a\in\mathcal{A}}$ denote a family of elements of $\HfinA$ indexed by a totally ordered set $(\cA,\preceq)$. For the purpose of this definition, set
$S(a) := \sum_{b\preceq a}\len(N_b)$ and $S(a-) := \sum_{b\prec a}\len(N_b)$ for each $a\in\mathcal{A}$. 
If $S(a-) < \infty$ for every $a\in \mathcal{A}$ and if for every consecutive $a\prec b$ in $\mathcal{A}$ we have $N_a(\{\len(N_a)\}\times\Exc) + N_b(\{0\}\times\Exc)\leq 1$, then we define the \emph{concatenation} to be the counting measure
\begin{equation}
 \Concat_{a\in\mathcal{A}}N_a := \sum_{a\in\mathcal{A}}\int\Dirac{S(a-)+t,f}N_a(dt,df).\label{eqn:concat-N}
\end{equation}

\begin{definition}[Skewer] \label{def:skewer}
 Let $N= \sum_{i\in \mathbb{N}} \Dirac{t_i,f_i} \in \cNRE$ and $X$ a c\`adl\`ag process such that 
 $\sum_{t\ge 0,\, \Delta X(t)> 0} \Dirac{t, \Delta X(t)} = \sum_{i\in \mathbb{N}} \Dirac{t_i, \zeta(f_i)}$. 
  The \emph{skewer} of the pair $(N,X)$ at level $y$ is the interval partition
 \begin{equation}
  \skewer(y,N,X) :=     
  	\{ (M^y(t-),M^y(t)) \colon M^y(t-)<M^y(t),\ t\ge 0 \},
 \end{equation}
 where $M^y(t) = \int_{[0,t]\times\Exc} f\big( y- X(s-) \big) N(ds,df)$. 
 Denote the process by 
 \[
  \skewerP(N,X):= (\skewer(y,N,X),\ y\ge 0). 
 \]
 For simplicity, when $X = \xiA(N)$ we write $\skewerP(N):= \skewerP(N,\xiA(N))$ and 
 $\skewer(y,N):= \skewer(y,N,\xiA(N))$.
\end{definition}

See Figure \ref{fig:scaf_marks} for an illustration of how $\skewer(y,N,X)$ extracts an interval partition from a point measure $N$ of spindles via the level set at level $y$ of $X$. As the level $y\in[0,\infty)$ is varied, this yields the skewer process $\skewerP(N,X)$. For more detail we refer to \cite[Section 2.3--2.4 and 3.1]{Paper1-1}.

We now recall the construction of random point measures used in \cite{Paper1-1}.

\begin{definition}[$\bN_\beta$ and $\Pr^{\alpha,0}_{\beta}$]\label{def:IPPA}
 Let $\beta\in\IPH$. If $\beta = \varnothing$, set $\bN_{\beta} := 0$. Otherwise, we carry out the following construction independently for each $U\in\beta$. Let 
$\bN\sim\PRM[\Leb\otimes\mBxcA]$ and $\bff\sim\BESQ_{\Leb(U)}(-2\alpha)$ independent, and consider the hitting time $T := \inf\{t>0\colon \xiA_{\bN}(t) = -\life(\bff)\}$. Let $\bN_U := \Dirac{0,\bff}+\restrict{\bN}{[0,T]}$. 
 We write $\Pr^{\alpha,0}_{\beta}$ to denote the law of $\bN_{\beta} := \ConcatIL_{U\in\beta}\bN_U$ on $\HfinA$. For probability distributions $\mu$ on $\IPH$, we write $\Pr^{\alpha,0}_{\mu} := \int \Pr^{\alpha,0}_{\beta}\mu(d\beta)$ to denote the $\mu$-mixture of the laws $\Pr^{\alpha,0}_{\beta}$.
\end{definition}

\begin{proposition}[Propositions 5.3 and 6.11 of \cite{Paper1-1}]\label{prop:0:len}
 \begin{enumerate}[label=(\roman*), ref=(\roman*)]
  \item For any $\beta\!\in\!\IPH$, for\linebreak $(\bN_U,\, U \in \beta)$ as in Definition \ref{def:IPPA}, $\sum_{U\in\beta}\len(\bN_U)$ is a.s.\ finite.\label{item:IPPA:finite}
  \item The map $\beta\mapsto\Pr^{\alpha,0}_\beta$ is a stochastic kernel from $\HIPspace$ to $\HfinA$.\label{item:IPPA:kernel}
  \item For $\beta\in\IPA$, there is a version of $\bN_\beta$ for which $\skewerP(\bN_\beta)$ is well-defined and $\dIA$-path-continuous. For $\beta\in\HIPspace$, the same holds except at time 0, where the process is merely $d'_H$-continuous.\label{item:IPPA:cts} 
 \end{enumerate}
\end{proposition}

\begin{proposition}[Theorems 1.2--1.8 of \cite{Paper1-2}]
 The $\dIA$-path-continuous version of $\skewerP(\bN_\beta)$ is the \SSIPEzA{\beta}, i.e.\ the Hunt process in the $\theta=0$ case of Proposition \ref{prop:sp}, with all of the properties claimed in that theorem.
\end{proposition}

\subsection{Scaffolding excursions above the minimum and associated clades}\label{sec:theta_1}

We now recall the developments in \cite{FVAT} that are used to construct Fleming--Viot diffusions, supported on the purely atomic measures, that are stationary with \PoiDirAT\ ranked atom masses, in particular with $\theta>0$.

\begin{definition}\label{def:stable_exc}
 (1) A \emph{clade} is a point measure of spindles $N\in\HfinA$ with the property that $\xiA(N)$ is a non-negative c\`adl\`ag excursion. Recall that heights of the scaffolding $\xiA(N)$ correspond to times in the skewer process $\skewerP(N,\xiA(N))$. In light of this, 
	for a point measure $N\in \HfinA$, we define its \emph{lifetime} by
 \begin{equation}\label{eq:PP:life_def}
  \zeta^+(N) := \sup_{t\ge 0}  \xiA_N(t). 
 \end{equation}
 (2) For a point measure $N\in\HA$ and an interval $[a,b]$, the \emph{shifted restriction} $\shiftrestrict{N}{[a,b]\times\Exc}$ is defined as the restriction of $N$ to $[a,b]\times\Exc$, translated so that it is supported on $[0,b-a]\times\Exc$. The shifted restriction of a function $X\colon [0,\infty)\to \BR$, denoted by $\shiftrestrict{X}{[a,b]}$, is defined correspondingly:
 \begin{equation}
 \begin{split}
  \ShiftRestrict{N}{[a,b]\times\Exc}([c,d]\!\times\! A) &= N\big(([c\!+\!a,d\!+\!a]\!\cap\![a,b])\! \times\! A\big),\ \ c\!\le\! d,\; A\!\in\!\SHA\\
  \ShiftRestrict{X}{[a,b]}(t) &= \cf\{t\in [0,b-a]\}X(t+a),\quad t\in\BR.
 \end{split}
 \end{equation}
\end{definition}

\begin{figure}
 \centering
 \includegraphics[width=13cm,height=6.5cm]{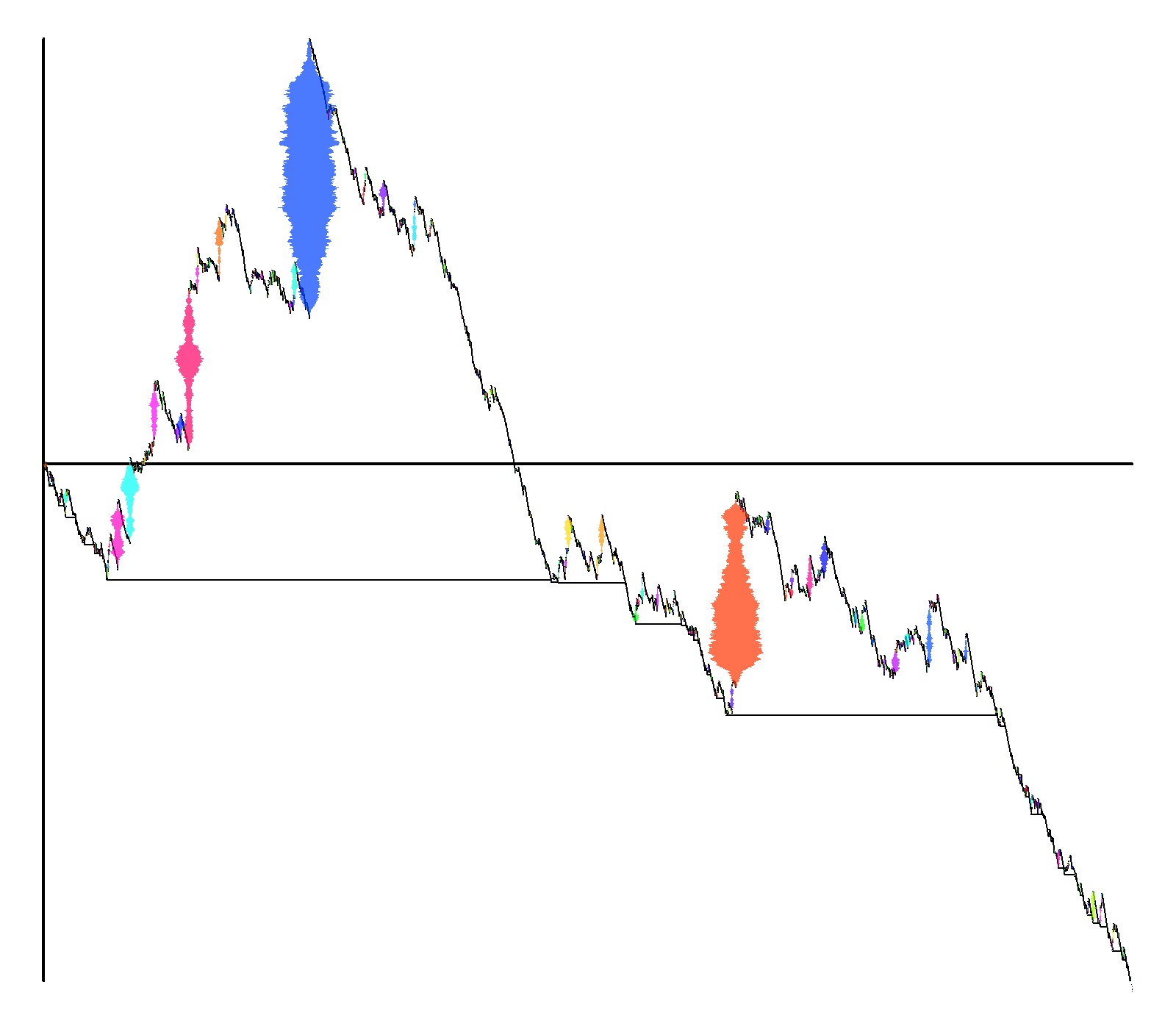}
 \caption{Simulated $\Stable(1.3)$ scaffolding and spindles, superposed with a plot of its infimum process. Its excursions above the infimum are the points of $\uF$ in \eqref{eq:imm_PPP}.\label{fig:mclades}}
\end{figure}

Let $\bN$ denote a \PRMLBA\ on $[0,\infty)\times\Exc$ and $\bX := \xiA(\bN)$. Let
\begin{equation}
 T^{-y} := \inf\left\{t\ge0\colon \bX(t) < -y\right\} \qquad \text{for} \qquad y\ge 0.
\end{equation}
This is a $\Stable(1/(1+\alpha))$ subordinator \cite[Theorem VII.1]{BertoinLevy}, the jumps of which correspond to the intervals of excursions of $\bX$ above its infimum process. We define
\begin{equation}\label{eq:imm_PPP}
 \uF :=\sum_{y\ge 0\colon T^{(-y)-}<T^{-y}}
 \Dirac{y,\ShiftRestrict{\bN}{\left[T^{(-y)-},T^{-y}\right)}}.
\end{equation}
Then this is a point measure of clades, with each point $\Dirac{y,N}$ corresponding to an excursion $\xiA(N)$ of $\bX$ above its infimum, when $\bX$ reaches level $-y$; see Figure \ref{fig:mclades}. 
In \cite{FVAT}, we considered the corresponding point measure of excursions for a point measure of spindles with type labels, $\mathbf{V} \sim \PRM\big(\Leb\otimes\mBxcA\otimes\distribfont{Unif}[0,1]\big)$. We denoted by $\varphi$ the map that projects away these type labels, so that we could have a coupled pair $(\mathbf{V},\mathbf{N})$ with $\bN = \varphi(\mathbf{V})$. In \cite[Proposition 4.3]{FVAT} the analog to $\uF$ in that setting was shown to be a $\PRM(\Leb\otimes\umCladeAbar)$ with an It\^o intensity measure $\umCladeAbar$ on a space of clades with type labels, paired with scaffoldings. The pairs $(V,X)$ arising as points in that measure had the a.s.\ property $X = (\xi\circ\varphi)(V)$. It follows that in the present setting, $\uF$ is a $\PRM(\Leb\otimes\umCladeA)$, the It\^o intensity measure of which, $\umCladeA$, is the pushforward of $\umCladeAbar$ via the map $(V,X)\mapsto\varphi(V)$.

In fact, the point measure $\uF$ can replace the role of $\bN$ in the construction of $\bN_U$ in Definition \ref{def:IPPA}:
 \begin{equation}\label{eq:clade_from_F}
  \bN_U = \delta(0,\mathbf{f}) + \Restrict{\mathbf{N}}{[0,T]} = \delta(0,\mathbf{f}) + \Concat_{\text{points }(y,N_y)\text{ in }\uF\colon y < \life(\bff)} N_y  \sim \mathbf{P}^{\alpha,0}_{\{(0,{\rm Leb}(U))\}},
 \end{equation}
 where the concatenation is in order of increasing $y$.

Along with the point measure $\bN_\beta$ of Definition \ref{def:IPPA}, the point measure $\uF$, up to a constant change of intensity dependent on $\theta$, is the final ingredient needed to construct \SSIPEAT\ with $\theta>0$. Before we proceed to this construction, we recall a few key properties of $\umCladeA$ from \cite{FVAT}.

For $c>0$ and $N\in\HfinA$ we define the scaling operator
\begin{equation}\label{eq:min-cld:xform_def}
 c\scaleHA N = \int \Dirac{c^{1+\alpha}t,c\scaleB f}N(dt,df), \text{ where } c\scaleB f = \big(cf(y/c),\,y\in\BR\big).
\end{equation}
Note the following relations: 
\begin{equation}\label{eq:min-cld:xform_property}
 \zeta^+(c \scaleHA N) = c\zeta^+(N), \quad 
 \len(c \scaleHA N) = c^{1+\alpha}\len(N).
\end{equation}
We may also define $c\scaleHA N$ in the same manner for $N\in\HA$ with $\len(N) = \infty$.

\begin{lemma}[Self-similarity of $\umCladeA$, Lemma 4.4 in \cite{FVAT}]\label{lem:min_cld:scaling}
	For $c>0$, we have
	\begin{equation}
	c\umCladeA(c\scaleHA A) = \umCladeA(A) \qquad \text{for }A\in\SHA.\label{eq:min_cld:scaling}
	\end{equation}
\end{lemma}

\begin{figure}
	\centering
	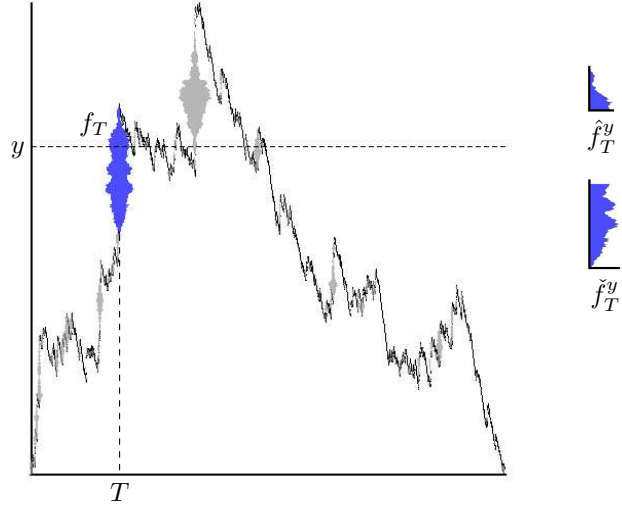
	\caption{Simulation of a clade $N\sim\umCladeA(\,\cdot\,\mid\,\zeta^+>y)$ (with $\alpha = 0.7$) and the broken spindle components, $\big(\check f^y_T,\hat f^y_T\big)$, of the leftmost spindle to cross level $y$, denoted $f_T$ and shaded blue. The leftmost block mass $m^y(N)$ is the width of this blue spindle as it crosses the horizontal dashed line.\label{fig:min_cld_y}}
\end{figure}

Let $N\in \HA$ and $X:= \xiA(N)$. For $y\ge 0$, if an atom $(t,f_t)$ of $N$ satisfies 
$y\in \big( X(t-), X(t)\big)$, i.e.\@ the spindle $f_t$ crosses level $y$, then we define $\hat f^y_{t}$ and $\check f^y_{t}$ to be its broken components split about that crossing. 
Let $m^y = f_t (y - X(t-))$. 
See Figure~\ref{fig:min_cld_y}. Recall the notation $\zeta^+$ and $\len$ from \eqref{eq:PP:life_def}
and \eqref{eq:PP:len_def}. 
\begin{proposition}[Proposition 4.5 and 4.10(i) in \cite{FVAT}]\label{prop:min_cld:stats}\ \ \vspace{-8pt}\\
	\begin{enumerate}[label=(\roman*), ref=(\roman*)] 
		\item $\displaystyle \umCladeA\big\{ \zeta^+ \!>\! z \big\} = \alpha z^{-1}, \quad z>0.$\label{item:MCS:max}
		\item $\displaystyle \umCladeA\big\{\len > x\big\} = \frac{\left(2^{\alpha}\Gamma(1+\alpha)\right)^{1/(1+\alpha)}}
		{\Gamma(\alpha/(1+\alpha)) }
		x^{-1/(1+\alpha)}, \quad x>0$.\label{item:MCS:len}
		\vspace{4pt}
		\item $\umCladeA(m^y \in \cdot \mid \zeta^+ > y) \sim \GammaDist[1-\alpha,1/2y], \quad y>0$.\label{item:MCS:m_y}
	\end{enumerate}
\end{proposition}




\begin{lemma}[Mid-spindle Markov property (MSMP) for $\umCladeA$, Proposition 4.7 in \cite{FVAT}]\label{lem:min_cld:mid_spindle}
 Fix $y>0$. Consider $\bn$ with law $\umCladeA\big(\,\cdot\;|\;\life^+ > y\big)$. Let $(T,f_T)$ denote the first spindle in $\bn$ that crosses level $y$, and let $\hat f^y_{\bn,T}$ and $\check f^y_{\bn,T}$ denote its broken components split about that crossing. 
 Finally, let $m^y(\bn) := \check f^y_T(y-\xiA_{\bn}(T-)) = \hat f^y_T(0)$. Given $m^y(\bn)$, the process $\restrict{\bn}{[0,T)}+\DiracBig{T,\check f^y_{\bn,T}}$ is conditionally independent of $\shiftrestrict{\bn}{(T,\infty)}+\DiracBig{0,\hat f^y_{\bn,T}}$. 
 Moreover, under the conditional law, 
 $(\shiftrestrict{\bn}{(T,\infty)}, \hat f^y_{\bn,T})$ has the law of  
 $(\restrict{\bN'}{[0,\tau]}, f')$, where $f'\sim\BESQ_{m^y(\bn)}(-2\alpha)$, $\bN'\sim\PRMLBA$ independent of $f'$, and $\tau$ is the hitting time of level $-(y+ \zeta(f'))$ by the scaffolding $\xiA(\bN')$. 
\end{lemma}

\section{\CPoiIPP{\alpha,\theta}s for $\theta>0$}\label{sec:IPAT}

For any point measure $\cev{F}$ on $[0,\infty)\times \HfinA$, 
we define an interval-partition-valued process by concatenation, whenever it is well-defined: 
\begin{equation}
 \fskewer\big(y, \cev{F}\big) := \Concat_{\text{points } (s,N_s) \text { of } \cev{F}\colon s\in [y,0]}
 	\skewer(y - s, N_s), \quad y\ge 0.\label{eq:fskewer}
\end{equation}
In this formula we adopt the convention that concatenation over $s\in [y,0]$, with $y>0$, denotes concatenation over $s\in [0,y]$ in reverse order. I.e.\ if for two points $(s_1,N_1)$, $(s_2,N_2)$, we have $s_1 < s_2$, then we concatenate $\skewer(y\! -\! s_2, N_2)$ to the left of $\skewer(y\! -\! s_1, N_1)$. In the population and branching interpretation of the skewer process, these two points represent two sub-populations that enter via immigration at respective times $s_1$ and $s_2$. We follow our convention that any new immigrant clade is placed to the left of the current population.

\begin{definition}[$\cev{\boldsymbol{\beta}}$]\label{defn:cev-beta}
 Fix $\alpha\in (0,1)$, $\theta > 0$. Let $\cev\bF_\theta$ be a $\PRM\big(\frac{\theta}{\alpha}\Leb\otimes\umCladeA\big)$ on $[0,\infty)\times \HfinA$. We define an interval-partition-valued process 
 \begin{equation}\label{eq:cev-beta}
  \cev{\boldsymbol{\beta}} := \big(\cev{\beta}^y, y\ge 0\big) \quad \text{where} \quad \cev{\beta}^y:= \fskewer\big(y, \cev\bF_\theta\big), \quad y\ge0.
 \end{equation}
\end{definition}

The transition kernel of Definition \ref{def:kernel:sp} possesses a branching property. In that branching perspective, $\cev{\boldsymbol{\beta}}$ will serve as an \emph{immigration process}, describing all descendants of individuals that enter via immigration, which occurs, in some sense, with rate proportional to $\theta$. 


\begin{proposition}\label{prop:cts_imm}
 $\big(\cev{\beta}^y,\,y\ge0\big)$ a.s.\ has continuous paths in $(\IPA,\dIA)$.
\end{proposition}

We will prove this in Section \ref{sec:imm_cts_pf}.

To describe a general \SSIPEAT, we must combine the immigration process with the process describing all descendants of the time-zero population; the latter is the \SSIPEA\ of \cite{Paper1-2}.

\begin{definition}[$\BPr^{\alpha,\theta}_{\mu}$]\label{def:IPPAT}
 Fix $\alpha\!\in\! (0,1)$, $\theta\!>\!0$ and a probability distribution $\mu$ on $\IPH$. Let $\bN_{\mu}\!\sim\! \Pr^{\alpha,0}_{\mu}$ as in Definition~\ref{def:IPPA} and, independently, let $\cev\bF_\theta$ and $\cev{\boldsymbol{\beta}}$ be as in Definition \ref{defn:cev-beta}. Let
 \begin{equation}\label{eq:SSIPEAT:def}
  \gamma^y = \skewer\left(y,\bN_{\mu}\right) \quad\text{and}\quad \beta^y := \cev{\beta}^y\concat \gamma^y \quad \text{for }y\ge0.
 \end{equation}
 From Propositions \ref{prop:0:len}(iii) and \ref{prop:cts_imm}, we see that $(\beta^{y},y\ge 0)$ is a.s.\ $\dIA$-path-continuous, except at time 0 if the initial state is in $\IPH\setminus\IPA$, in which case the process is a.s.\ 
 $\dHs$-continuous at time $0$. We denote by $\BPr^{\alpha,\theta}_{\mu}$ the law of $(\beta^{y}, y\ge 0)$ on $\cC([0,\infty),\IPH)$. 
\end{definition}

In Section \ref{sec:main_results} we introduced the term \PoiIPPAT\ (\SSIPEAT) to describe a Hunt process with $\big(\kappa^{\alpha,\theta}_y,\,y\ge0\big)$ as its transition semigroup.

\begin{proposition}\label{prop:identify_SSIPE}
 $\BPr^{\alpha,\theta}_\beta$ is the law of an $\SSIPEzAT{\beta}$, for $\beta\in\IPA$.
\end{proposition}

We prove this, along with Proposition \ref{prop:sp}\ref{item:SSIP:diffusion}, in Section \ref{sec:Hunt}. To do so, we must prove the Hunt property and the claimed transition semigroup for $(\beta^y,\,y\ge0)$ as in \eqref{eq:SSIPEAT:def}.

\subsection{Entrance law for the immigration process}\label{sec:immigration}

\begin{lemma}\label{lem:min_cld:skewer}
 Fix $y>0$ and consider $\bn\sim\umCladeA\big(\,\cdot\;\big|\;\life^+>y\big)$. 
 \begin{enumerate}[label=(\roman*),ref=(\roman*)]
  \item The process $\skewerP (\bn)$ constructed in Definition~\ref{def:skewer} is  a well-defined random variable in $\cC ([0,\infty), (\IPA,  \dIA))$ under the Borel $\sigma$-algebra generated by uniform convergence.  
It is a path-continuous excursion in $\IPA$ starting from $\varnothing$.\label{item:MCSk:cts}
  \item $\skewer(\bn,y) \stackrel{d}{=} B^y \bar\beta$, where $B^y \sim \ExpDist(1/2y)$ independent of $\bar\beta\sim \PDIP(\alpha,0)$.\label{item:MCSk:entrance}
 \end{enumerate}
\end{lemma}
These results have superprocess analogs in \cite[Propositions 4.6, 4.10(ii)]{FVAT}.

\begin{proof}
 We may assume that $\bn$ is the clade corresponding to the first excursion above the minimum of $\bX$ with lifetime $\life^+ > y$. Then we can write $\bn = \ShiftRestrict{\bN}{[T', T'')}$ for a pair of a.s.\@ finite random times $T',T''$. 
 
 
 (i) It is well known that $\bX$ does not jump when it attains a local minimum at time $T'$; see e.g.\ \cite[Proposition 4.5(iii)]{FVAT}. Thus, $\bX(T') = \bX(T'-)$, and 
 \[\begin{split}
  &\skewerP (\restrict{\bN}{[0, T'')},\restrict{\bX}{[0, T'')} - \bX(T'))\\
  &\ \ = \skewerP (\restrict{\bN}{[0, T')},\restrict{\bX}{[0, T')} - \bX(T'))\concat \skewerP (\bn).
 \end{split}\]
 From \cite[Proposition 3.8]{Paper1-1}, the first two skewer processes in this formula are a.s.\ $\dIA$-continuous. 
 The $\dIA$-continuity of $\skewerP(\bn)$ follows, by Lemma \ref{eq:IP:concat_bound}. 
 Finally, by definition of $\bn$ we get $\skewer(0,\bn) = \varnothing$. 
 
 (ii) By the MSMP, Lemma \ref{lem:min_cld:mid_spindle}, given $\restrict{\bn}{[0,T^{\ge y}]}$, the point measure $\shiftrestrict{\bn}{(T^{\ge y},\infty)}$ is conditionally a \PRMLBA\ killed when its scaffolding hits level $-\xi_{\bn}(T^{\ge y})$. Let $\gamma$ denote $\skewer(\bn,y)$ minus its leftmost block. Then by \cite[proof of Proposition 3.4]{Paper1-2}, $G := \IPmag{\gamma}\sim \GammaDist[\alpha,1/2y]$ and, independently, $\bar\gamma := (1/G) \gamma\sim \PDIP[\alpha,\alpha]$. By the Markov property of $\bn$ when $\xi(\bn)$ first hits level $y$, $\gamma$ is independent of $m^y(\bn)$. Finally, our claim follows from Lemma \ref{prop:min_cld:stats}\ref{item:MCS:m_y}, which notes that $m^y(\bn) \sim \GammaDist(1-\alpha,1/2y)$, and the characterization of $\PDIP(\alpha,0)$ in Proposition \ref{prop:PDIP}.
\end{proof}

The following proposition confirms that, for each $y>0$, $\cev{\beta}^y$ is well-defined and lies in $\IPA$ a.s.. This is not obvious, as the diversity property of Definition \ref{def:diversity_property} is not generally preserved over infinite concatenations. 
\begin{proposition}[Cf.\ Proposition 5.1 in \cite{FVAT}]\label{prop:IPPAT:entrance}
 For any $y > 0$, consider $G\sim\GammaDist[\theta,1/2y]$ and an independent interval partition $\gamma\sim\PDIP[\alpha,\theta]$. Then $G \gamma \stackrel{d}{=} \cev\beta^y$, where $\cev\beta^y$ is as in Definition \ref{defn:cev-beta}.
\end{proposition}

\begin{proof}
 Let $\big((S_i,N_i),\,i\ge1\big)$ denote the sequence of points in $\cev\bF_\theta$ in which the clade crosses level $y$, i.e.\ for which $\life^+(N_i) + S_i > y > S_i$. These comprise the points of an inhomogeneous Poisson random measure. Let
 \begin{equation*}
  M_i := \IPmag{\skewer\left(y\!-\!S_i,N_{i}\right)}, \quad  
  \bar\beta_i := \frac{1}{M_i} \skewer\left(y\!-\!S_i,N_{i}\right).
 \end{equation*}
 Then $\cev\beta^y = \ConcatIL_{i=\infty}^1 M_i\bar\beta_i$, where the concatenation is ordered so that the $\bar\beta_{i+1}$ term is attached to the left of the $\bar\beta_i$ term, for each $i$.
 
 From the arguments in \cite[proof of Proposition 5.1]{FVAT}, we get:
 \begin{enumerate}
  \item the sequence $B_i := (y-S_i)/(y - S_{i-1}),\ i\ge1$, is i.i.d.\ $\BetaDist(\theta,1)$, and 
  \item the $(N_i,\,i\ge1)$ are conditionally independent given $(S_j,\,j\ge1)$, with respective conditional distributions $\umCladeA\{\,\cdot \mid \life^+ + S_i > y\}$.
 \end{enumerate}
 By Lemma \ref{lem:min_cld:skewer}\ref{item:MCSk:entrance}, conditionally given $(S_j,\,j\ge1)$, we get 
 \[
  (M_i,\bar\beta_i) \sim \ExpDist(1/2(y-S_i))\otimes\PDIP(\alpha,0) \qquad \text{for }i\ge1.
 \] 
 We get a nicer characterization via the telescoping product of the $\BetaDist$ variables $(B_i)$ and their interaction with the exponential variables $(M_i)$:
 \begin{equation}\label{eq:star}E_i := M_i \Bigg(\prod_{j=1}^i B_j\Bigg)^{\!\!-1} \sim \ExpDist[1/2y].
\end{equation}
 Then, as in \cite{FVAT}, the sequences $(B_i)$, $(E_i)$, and $(\bar\beta_i)$ are each i.i.d.\ and are jointly independent of each other. Now, by a multivariate distributional identity noted in \cite[Lemma 5.3]{FVAT}, which extends the usual Beta-Gamma calculus,
 \begin{equation}\label{eq:cevbetadecomp}
  \cev\beta^y
   = \Concat_{i=\infty}^1 \Bigg(\Bigg(E_i\prod_{j=1}^i B_j\Bigg) \bar\beta_i\Bigg) \stackrel{d}{=} G \Concat_{i=\infty}^1 \Bigg((1-B_i)\Bigg(\prod_{j=1}^{i-1} B_j\Bigg) \bar\beta_i\Bigg),
 \end{equation}
 where $G\sim\GammaDist(\theta,1/2y)$. To conclude, we appeal to a classical decomposition (see e.g.\ \cite[Corollary 8]{PitmWink09}):
 \begin{equation}\label{eq:at_a0}
  \gamma \stackrel{d}{=} (B \bar\gamma) \concat ((1-B) \bar\beta),
 \end{equation}
 where $\bar\gamma\sim\PDIP[\alpha,\theta]$, $\bar\beta\sim\PDIP[\alpha,0]$, and $B\sim\BetaDist[\theta,1]$ independent of each other. Iterating \eqref{eq:at_a0} yields that the right hand side in \eqref{eq:cevbetadecomp} has the claimed distribution.
\end{proof}

\subsection{Proof of Proposition \ref{prop:cts_imm}: path-continuity of the immigration process}\label{sec:imm_cts_pf}
      
  Let $\cev{\bF}_\theta\sim{\tt PRM}\big(\frac{\theta}{\alpha}\restrict{\Leb}{(0,\infty)}\otimes\umCladeA\big)$. 
  Fix $y_0>0$. We first lift from \cite[Section 7.2]{FVAT} a result that controls uniformly in level $[0,y_0]$ the contributions of newly entered clades. To this end, let us introduce some notation. 
For every $z\ge 0$, we define 
    \[\cev{\beta}_{z}^y:= \fskewer\Big(y,\Restrict{\cev{\bF}_{\theta}}{[z,\infty)\times\HfinA}\Big).\]
   That is, only those clades entering above level $z$ count for the process $(\cev{\beta}_{z}^y,y\ge 0)$. In particular, $\cev{\beta}_{z}^y= \varnothing$ for all $y\le z$. 
   For every $z\ge 0$, let 
    \begin{equation*}
      \overline{D}_z^y := \Gamma(1-\alpha)\limsup_{h\downarrow 0} h^{\alpha} \#\{ U\in \cev{\beta}_{z}^y\colon \Leb(U)>h \}, \qquad y\ge 0.
    \end{equation*}
\vspace{-0.3cm}
  \begin{lemma}\label{lem:cts_imm-step2}
     Almost surely, for all $\varepsilon>0$, there exists $\delta^\prime>0$   (that depends on $\varepsilon$ and the realization) such that 
    \begin{equation}\label{eqn:lemma1}
      \sup_{y\in (z, z+\delta^\prime]} \|\cev{\beta}_{z}^y\|<\varepsilon/3,\quad  \sup_{y\in (z, z+\delta^\prime]} \overline{D}_z^y <\varepsilon/3, \quad\text{ for every}~ z\in [0,y_0].
    \end{equation}
  \end{lemma}

The second part of \eqref{eqn:lemma1} has been obtained as Claim 1 in the proof of \cite[Proposition 7.3]{FVAT}. 
That argument was made in the setting of the measure-valued ``superskewer'' map, but it translates without modification to the setting of the skewer map; see Figure \ref{fig:scaf_marks} for an illustration relating these maps. 
The same arguments can be adapted to prove the total mass part. We omit the details. 

Our next step is to prove that almost surely, $\cev{\beta}^y$ has the $\alpha$-diversity property for every $y\ge 0$. \vspace{-0.1cm}
\begin{lemma}\label{lem:cts_imm-step3}
 Almost surely, $\cev{\beta}^y \in \IPA$ for every $y\!\in\! [0,y_0]$. 
\end{lemma}
 
\begin{proof}[Proof of Lemma~\ref{lem:cts_imm-step3}]
By Lemma~\ref{lem:min_cld:skewer}\ref{item:MCSk:cts}, Proposition \ref{prop:min_cld:stats}\ref{item:MCS:max}, and standard properties of Poisson random measures, the following holds almost surely: for all $\delta>0$, there are finitely many points $(s,N_s)$ of $\cev\bF_\theta$ with $s\in[0,y_0)$ and $\zeta^+(N_s)>\delta$; moreover, for each $(s,N_s)$, the process $y\mapsto \skewer(y,N_s)$ is path-continuous in $(\IPA, \dIA)$.
 For the remainder of this proof, we will argue on the intersection of this almost sure event and the one in Lemma~\ref{lem:cts_imm-step2}.

  Fix any $\varepsilon\!>\!0$ and take $\delta^\prime\!>\!0$ so that \eqref{eqn:lemma1} holds. 
 Fix any $y\in [0, y_0]$ and let $y^{\prime}:= \max(y-\delta^{\prime}, 0)$. 
Define the concatenation of long-living clades that enter below level $y^{\prime}$ by
\begin{equation}\label{eq:gamma_y'}
\gamma_{y^{\prime}}^z = \fskewer\Big( z,\ \Restrict{\cev{\bF}_\theta}{[0,y')\times\{N\in\HA\colon \zeta^+(N)>\delta^\prime/4 \}} \Big), \quad z\ge 0,
\end{equation}  
 which takes values in $\IPA$, due to the choice of the almost sure event.
Then we have the decomposition $\cev{\beta}^y =\cev{\beta}_{y^{\prime}}^{y}  \concat \gamma_{y^{\prime}}^y$. 
 Indeed, among the clades that are born below level $y^{\prime}$, only those long-living ones contribute to $\cev{\beta}^y$ at level $y$. 
	
	For every $t\ge 0$, set 
\begin{align*}
\overline{D}_{\cev{\beta}^y}(t)&:= \Gamma(1-\alpha)\limsup_{h\downarrow 0} h^{\alpha} \#\{ U\in \cev{\beta}^y\colon \Leb(U)>h, \sup U \le t \}, \\
\underline{D}_{\cev{\beta}^y}(t)&:= \Gamma(1-\alpha)\liminf_{h\downarrow 0} h^{\alpha} \#\{ U\in \cev{\beta}^y\colon \Leb(U)>h, \sup U \le t \}. 
\end{align*}
Then $\cev{\beta}^y$ has the $\alpha$-diversity property if and only if $\overline{D}_{\cev{\beta}^y}(t)= \underline{D}_{\cev{\beta}^y}(t)$ for every $t\ge 0$. 
 Moreover, we then have $\IPLT_{\cev{\beta}^y}(t)= \overline{D}_{\cev{\beta}^y}(t)= \underline{D}_{\cev{\beta}^y}(t)$.

Write $\Delta:= \|\cev{\beta}_{y^{\prime}}^{y} \|$. By the identity $\cev{\beta}^y =\cev{\beta}_{y^{\prime}}^{y}  \concat \gamma_{y^{\prime}}^y$, we have, for every $t\ge 0$, 
\[
\overline{D}_{\cev{\beta}^y}(t)\le \overline{D}_{y^{\prime}}^y + \mathbf{1}\{t\ge \Delta\}\IPLT_{\gamma_{y^{\prime}}^y}(t-\Delta), \quad 
\underline{D}_{\cev{\beta}^y}(t) \ge \mathbf{1}\{t\ge \Delta\}\IPLT_{\gamma_{y^{\prime}}^y}(t-\Delta). 
\]
  Since $\overline{D}_{y^{\prime}}^y<\varepsilon/3$ by \eqref{eqn:lemma1}, we have  
$    \left|\overline{D}_{\cev{\beta}^y}(t) - \underline{D}_{\cev{\beta}^y}(t)\right|
 <\varepsilon.$ 
  Since $\varepsilon$ was arbitrary, we deduce the existence of
$\IPLT_{\cev{\beta}^y}(t)= \overline{D}_{\cev{\beta}^y}(t)= \underline{D}_{\cev{\beta}^y}(t)$ for all $t\ge 0$. We conclude that $\cev{\beta}^y \in \IPA$ for every $y\!\in\! [0,y_0]$. 
\end{proof}

To complete the proof of Proposition \ref{prop:cts_imm}, it suffices to fix any realization in the almost sure event considered in the proof of Lemma~\ref{lem:cts_imm-step3} and to check the continuity of $y\mapsto \cev{\beta}^y$ at any $x\in [0,y_0]$, under the metric $\dIA$. 
 
  Fix any $\varepsilon\!>\!0$ and take $\delta^\prime\!>\!0$ so that \eqref{eqn:lemma1} holds. 
Let $x^{\prime}:= \max(x - \delta^\prime\!/2, 0)$ and define the process $(\gamma_{x^{\prime}}^y, y\ge 0)$ as in \eqref{eq:gamma_y'}, which evolves continuously in $(\IPA, \dIA)$. 
  In particular, there is $\delta\in(0,\delta^\prime\!/4)$ such that for every $z\!\in\![0,y_0]$ with $|x\!-\!z|\!<\!\delta$, we have $\dIA(\gamma_{x^{\prime}}^x, \gamma_{x^{\prime}}^z )\!<\!\varepsilon/3$. 
Recall \eqref{eq:IP:metric_def}, then there exists a correspondence $(U^x_j,U^z_j)_{j\in [n]}$ from $\gamma_{x^{\prime}}^x$ to $\gamma_{x^{\prime}}^z$, such that the $\alpha$-distortion satisfies 
\[
 \dis_\alpha (\gamma_{x^{\prime}}^x, \gamma_{x^{\prime}}^z, (U^x_j,U^z_j)_{j\in [n]}) \le  \dIA(\gamma_{x^{\prime}}^x, \gamma_{x^{\prime}}^z )  + \varepsilon/3\le 2\varepsilon/3. 
\]
Since $\cev{\beta}^y =\cev{\beta}_{x^{\prime}}^{y}  \concat \gamma_{x^{\prime}}^y$ for all $y\ge x^{\prime}\!+\!\delta^{\prime}\!/4$, which in particular holds for $y=x$, we have 
\[
\max_{i\in [n]}\left|\IPLT_{\cev{\beta}^x} (U^x_i) \!-\! \IPLT_{\gamma_{x^{\prime}}^x} (U^x_i) \right|\!\le\!\overline{D}_{x^{\prime}}^x,\,  
\left|\IPLT_{\cev{\beta}^x} (\infty) \!-\! \IPLT_{\gamma_{x^{\prime}}^x} (\infty) \right|\!\le\!\overline{D}_{x^{\prime}}^x,\, 
\left|\|\cev{\beta}^x\|\!-\! \|\gamma_{x^{\prime}}^x\| \right|\!\le\! \|\cev{\beta}_{x^{\prime}}^x\|, 
\]
and similar bounds with $x$ replaced by $z$. 
Note that $(U^x_j,U^z_j)_{j\in [n]}$ induces a correspondence from $\cev{\beta}^x$ to $\cev{\beta}^z$.
Consequently, we have
\[
\left| \dis_\alpha (\gamma_{x^{\prime}}^x, \gamma_{x^{\prime}}^z, (U^x_j\!,U^z_j)_{j\in [n]}) \!-\! \dis_\alpha ( \cev{\beta}^x\!\!, \cev{\beta}^z\!\!, (U^x_j\!,U^z_j)_{j\in [n]}) \right|\!\le \! \max (\overline{D}_{x^{\prime}}^x ,\overline{D}_{x^{\prime}}^z, \|\cev{\beta}_{x^{\prime}}^x\|,\|\cev{\beta}_{x^{\prime}}^z\|).   
\] 
Since $x,z\in (x^{\prime}, x^{\prime}+\delta^{\prime})$, we have by \eqref{eqn:lemma1} that $\max (\overline{D}_{x^{\prime}}^x ,\overline{D}_{x^{\prime}}^z, \|\cev{\beta}_{x^{\prime}}^x\|,\|\cev{\beta}_{x^{\prime}}^z\|)<\varepsilon/3$. It follows that 
$\dIA(\cev{\beta}^x, \cev{\beta}^z)\!\le\!  \dis_\alpha ( \cev{\beta}^x, \cev{\beta}^z, (U^x_j,U^z_j)_{j\in [n]}) \!\le \!\varepsilon$. 
This implies the continuity at $x$, as required.  \qed

\subsection{Simple Markov property}\label{sec:Markov}

Recall the broken spindle notation illustrated in Figure \ref{fig:min_cld_y}. We define the \emph{lower} and \emph{upper cutoff processes}, depicted in Figure \ref{fig:cutoff}, by
\begin{equation}\label{eq:cutoff}
\begin{split}
 \cutoffLB{y}{N} &:= \sum_{\text{points }(t,f_t)\text{ of }N} \left(\begin{array}{r@{\,}l}
	\cf\big\{y\in \big( X(t-), X(t)\big)\big\}&\DiracBig{\Theta(t),\check f^y_t}\\[3pt]
	+\; \cf\big\{X(t) \leq y\big\}&\DiracBig{\Theta(t),f_t}
   \end{array}\right),\\
 \cutoffGB{y}{N} &:= \sum_{\text{points }(t,f_t)\text{ of }N} \!\left(\!\!\begin{array}{r@{\,}l}
 	 \cf\big\{y\in \big( X(t-), X(t)\big)\big\}&\DiracBig{t\!-\!\Theta(t),\hat f^y_t}\\[3pt]
	 +\; \cf\big\{X(t-) \geq y\big\}&\DiracBig{t\!-\!\Theta(t),f_t}
	\end{array}\!\!\right)\!,
\end{split}
\end{equation}
where $\Theta(t) := \Leb\{u\le t\colon X(u)\le y\}$ for $t\ge 0$.

\begin{figure}
 \centering
 \input{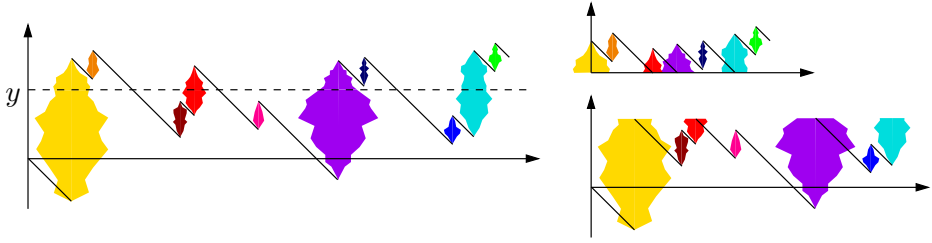}
 \caption{Decomposition of a point measure of spindles into upper and lower cutoff processes, as in \eqref{eq:cutoff}.\label{fig:cutoff}}
\end{figure}

\begin{corollary}[Markov-like property of $\umCladeA$]\label{cor:min_cld:mid_spindle}
	Let $y\!>\!0$ and $\bn\sim\umCladeA\big(\,\cdot\,\big|\,\life^+\!>\!y\big)$.\linebreak
	Then given $\cutoffLB{y}{\bn}$, the point measure $\cutoffGB{y}{\bn}$ has conditional distribution 
	$\Pr^{\alpha,0}_{\beta^y}$ with $\beta^y = \skewer(y,\bn)$. 
\end{corollary}

This is akin to a Markov property, with $\cutoffLB{y}{\bn}$, $\beta^y$, and 
$\cutoffGB{y}{\bn}$ playing the respective parts of past, present, and future. A closely related statement is proved in \cite[Corollary 4.8]{FVAT}
for the construction of measure-valued processes. While the proof is identical, we repeat it here because it is short and the upper cutoff process of \eqref{eq:cutoff} captures important additional order structure compared to the upper point measures of \cite{FVAT}. 

\begin{proof}
 Using the MSMP Lemma~\ref{lem:min_cld:mid_spindle} and the notation in its statement, the conditional distribution that we wish to characterize is the same as the distribution of $\cutoffGB{0}{N'}$ given $\beta'=\skewer(0,N')$, where $N'= \Dirac{0,f'}+ \restrict{\bN'}{[0,\tau]}$. 
 By the Markov-like property of $N'$, noted in \cite[Proposition~5.9]{Paper1-1}, the latter conditional law is the same as $\Pr^{\alpha,0}_{\beta'}$. 
 This proves the corollary.
\end{proof}

To describe the Markov property we require additional notation. Recall \eqref{eq:cutoff}. 
For any point measure $\cev{F}$ on $[0,\infty)\times \mathcal{N}_{\rm fin}^{\rm sp}$ we define
\begin{equation}\label{eq:cevF:N0}
 N^{\ge y}_0 (\cev{F}) := \Concat_{\text{points }(s,N_s)\text{ of }\cev{F}\colon s\in [y,0]}  \cutoffGB{y-s}{N_s},
\end{equation}
provided that this concatenation is well-defined, with the same convention as in \eqref{eq:fskewer} that we write $[y,0]$ as the time interval over which we concatenate, despite $y$ being positive, to indicate that we concatenate in reverse order, setting cutoff point measures with $s$ high to the left of those with $s$ low. Extending the notion of cutoff processes to point measures of clades, we set 
\begin{equation}\label{eq:cevF:cutoffL}
 \cutoffLB{y}{\cev{F}}:= \sum_{\text{points }(s,N_s)\text{ of }\cev{F}\colon s\in [y,0]} \delta\big(s, \cutoffLB{y-s}{N_s}\big).
\end{equation}

The following result is analogous to \cite[Lemma 5.7]{FVAT}. 

\begin{lemma}\label{lem:cevF:cutoff}
 Fix $y>0$ and let $\cev{\bF}_\theta$ be a $\PRM[\frac{\theta}{\alpha}\Leb\otimes\umCladeA]$ on $(0,\infty)\times \HfinA$. Then given $\cutoffLB{y}{\cev{\bF}_\theta}$, we have the following conditional distribution: 
 \begin{equation}\label{eq:lem:cevF:cutoff}
  \left( \Restrict{\cev{\bF}_\theta}{(y,\infty)\times \mathcal{N}_{\rm fin}^{\rm sp}}
  , N^{\ge y}_0\big(\cev{\bF}_\theta\big) \right) 
  \sim \PRM\Big(\frac{\theta}{\alpha}\Restrict{\Leb}{(y,\infty)}\otimes\umCladeA\Big)
  \!\otimes \Pr^{\alpha,0}_{\beta^y},
 \end{equation}
 where $\beta^y= \fskewer(y, \cev{\bF}_\theta)$. 
\end{lemma}
\begin{proof} The proof of \cite[Lemma 5.7]{FVAT} is easily adapted, with no adjustments needed for the conditional independence and first marginal of
  \eqref{eq:lem:cevF:cutoff}. For the second marginal, 
  we recall from the proof of Proposition \ref{prop:IPPAT:entrance} the notation $(S_i,N_i)$, denoting the $i^{\text{th}}$ point in $\restrict{\cev{\bF}_\theta}{[0,y)\times\HA}$ whose scaffolding $X_i=\xi(N_i)$ exceeds level $y-S_i$. In this notation,
 $N^{\ge y}_0\big(\cev{\bF}_\theta\big) = \ConcatIL_{i=\infty}^1 \cutoffGB{y\! -\! S_i}{N_i}$. 
  As noted in that proof, the $(N_i,\,i\ge1)$ are conditionally independent given the $(S_i,\,i\ge1)$, with respective conditional laws  $N_i \sim \umCladeA(\,\cdot\,\mid\,\zeta^+ \!>\! y\!-\!S_i)$, $i\ge1$. 
  Corollary \ref{cor:min_cld:mid_spindle} then implies that, given $\cutoffLB{y}{\cev{\bF}_\theta}$, the $i^{\rm th}$ term in this concatenation has conditional law 
  $\Pr^{\alpha,0}_{\beta_i^y}$, where $\beta_i^y=\skewer(y-S_i,N_i)$. Finally, $N_0^{\ge y}(\cev{\bF})$   
  concatenates upper cutoff processes in the same order in which $\fskewer(y,\cev{\bF})$ concatenates the contributions $\beta_i^y$ to the level-$y$ interval partition, $\beta^y$. It follows straight from Definition \ref{def:IPPA}, applied to $\beta=\beta^y$, that the conditional distribution of 
  $N_0^{\ge y}(\cev{\mathbf{F}}_\theta)$ given $\cutoffLB{y}{\cev{\bF}_\theta}$ is indeed $\mathbf{P}_{\beta^y}^{\alpha,0}$.
\end{proof} 

The following version of the Markov property was already obtained in \cite[Proposition 6.13]{Paper1-1} in the case $\theta=0$ and in \cite[Proposition 3.11]{Paper1-2} in the case $\theta=\alpha$. We can now establish the general case.

\begin{proposition}[Simple Markov property]\label{prop:MP}
 For $\mu$ a probability distribution on $\IPH$, let $(\beta^y,y\ge 0)\sim \BPr^{\alpha,\theta}_{\mu}$. Then for any $y>0$, given $(\beta^r,r\le y)$, the process $(\beta^{z+y},z\ge 0)$ has conditional distribution $\BPr^{\alpha,\theta}_{\beta^y}$. 
\end{proposition}
\begin{proof}
 We follow the notation of equation \eqref{eq:SSIPEAT:def}. We extend the notation from Definition \ref{def:stable_exc} for the shifted restriction of a point measure of spindles, $\shiftrestrict{N}{[a,b]\times\Exc}$, to apply to point measures of clades. 
Using \cite[Lemma 5.4(i)]{Paper1-1}, which asserts that 
 $$\skewer(z,N) = \skewer(z\!-\!y,\cutoffG{y}{N}),$$
 we get, for $z\ge0$, 
 \[
 \begin{split}
  \beta^{z+y} &= \fskewer \Big(z, \ShiftRestrict{\cev{\bF}_\theta}{(y,\infty)\times \HA}\Big) \concat \skewer \Big(z, N^{\ge y}_0(\cev{\bF}_\theta) \concat \cutoffG{y}{\bN_{\mu}}\Big).
 \end{split}
 \]
 The process $(\beta^r,r\le y)$ is a function of $\cutoffLB{y}{\cev{\bF}_\theta}$ and $\cutoffL{y}{\bN_\mu}$. Thus, it suffices to prove that, given these two cutoff point measures, $(\beta^{z+y},\,z\ge0)$ has conditional law $\BPr^{\alpha,\theta}_{\beta^y}$.
 
 Lemma \ref{lem:cevF:cutoff} asserts that under this conditioning, $\ShiftRestrict{\cev{\bF}}{(y,\infty)\times \HA}$ and $N^{\ge y}_0(\cev{\bF}_\theta)$ are conditionally independent, with conditional laws $\PRM\big(\frac\theta\alpha\Leb\otimes\umCladeA\big)$ and $\Pr^{\alpha,0}_{\cev\beta^y}$\linebreak respectively. Analogously, \cite[Proposition 6.6]{Paper1-1} implies that $\cutoffG{y}{\bN_\mu}$ has conditional law $\Pr^{\alpha,0}_{\gamma^y}$, and by construction, this is conditionally independent of the other two point measures. Finally, it follows from Definition \ref{def:IPPA} of these laws that
 $$N^{\ge y}_0(\cev{\bF}_\theta) \concat \cutoffG{y}{\bN_{\mu}}\text{\ \ has conditional law\ \ }\Pr^{\alpha,0}_{\cev\beta^y\concat\gamma^y} = \Pr^{\alpha,0}_{\beta^y}.$$
 We conclude by Definition \ref{def:IPPAT} of $\BPr^{\alpha,\theta}_{\beta^y}$.
\end{proof}

\subsection{Semigroup, Hunt property, and 1-self-similarity}\label{sec:Hunt}

Having prepared the required ingredients in the preceding sections, we can now prove parts of Proposition \ref{prop:sp}. The proofs in this section largely mirror the corresponding proofs in \cite[Section 5]{FVAT}, with the differences being predominantly carried in the aforementioned preparation.

\begin{proposition}[Transition semigroup under $\BPr^{\alpha,\theta}_{\mu}$]\label{prop:IP:transn}
 The kernels $\kappa^{\alpha,\theta}_y$ of Definition \ref{def:kernel:sp} comprise the transition semigroup of a process with law $\BPr^{\alpha,\theta}_{\mu}$ for any law $\mu$ on $\IPH$.
\end{proposition}
\begin{proof}
 Fix $\gamma\in\IPH$. 
 As we have already shown in Proposition \ref{prop:MP} that $(\beta^y,\,y\ge0)\sim \BPr^{\alpha,\theta}_{\gamma}$ possesses the Markov property, it suffices to show that $\beta^y \sim \kappa^{\alpha,\theta}_y(\gamma,\,\cdot\,)$ for any $y$. 
 Recall from Definition \ref{def:kernel:sp} that a $\kappa^{\alpha,\theta}_y(\gamma,\,\cdot\,)$-distributed interval partition comprises a \PDIPAT, $\bar\beta$, scaled by a Gamma-distributed factor $G^y$, concatenated with partitions $\beta_U^y$ corresponding to the blocks $U\in\gamma$. We therefore complete the proof with reference to \cite[Corollary 3.7]{Paper1-1} and Proposition~\ref{prop:IPPAT:entrance}. The former states the $\theta=0$ case of the present result, accounting for all of the $(\beta_U^y,\,U\in\gamma)$; the latter describes the blocks arising from immigration with rate $\theta>0$, accounting for $G^y \bar\beta$, as required.
\end{proof}

\begin{proposition}[Continuity in the initial condition]\label{prop:IP:initial}
 For a convergent sequence $\beta_n \to \beta$ in $(\IPH,\dHs)$, we get weak convergence 
 $\BPr^{\alpha, \theta}_{\beta_n} \to \BPr^{\alpha, \theta}_{\beta}$ 
 in the sense of finite-dimensional distributions on $(\IPA,\dIA)$ at positive times.
\end{proposition}
\begin{proof}
 In the case $\theta= 0$, this has been proved by 
 \cite[Corollary 6.16]{Paper1-1} in the space $(\IPA, \dIA)$. 
 The same conclusion holds for $(\IPH, \dHs)$ as well; 
 see \cite[Proposition 6.20]{Paper1-1} and the remarks below it. 
 By these results, there exists a probability space supporting a sequence of processes 
 $(\beta^y_n,\,y\ge0)$, $n\ge0$, that converges in probability, with respective distributions $\BPr^{\alpha,0}_{\beta_n}$. Now, on (a suitable enlargement of) the same probability space, consider and independent process $\big(\cev\beta^y,\,y\ge0\big) \sim \BPr^{\alpha,\theta}_\varnothing$. Then for each $n$, by Definition \ref{def:IPPAT}, $\big(\cev\beta^y\concat\beta^y_n,\,y\ge0\big) \sim \BPr^{\alpha,\theta}_{\beta_n}$, and these processes converge in probability in the sense of finite-dimensional distributions.
\end{proof}


\begin{proposition}[Strong Markov property]\label{prop:IP:SMP}
 Let $(\beta^y,y\ge 0)\sim \BPr^{\alpha,\theta}_{\mu}$ for some initial law $\mu$ on $\IPH$. 
 Denote by $(\cFI^y, y\ge 0)$ the right-continuous natural filtration of this process and consider an a.s.\@ finite $(\cFI^y ,y\ge 0)$-stopping time, $Y$. 
Then given $\cFI^Y$, the process $(\beta^{Y+y}, y\ge 0)$ has conditional distribution 
$\BPr^{\alpha,\theta}_{\beta^Y}$.
\end{proposition}
\begin{proof}
 This follows from standard arguments: first assume the stopping time takes a finite number of possible values, then pass to the general case via approximation, using continuity in the initial condition, Proposition~\ref{prop:IP:initial}. See the proof of \cite[Theorem 19.17]{Kallenberg} for such an argument written in detail. 
\end{proof}

We now prove the existence of a Hunt process with the transition kernel introduced in Definition \ref{def:kernel:sp} by showing that $\BPr^{\alpha,\theta}_{\mu}$ is the law of such a process, thereby proving Proposition \ref{prop:sp}\ref{item:SSIP:diffusion} and Proposition \ref{prop:identify_SSIPE}.

\begin{proof}[Proof of Proposition \ref{prop:sp}\ref{item:SSIP:diffusion} and Proposition \ref{prop:identify_SSIPE}]
 The process $(\beta^y,\,y\ge0)\sim\BPr^{\alpha,\theta}_\beta$ was constructed by concatenating path-continuous parts (Propositions \ref{prop:0:len}\ref{item:IPPA:cts} and \ref{prop:cts_imm}), so we know that it is continuous, and we have shown that it has the strong Markov property (Proposition \ref{prop:IP:SMP}) with the claimed semigroup (Proposition \ref{prop:IP:transn}). These properties, along with continuity in the initial condition and the Lusin property of the state space $(\IPA, \dIA)$, noted in Propositions \ref{prop:IP:initial} and \ref{prop:IPH} respectively, satisfy Sharpe's definition of a Hunt process; see e.g.\@ \cite[Definition A.18]{Li11}. 
\end{proof}

Now, we establish the scaling invariance and total mass process of \SSIPEAT.

\begin{proof}[Proof of Proposition \ref{prop:sp}\ref{item:SSIP:SS}, scaling invariance for \SSIPEAT]
 We follow the notation of \eqref{eq:SSIPEAT:def}. Since the statement in the case $\theta=0$, which applies to $(\gamma^y,\,y\ge0)$, has been obtained as a part of \cite[Theorem~1.2]{Paper1-1}, it suffices to prove scaling invariance for $\big(\cev{\beta}^y,\,y\ge0\big)$. 
 
 Recall the definition of scaling of clades in \eqref{eq:min-cld:xform_def}. Let $\cev{\bF}_\theta^{(c)}$ denote the point measure obtained by replacing each atom $(s,N_s)$ of $\cev{\bF}_\theta$ with $(cs,c\scaleH N_s)$. Note that 
 $\skewer (y, c\scaleH N_s) = c  \skewer(y/c, N_s)$ for every $y>0$. 
 Using this fact and a change of variable,  we have the identity
 \begin{equation*}
 \begin{split}
  \cev{\beta}^{y}_c :=&\ \fskewer(y,\cev{\bF}_\theta^{(c)}) = \Concat_{\text{points } (s,N_s) \text { of } \cev{\bF}_\theta^{(c)}\colon s\in [y,0]}
  \skewer (y- s, N_s )\\
  =& \Concat_{\text{points } (r,N_r) \text { of } \cev{\bF}_\theta\colon r\in [y/c,0]}
  	c \,\skewer (y/c- r, N_r )
  = c \cev{\beta}^{y/c}, \qquad\quad y\ge 0. 
 \end{split}
 \end{equation*}
 On the other hand, it follows from \eqref{eq:min_cld:scaling} that $\cev{\bF}_\theta^{(c)}\stackrel{d}{=}\cev{\bF}_\theta$, so we have
 \[
  (\cev{\beta}^y,\,y\ge 0)\stackrel{d}{=}(\cev{\beta}^{y}_c,\,y\ge 0) = (c \cev{\beta}^{y/c},\,y\ge 0), 
 \]
 as desired. 
\end{proof}

\begin{proof}[Proof of Proposition \ref{prop:sp}\ref{item:SSIP:mass}, total mass process for \SSIPEAT]
 Again following the notation of \eqref{eq:SSIPEAT:def}, let $(A^y,\,y\ge0)\sim \BESQ_{\|\beta^0\|}(0)$ given $\beta^0$, and independently let $(Z^y,\,y\ge 0)\sim \BESQ_0(2 \theta)$. We know from \cite[Equation (50)]{GoinYor03} and Proposition \ref{prop:IPPAT:entrance} that 
 $Z^y \sim \GammaDist(\theta,1/2y) \sim \big\|\cev{\beta}^y\big\|$. From \cite[Theorem 1.4]{Paper1-2}, which is the $\theta=0$ special case of Proposition \ref{prop:sp}\ref{item:SSIP:mass}, we know $\|\gamma^y\| \stackrel{d}{=}A^y$.  Thus,
 $$\|\beta^y\| = \big\|\cev\beta^y\big\| + \|\gamma^y\| \stackrel{d}{=} Z^y+A^y.$$ 
 By the additivity of squared Bessel processes \cite[Theorem 7]{GoinYor03}, this has the law of the level $y$ evaluation of a $\BESQ_{\|\beta^0\|}(2\theta)$.
 
 To pass from one-dimensional to finite-dimensional distributions, we appeal to the simple Markov property. In particular, $\cev\beta^{y+z}$ can be decomposed into an \SSIPEzAT{\varnothing} (analogous to $\cev\beta$) comprising blocks entering via immigration between levels $y$ and $y+z$, and, given $\beta^y$, a conditionally independent \SSIPEzA{\beta^y} (analogous to $\gamma$). The same \BESQ\ additivity argument as above completes the argument.
\end{proof}

\subsection{Pseudo-stationarity}\label{sec:IPAT:pseudostat}

In this section, we will prove Proposition \ref{prop:sp}\ref{item:SSIP:pseudostat}, which asserts that, if the initial state of a \SSIPEAT, $\beta^0$, is a \PDIPAT\ scaled by an independent random mass, then so is its state at every subsequent time. As in the proof of the superprocess analog to this result, \cite[Theorem 6.1]{FVAT}, we begin with a special case.

\begin{proposition}\label{prop:pseudostat:gamma}
 Suppose that in the setting of Proposition \ref{prop:sp}\ref{item:SSIP:pseudostat} with $\theta \!>\! 0$, we have $Z(0) \!\sim\! \GammaDist[\theta,\rho]$ for some $\rho \!\in\! (0,\infty)$. Then $\beta^y \stackrel{d}{=} (2y\rho\!+\!1)Z(0) \bar\beta$.
\end{proposition}

This proposition is proved in the same manner as its analog in the superprocess setting, \cite[Proposition 6.2]{FVAT}. For completeness, we include the short proof here.

\begin{proof}
 Let $\boldsymbol{\gamma} = (\gamma^y,\,y\ge0)$ be a \SSIPEzAT{\varnothing}. 
 By Proposition \ref{prop:IPPAT:entrance}, $\gamma^{1/2\rho}$ is distributed like $\beta^0$, 
 the initial state of the process $\boldsymbol{\beta}$. 
 Then it follows from the simple Markov property that $\widetilde{\gamma}^y:=\gamma^{y+(1/2\rho)}$ has the same law as $\beta^y$ for $y\ge0$. 
 Applying Proposition~\ref{prop:IPPAT:entrance} to $\boldsymbol{\gamma}$, 
 $\beta^y \stackrel{d}{=} \gamma^{y+(1/2\rho)} \stackrel{d}{=} G^y  \bar\beta$, 
 where $G^y\sim\GammaDist[\theta,1/2(y+(1/2\rho))]$, independent of $\bar\beta$. 
 Since $G^y$ has the same the distribution as $(2y\rho+1)Z(0)$, the desired statement follows.
\end{proof}

\begin{proof}[Proof of Proposition~\ref{prop:sp}\ref{item:SSIP:pseudostat}]
 The cases $\theta=0$ and $\theta =\alpha$ have been proved in \cite[Thoerem~1.5]{Paper1-2}. 
 When $\theta>0$, it is known \cite[Equation~49]{GoinYor03} that the transition density of a $\BESQ(2\theta)$ is 
 \[
  q_y (b,c):= \frac{1}{2y} \left(\frac{c}{b}\right)^{(\theta-1)/2} 
  \exp\left(-\frac{b+c}{2y}\right) I_{\theta-1} \! \left(\frac{\sqrt{bc}}{y}\right), \quad y,b,c>0,
 \]
 where $I_{\theta-1}$ is the modified Bessel function of the first kind of index $\theta-1$. 
 Since $(q_y(b,c),\, c> 0)$ is a probability density, substituting $u=bc/y^{2}$ and $x = y/2b$, we get the identity  
 \begin{equation}\label{eq:besselfuction}
  \int_{0}^{\infty}  u^{(\theta -1)/2} e^{-x u} I_{\theta -1} (\sqrt{u}) d u 
  	= e^{1/(4x)} x^{-\theta} 2^{-(\theta-1)}, \quad \forall x>0. 
 \end{equation}
 
 For every $b>0$, let $(\beta_b^y, y\ge 0)$ be a \SSIPEzAT{b\bar\beta}. 
 For all bounded continuous  
 $f\colon \IPH \to \mathbb{R}_+$ with $f(\varnothing)=0$, 
 Proposition~\ref{prop:pseudostat:gamma} leads to the identity 
 \begin{equation*}
 \begin{split}
  &\int_{0}^{\infty} \! \frac{1}{\Gamma(\theta)} \rho^{\theta} b^{\theta-1} e^{-\rho b}
  		\mathbf{E}[ f(\beta_b^y) ] db 
  	= \int_0^{\infty} \! \frac{1}{\Gamma(\theta)} \rho^{\theta} b^{\theta-1} e^{-\rho b} 
  		\mathbf{E}[ f( (2y\rho \!+\! 1)b   \bar\beta) ] db \\
  &\qquad\qquad\qquad 
  = \int_0^{\infty} \frac{1}{\Gamma(\theta)} \frac{\rho^{\theta}}{(2 y \rho +1)^{\theta}} c^{\theta-1} 
  		\exp\left(-\frac{\rho}{2y\rho+1} c\right) \mathbf{E}[f (c  \bar\beta)] dc.
 \end{split}
 \end{equation*}
 From \eqref{eq:besselfuction}, substituting $u = bc/y^2$ and $x=y(2y\rho+1)/2c$, we get
 \[
 \begin{split}
  \int_0^\infty \!\!q_y(b,c)e^{-\rho b}b^{\theta-1}db
  	&= \frac{1}{2y}e^{-c/2y}\int_0^{\infty}\!\!\exp\!\left(\!-\frac{2y\rho \!+\! 1}{2y} b\right)\! (bc)^{(\theta-1)/2} I_{\theta-1}\!\left(\!\frac{\sqrt{bc}}{y}\right)\! db\\
  	&= \exp\left( -\frac{c\rho}{2y\rho+1}\right) \frac{c^{\theta-1}}{(2yp+1)^{\theta}}.
 \end{split}
 \]
 Then, by Fubini's Theorem,
 \begin{multline*}
  \int_0^{\infty} \frac{1}{\Gamma(\theta)} \frac{\rho^{\theta}}{(2 y \rho +1)^{\theta}} c^{\theta-1}  \exp\left(-\frac{c\rho}{2y\rho+1}\right) 
  		\mathbf{E}[f (c  \bar \beta)] d c \\
  	= \int_{0}^{\infty}  \frac{1}{\Gamma(\theta)} \rho^{\theta} b^{\theta-1} e^{-\rho b} 
  		\left( \int_0^{\infty} q_y (b,c) \mathbf{E}[f (c  \bar \beta)] d c\right) d b. 
 \end{multline*}
 Uniqueness of the Laplace transform implies that for Lebesgue-a.e.\@ $b>0$, 
 \[
  \mathbf{E}[f (\beta_b^y)] = \int_0^{\infty} q_y (b,c) \mathbf{E}[f (c  \bar\beta)] d c.  
 \]
 This extends to all $b>0$ by continuity.  
\end{proof}

\subsection{Ray--Knight theorem for perturbed marked stable L\'evy processes}\label{sec:rk1}

In this section we prove the first Ray--Knight theorem stated in Theorem \ref{thm:rayknight}(ii). Recall that we defined 
$\cev\bF_\theta$ as a $\PRM\big(\frac{\theta}{\alpha}\Leb\otimes\umCladeA\big)$ on $[0,\infty)\times \HfinA$ to serve as the key ingredient to our construction of 
${\tt SSIPE}_\varnothing(\alpha,\theta)$ in Definition \ref{defn:cev-beta}. We then showed in Proposition \ref{prop:identify_SSIPE} that 
 \begin{equation}\label{eq:cev-beta}
  \big(\cev{\beta}^y,\, y\ge 0\big)\sim{\tt SSIPE}_\varnothing(\alpha,\theta) \quad \text{where} \quad \cev{\beta}^y:= \fskewer\big(y, \cev\bF_\theta\big), \quad y\ge0.
 \end{equation}  
On the other hand, we claim in Theorem \ref{thm:rayknight}(ii) that an alternative construction of ${\tt SSIPE}_\varnothing(\alpha,\theta)$ is, as follows. Here, the key 
ingredient is $\mathbf{N}=\sum_{t\ge 0\colon \Delta\mathbf{X}_t>0}\delta(t,f_t)$, introduced by marking the jumps $(t,\Delta\mathbf{X}_t)$ of a spectrally positive stable 
L\'evy process $\bX$ of index $1+\alpha$, with Laplace exponent $\psi(\lambda) = \frac{\lambda^{1+\alpha}}{2^{\alpha}\Gamma(1+\alpha)}$, using the kernel 
$s\mapsto \BESQ^{s}_{0,0}(4+2\alpha)$. Then the claim is that
\[
(\beta^y,\,y\in[0,x])\sim{\tt SSIPE}_\varnothing(\alpha,\theta)\quad\text{where}\quad\beta^y := \skewer\left(y,\bN|_{[0,T^{(\theta)}_{-x}]}, x +\bX^{(\theta)} \right),
\]
and $\bX^{(\theta)}:= \bX-(1-\frac{\alpha}{\theta})\underline{\bX}$ is the L\'evy process $\bX$ perturbed at its infimum process $\underline{\bX} = \big(\underline{\bX}_t = \inf_{r\le t} \bX_r,\, t\ge 0\big)$, and $T^{(\theta)}_{-x} = \inf\{t\ge 0\colon \bX^{(\theta)}_t < -x\}$. 
	
\begin{proof}[Proof of Theorem \ref{thm:rayknight}(ii)] We first recall from \cite[Remark 5.8(i)]{PitmYor82} that the squared Bessel excursion measure  
  $\nu_{\BESQ}^{(-2\alpha)}$ conditionally on its length $\zeta=s$ is the bridge $\BESQ^{s}_{0,0}(4+2\alpha)$ of length $s$ from 0 to 0. Furthermore, we noted in 
  Proposition \ref{prop:stable_JCCP} that the lifetime intensity under $\nu_{\BESQ}^{(-2\alpha)}$ coincides with the jump size intensity of $\mathbf{X}$. By standard marking
  properties of Poisson random measures, this entails that $\mathbf{N}$ here is indeed a $\PRM[\Leb\otimes\mBxcA]$ as in \eqref{eq:imm_PPP}. 

  But $\mathbf{F}_{\perp}\sim\PRM(\Leb\otimes\umCladeA)$ is the Poisson random measure of marked excursions of $\mathbf{X}$ above $\underline{\bX}$, in which each such 
  excursion is associated with its starting level. While $\cev\bF_\theta$ actively varies this intensity by a factor of $\theta/\alpha$ and concatenates 
  \begin{equation}\label{firstconstr}
  \cev{\beta}^y= \Concat_{\text{points } (s,N_s) \text { of } \cev\bF_\theta\colon s\in [y,0]}
 	\skewer(y - s, N_s),
  \end{equation}
  the construction of Theorem \ref{thm:rayknight}(ii) modifies the scaffolding by perturbing the stable L\'evy process when it is at its infimum. Specifically, note that
  $x+\bX^{(\theta)}$ has the same excursions above $x+\underline{\bX}^{(\theta)}=x+\frac{\alpha}{\theta}\underline{\bX}$ as $\bX$ above $\underline{\bX}$. However, their intensity
  when associating each excursion with its starting level is again varied by a factor of $\theta/\alpha$. Finally, we have for all $y\in[0,x]$
  \[
  \skewer\!\left(\!y,\bN|_{[0,T^{(\theta)}_{-x}]}, x\! +\!\bX^{(\theta)} \!\right)\!=\!\!\!\Concat_{\overset{\scriptstyle\text{excursion intervals } [a,b)}{\text { of }\bX^{(\theta)}-\underline{\bX}^{(\theta)}\colon s:=x+\underline{\bX}^{(\theta)}_a\in [y,0]}}\!\!\!\skewer\left(\!y\!-\!s,\ShiftRestrict{\bN}{[a,b)}\right)\!.
  \]
  as in \eqref{firstconstr}. In particular, the construction in Theorem \ref{thm:rayknight}(ii) does indeed yield an ${\tt SSIPE}_\varnothing(\alpha,\theta)$, too. By Proposition 
  \ref{prop:sp}\ref{item:SSIP:mass} this also entails the claimed ${\tt BESQ}_0(2\theta)$ total mass evolution.
\end{proof}

\section{Hausdorff-continuous entry from generalized interval partitions} \label{sec:generalized}

In this section we explain how {\tt SSIPE} and {\tt PDIPE} can enter from the Hausdorff-completion of our space of interval partitions 
and satisfies continuity in the initial state thereby proving Theorems \ref{thm:sp} and \ref{thm:dP-IP}.


\subsection{Generalized interval partitions}\label{sec:F0}

Recall from Section \ref{sec:main_results} the (isometric) spaces $(\mathcal{K},d_H)$ of compact subsets of $[0,\infty)$ that contain 0 and $(\IPHs,d_H\circ C)$ of generalized interval partitions, which we use as completions of $(\HIPspace,d_H\circ C)$. While we emphasized the interval partition setup in the introduction, we will sometimes work on
$(\mathcal{K},d_H)$ in this section, using notation $\beta(K)$ for the collection of open subsets of $[0,\max K]\setminus K$. Also recall from Section \ref{sec:unitmass} the completion $(\overline{\mathcal{I}}_{H,1},d_H\circ C_1)$ of the subspace $(\mathcal{I}_{H,1},d_H\circ C)$ of $(\HIPspace, d_H\circ C)$. This space is isometric to the subspace
$\mathcal{K}_1:=\{K\in\mathcal{K}\colon \max K=1\}$ of $(\mathcal{K},d_H)$.  

We have equipped both $\IPH$ and $\IPHs$ with the metric $d_H\circ C$. In \cite[Theorem 2.3(b)]{Paper1-0} we showed that this is topologically equivalent to $\dHs$ on 
$\IPH$. Note that $\gamma\mapsto (\gamma,1)$ isometrically embeds $(\IPHos,d_H\circ C_1)$ in $(\IPHs,d_H\circ C)$. 
%
%
It is well-known that $(\mathcal{K},d_H)$ is locally compact and separable, while $(\mathcal{K}_1,d_H)$ is compact. 
We finally note, that 
$(\mathcal{K},d_H)$, or equivalently $(\IPHs,d_H\circ C)$, is also a metric completion of $(\IPA,d_H\circ C)$ in that every $K\in\mathcal{K}$ is a $d_H$-limit of some  
$C(\beta_n)$, $\beta_n\in \mathcal{I}_\alpha$, and likewise for $(\IPHos,d_H\circ C_1)$ in relation to $(\IPAo,d_H\circ C)$, as can be seen from Proposition \ref{prop:genIPE:cont0}.

\subsection{${\tt SSIPE}(\alpha,\theta)$ starting from a generalized interval partition}\label{sec:F:construction}

Fix $\alpha\in(0,1)$ and $\theta\ge 0$, and let $(\beta,M)\in\IPHs$. Let $K:= C(\beta,M)$. For each $U\in\beta$ independently, consider a point measure
$\mathbf{N}_U$ of law $\mathbf{P}_{\{(0,{\rm Leb}(U))\}}^{\alpha,0}$, as in Definition \ref{def:IPPA}. Also consider an independent
\begin{equation}\label{eq:Fdust}
 \mathbf{F}_{K}\sim{\tt PRM}\Big({\rm Leb}|_K\otimes\frac{1}{2\alpha}\nu_{\perp\rm cld}^{(\alpha)}\Big)
\end{equation}
so that we associate clades with the part ${\rm Leb}(K)$ of the initial mass that is associated with the partition 
points $K$ rather than the intervals in between. This point measure is non-trivial if and only if ${\rm Leb}(K)>0$. It is of the form $\sum_{i\in I}\delta(R_i,N_i)$ for some $R_i\in K$ and $N_i\in\mathcal{N}_{\rm fin}^{\rm sp}$, $i\in I$. To combine all clades into a single point measure of spindles, and to specify the appropriate scaffolding, we 
unify notation and set $U_i:=\{R_i\}$ and $\mathbf{N}_{U_i}:=N_i$. Using the natural total order on the collection $\beta(K)\cup\{U_i,i\in I\}$
of disjoint subsets of $[0,\infty)$, we define 
\begin{equation}\label{eq:NC}
 \mathbf{N}_K:=\Concat_{U\in\beta(K)\cup\{U_i,i\in I\}}\mathbf{N}_{U_i},\quad\mbox{and}\quad
 \mathbf{X}_K:=\Concat_{U\in\beta(K)\cup\{U_i,i\in I\}}\xi\big(\mathbf{N}_{U_i}\big),
\end{equation}
where the concatenation of scaffolding is defined in the natural sense of concatenating excursions above level 0. 
We will consider the skewer process of this pair of point measure of spindles and scaffolding. As a first crude check, we note that this achieves the 
correct total mass process to extend Proposition \ref{prop:sp}(iii). By the additivity of ${\tt BESQ}(0)$, we only need to consider the contribution 
from $\mathbf{F}_K$.

\begin{proposition}\label{prop:F:totalmass} Let $K\in\mathcal{K}$ and $\mathbf{F}_K$ as in \eqref{eq:Fdust}. Then setting\vspace{-0.1cm}
 $$M^0:={\textnormal{Leb}}(K)\quad\mbox{and}\quad M^y:=\sum_{{\rm points}\,(R,N)\,{\rm of}\,\mathbf{F}_K}\|\skewer(y,N,\xi(N))\|,\ y>0,\vspace{-0.1cm}
 $$
 yields $(M^y,\,y\ge 0)\sim{\tt BESQ}_{{\rm Leb}(K)}(0)$.
\end{proposition}
\begin{proof}
 Recall from the discussion after \eqref{eq:imm_PPP} that $\nu_{\perp\rm cld}^{(\alpha)}$ is the pushforward of $\overline{\nu}_{\perp\rm cld}^{(\alpha)}$ under the projection 
  map $(V,X)\mapsto\varphi(V)$ that removes allelic types. Hence, \cite[Corollary 4.11]{FVAT} yields that the pushforward of $\frac{1}{2\alpha}\nu_{\perp\rm cld}^{(\alpha)}$ 
  onto the total mass evolution is the excursion measure $\nu_{\tt BESQ}^{(0)}$ of ${\tt BESQ}(0)$ with the 
  normalization $\nu_{\tt BESQ}^{(0)}(\zeta>y)=1/2y$. For this normalization, Pitman and Yor \cite[Theorem (4.1)]{PitmYor82} showed that 
  sums over a ${\tt PRM}$ of intensity $x\nu_{\tt BESQ}^{(0)}$ are ${\tt BESQ}_x(0)$.
\end{proof}

\begin{proposition}\label{prop:entrlaw}
  Let $K\in\mathcal{K}$ and consider $(\mathbf{N}_K,\mathbf{X}_K)$ as in \eqref{eq:NC}. Then the distributions of 
  $\skewer(y,\mathbf{N}_K,\mathbf{X}_K)$, $y>0$, form an entrance law for ${\tt SSIPE}(\alpha,0)$. In particular, we have
  $\skewer(y,\mathbf{N}_K,\mathbf{X}_K)\in\IPA$ for all $y>0$ almost surely.
\end{proposition}
\begin{proof}
  This follows straight from the clade construction of ${\tt SSIPE}(\alpha,0)$ for the clades starting from $U\in\beta(K)$ and from the   
  Markov-like property of Corollary \ref{cor:min_cld:mid_spindle} for $\nu_{\perp\rm cld}^{(\alpha)}$-clades. Specifically, 
  it follows from easy computations (\cite[Lemma 6.18]{Paper1-1} and Proposition \ref{prop:min_cld:stats}\ref{item:MCS:max}) 
  that at most 
  finitely many of both types of initial clades survive to any given positive level $y$, and the concatenation of their upper cutoff processes above level $y$ yields an 
  ${\tt SSIPE}(\alpha,0)$ starting in $\IPA$, by definition.
\end{proof}

Let $\beta^0 := \beta$ and for $y>0$ let $\beta^y := \skewer(y,\mathbf{N}_K,\mathbf{X}_K)$. Similarly let $M^0 := M$ and $M^y := \|\beta^y\|$ for $y>0$. In light of the preceding results $((\beta^y,M^y),\,y\ge0)$ extends the ${\tt SSIPE}(\alpha,0)$ of Definition \ref{def:IPPA} as a Markov process. 
We refer to this process as ${\tt SSIPE}^*_K(\alpha,0)$. 
For $(\cev\beta^y,\,y\ge0)$ an independent ${\tt SSIPE}_\varnothing(\alpha,\theta)$, we call the process $\big(\big(\cev\beta^y\concat\beta^y , \|\cev\beta^y\|+M^y\big),\,y\ge0\big)$ a ${\tt SSIPE}^*_K(\alpha,\theta)$. With this definition, for $\gamma\in\IPH$,  
 ${\tt SSIPE}_\gamma(\alpha,\theta)$ equals the projection of ${\tt SSIPE}^*_{C(\gamma)}(\alpha,\theta)$ onto its first coordinate. In this sense, this definition is 
consistent with Definition \ref{def:IPPAT} of ${\tt SSIPE}_\gamma(\alpha,\theta)$.

\begin{corollary}\label{thm:semigroup*}
   The transition semigroup of the $\IPHs$-valued ${\tt SSIPE}^*_K(\alpha,\theta)$ is given by $(\kappa^{\alpha,\theta}_y,\,y\ge 0)$ as defined in Definition \ref{def:kernel:sp2}.
\end{corollary}
\begin{proof} By Proposition \ref{prop:entrlaw}, the marginal distributions of ${\tt SSIPE}^*_K(\alpha,\theta)$ are supported on $\mathcal{I}_\alpha$. In view of Proposition 
  \ref{prop:IP:transn}, the semigroup property of the extended kernels $(\kappa^{\alpha,\theta}_y,\,y\ge 0)$ of Definition \ref{def:kernel:sp2} follows as soon as we have 
  identified the marginal distribution of ${\tt SSIPE}^*_K(\alpha,\theta)$ as $\kappa^{\alpha,\theta}_y(\beta^*(K),\,\cdot\,)$. By the definition of $(\mathbf{N}_K,\mathbf{X}_K)$
  in \eqref{eq:NC}, this follows as in Proposition \ref{prop:IP:transn}, noting that the additional contributions from dust are obtained from Proposition \ref{prop:min_cld:stats}(i), 
  Lemma \ref{lem:min_cld:skewer}(ii) and standard properties of Poisson random measures.
\end{proof}
 
\subsection{Path-continuity}\label{sec:F2}

By Propositions \ref{prop:0:len}, \ref{prop:cts_imm} and \ref{prop:entrlaw}, the first coordinate of an ${\tt SSIPE}^*_K(\alpha,\theta)$ evolves as an almost surely path-continuous process in 
$(I_H,d_H^\prime)$ on any time interval $[y,\infty)$, hence on $(0,\infty)$. By \cite[Theorem 2.3(b)]{Paper1-0}, this is equivalent to path-continuity in $(I_H,d_H\circ C)$. This leaves the study of $(d_H\circ C)$-path-continuity at time 0, which also justifies referring to these processes as ${\tt SSIPE}^*_K(\alpha,\theta)$

\begin{lemma}\label{lm:dH} Let $K,\,K_n\in\mathcal{K}$, $n\ge 1$. Then $d_H(K_n,K)\rightarrow 0$ as $n\rightarrow\infty$
  if and only if 
  \begin{equation}\label{crit1}\forall_{(a,b)\in\beta(K)}\ \exists_{n_0\ge 1}\ \forall_{n\ge n_0}\ \exists_{(a_n,b_n)\in\beta(K_n)}\ a_n\rightarrow a\ \mbox{and}\ b_n\rightarrow b
  \end{equation}
  and
  \begin{equation}\label{crit2}\forall_{(n_k)_{k\ge 1}\colon n_k\rightarrow\infty}\ \forall_{(c_k,d_k)\in\beta(K_{n_k}),\,k\ge 1\colon d_k\rightarrow d\in(0,\infty],\,c_k\rightarrow c\neq d}\ (c,d)\in\beta(K).
  \end{equation}
\end{lemma}
\begin{proof} ``$\Rightarrow$'': We first prove \eqref{crit1}. Without loss of generality, $\beta(K)\neq\varnothing$. Then  
  $n_0:=\inf\{n\ge 1\colon\forall_{m\ge n}\ \beta(K_m)\neq\varnothing\}<\infty$. Let $(a,b)\in\beta(K)$ and $\varepsilon\in(0,(b-a)/2)$. Then there
  is $n_\varepsilon\ge n_0$ such that for all $n\ge n_\varepsilon$, we have a unique $(a_n,b_n)\in\beta(K_n)$ with 
  $a-\varepsilon<a_n<a+\varepsilon<b-\varepsilon<b_n<b+\varepsilon$. But since these intervals are unique and therefore cannot depend on $\varepsilon$, the endpoints must converge to $a$ and $b$. 
  
    To prove \eqref{crit2}, note that $(c,d)\subset K$ is impossible since $\varnothing\neq(c+\varepsilon,d-\varepsilon)\subset(c_k,d_k)$ for 
    infinitely many $k$, for all $\varepsilon<(d-c)/2$, so $(c+d)/2$ stays at a positive distance from $K_{n_k}$ for all $k$ sufficiently large, which 
    contradicts $d_H(K_{n_k},K)\rightarrow 0$. Also more generally, $d_H(K_{n_k},K)\rightarrow 0$ implies $d<\infty$. Now, 
    if there is $(a,b)\in\beta(K)$ with $(a,b)\cap(c,d)\neq\varnothing$ and  
    $(a,b)\neq(c,d)$, this contradicts what we proved in \eqref{crit1}, since $(c_k,d_k)\in\beta(K_{n_k})$ cannot overlap 
    any other $(a_{n_k},b_{n_k})\in\beta(K_{n_k})$, nor can they be equal for infinitely many $k$ while $(a,b)\neq(c,d)$.
    
    ``$\Leftarrow$'': Let $\varepsilon>0$. Then there are at most finitely many intervals in $\beta(K)$ of mass exceeding $2\varepsilon$. By
    \eqref{crit1}, there are intervals in $\beta(K_n)$ whose corresponding endpoints differ by less than $\varepsilon$. Therefore, all points of 
    $K_n$ are within $\varepsilon$ of an element of $K$. 
    
    Now assume for contradiction that there are $\varepsilon>0$ and $(n_k)_{k\ge 1}$ with $n_k\rightarrow\infty$ such that $K$ has points $x_k$ 
    that are not within $2\varepsilon$ of $K_{n_k}$, $k\ge 1$. As $K$ is compact, $(x_k)_{k\ge 1}$ has a convergent 
    subsequence, so we may assume that $(n_k)_{k\ge 1}$ was chosen with $x_k\rightarrow x\in K$. But then 
    $x\in K$ is not within $\varepsilon$ of $K_{n_k}$ for $k$ sufficiently large. Hence, 
    $(x-\varepsilon,x+\varepsilon)\subseteq(c_k,d_k)$ for some $(c_k,d_k)\in\beta(K_{n_k})$. Again, by successively passing to suitable 
    subsequences, we may assume that $c_k\rightarrow c\in[0,x-\varepsilon]$ and $d_k\rightarrow d\in[x+\varepsilon,\infty]$. By
    \eqref{crit2}, we have $(c,d)\in\beta(K)$, which contradicts $x\in K$.
\end{proof}

\begin{proposition}\label{prop:genIPE:cont0}
 Fix $K\in\mathcal{K}$ and consider $(\mathbf{N}_K,\mathbf{X}_K)$ as in \eqref{eq:NC}. Let
  $\beta^y=\skewer(y,\mathbf{N}_K,\mathbf{X}_K)$, $y>0$. Then $d_H(C(\beta^y),K)\rightarrow 0$ almost surely as $y\downarrow 0$.
\end{proposition}
\begin{proof} We check the criterion of Lemma \ref{lm:dH}. To establish that \eqref{crit1} holds, fix an arbitrary interval $U = (a,b)\in \beta(K)$, and we will show that for all sufficiently small $y$, there exists $(a(y),b(y))\in\beta^y$ so that $a(y)\to a$ and $b(y)\to b$ as $y\downto 0$. Now, consider $K_a:=K\cap[0,a]\in\mathcal{K}$ with
  similarly restricted $(\mathbf{N}_{K_a},\mathbf{X}_{K_a})$ obtained from $(\mathbf{N}_K,\mathbf{X}_K)$. By 
  Proposition \ref{prop:F:totalmass} and additivity of ${\tt BESQ}(0)$, we have
  $$a(y):=\big\|\skewer(y,\mathbf{N}_{K_a},\mathbf{X}_{K_a})\big\|\rightarrow a\quad\mbox{as $y\downarrow 0$.}$$
  By Definition \ref{def:IPPA}, the initial spindle of $\mathbf{N}_U$ is $\mathbf{f}_U \sim {\tt BESQ}_{b-a}(-2\alpha)$, so for all $y < \zeta(\mathbf{f}_U)$ we get $(a(y),b(y))\in \beta^y$ with $b(y) = a(y)+\mathbf{f}_U(y)$. By the continuity of \BESQ, the desired convergence holds. This satisfies \eqref{crit1}.
  
  For \eqref{crit2}, consider the maximal block mass at level $y>0$ among blocks that do not arise from an initial spindle $\mathbf{f}_U$ in a clade 
  $\mathbf{N}_U$, $U\in\beta(K)$. This collection consists of two types. 
  
  The first type is blocks arising from $\mathbf{N}_U$, $U\in\beta(K)$, that are not obtained from the respective initial spindle $\mathbf{f}_U$. The ${\tt BESQ}_{\max K-{\rm Leb}(K)}(0)$ sum of total masses of these clades is continuous. By an elementary analysis argument, the sum of contributions of initial spindles, $\sum_{U\in\beta(K)}\mathbf{f}_U$, is continuous as well, and satisfies
  $$\sum_{U\in\beta(K)}\mathbf{f}_U(0) = \max K-{\rm Leb}(K).$$
  Thus, the other spindles of these clades have a combined total mass that vanishes as $y\downarrow 0$. 
  
  The second type is blocks arising from $\mathbf{F}_K$. Their masses can
  be bounded by the total masses in each $\nu_{\perp\rm cld}^{(\alpha)}$-clade, which form a 
  ${\tt PRM}\big({\rm Leb}|_K\otimes\nu_{\tt BESQ}^{(0)}\big)$. 
  Fix $\varepsilon > 0$. There are at most finitely many clades in $\mathbf{F}_K$ whose total mass processes ever exceed $\varepsilon$, and they all start continuously from $0$ mass. Hence, there is some $z>0$ such that no block of this type in $\beta^y$, $y<z$, has mass greater than $\varepsilon$. Hence, this maximum must also tend to zero. 
  
  This implies \eqref{crit2}, and we conclude Hausdorff convergence by Lemma \ref{lm:dH}.
\end{proof}

\begin{corollary}
 The ${\tt SSIPE}_K^*(\alpha,\theta)$ is $(d_H\circ C)$-path-continuous.
\end{corollary}

\begin{proof}
 The concatenation operation defining the ${\tt SSIPE}_K^*(\alpha,\theta)$ from ${\tt SSIPE}_\varnothing(\alpha,\theta)$ and ${\tt SSIPE}_K^*(\alpha,0)$ is $(d_H\circ C)$-continuous. As the two component processes are each continuous, by Proposition \ref{prop:cts_imm} in the first component and by Propositions \ref{prop:0:len}\ref{item:IPPA:cts}, \ref{prop:entrlaw}, and \ref{prop:genIPE:cont0} in the second, so is their concatenation.
\end{proof}

\subsection{Continuity in the initial condition for ${\tt SSIPE}^*(\alpha,\theta)$}

Consider Brownian motion $B$ and $\gamma\in[-1,1]$. Recall from \cite{HarrShep81} that an \em associated skew Brownian motion \em can be obtained as the unique strong solution to the equation
$$X_\gamma(t)=B(t)-\gamma\ell_\gamma(t),\quad\mbox{where }\ell_\gamma(t)=\lim_{h\downarrow 0}\frac{1}{2h}\int_0^t1\{-h<X_\gamma(s)<h\}ds,\quad t\ge 0,$$
where we refer to $\ell_\gamma$ as the \em local time of $X_\gamma$\em. Figure \ref{fig:PW18} illustrates this as a decomposition of $B$ into 
$\gamma\ell_\gamma$ and $X_\gamma$. This figure also illustrates the implications for the Brownian local times, which we also state as a proposition.

\begin{figure}[t]
 \centering
 \includegraphics{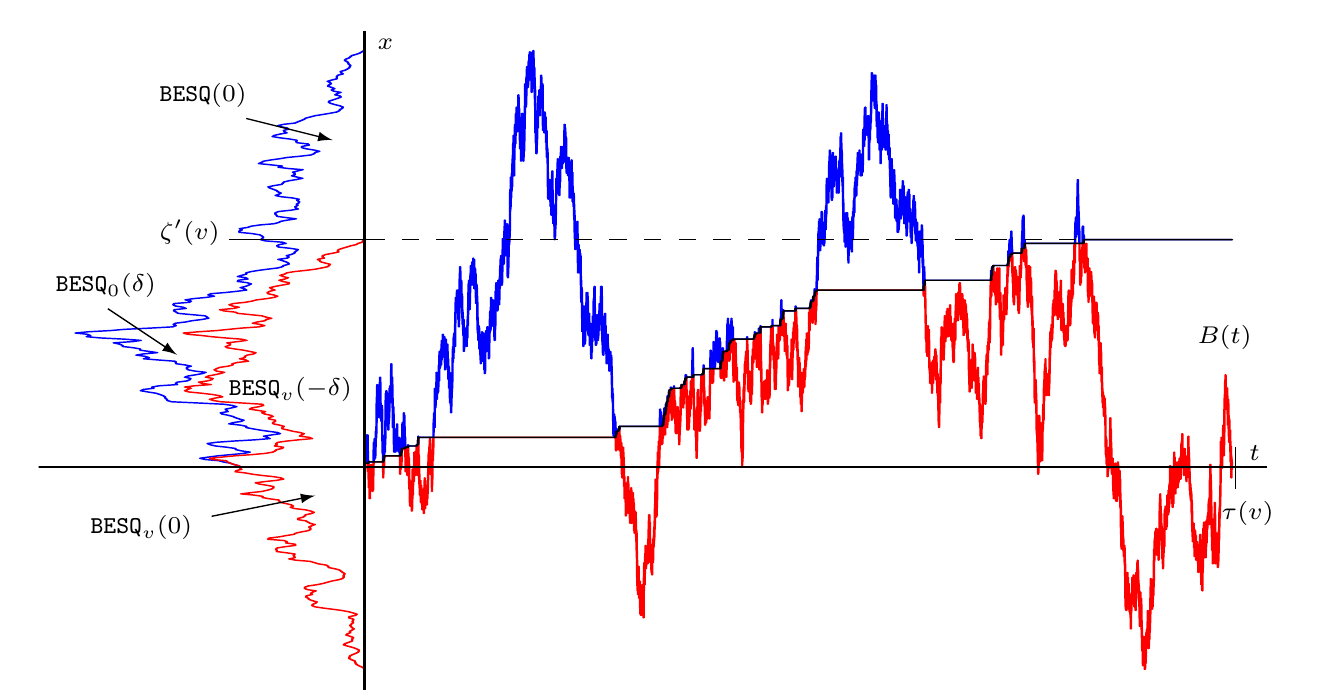}%
 \caption{This figure is from \cite{PW18}, showing the Brownian motion split along the increasing path $\gamma\ell_\gamma$, with the positive excursions of $X_\gamma$ in blue, above the increasing path, and the negative excursions of $X_\gamma$ in red, below the increasing path. The construction is stopped at $\tau(v)$, when the inverse local time of Brownian motion at level 0 exceeds $v>0$. The left-hand side of the figure shows the total Brownian local times (up to time $\tau(v)$), split according to the red and blue contributions, and identified as squared Bessel processes of dimensions $\delta$, $-\delta$ and 0, where $\delta=2\alpha$.\label{fig:PW18}\vspace{-0.1cm}}
\end{figure}

\begin{proposition}[Theorem 1.3 of \cite{PW18}]\label{propPW18} Let $\alpha\!\in\!(0,1)$ and $\gamma\!=\!1/(1\!+\!2\alpha)$. For $B$ Brownian motion and $X_\gamma$ 
  the associated skew Brownian motion with local time $\ell_\gamma$, let $S(x)=\inf\{t\!\ge\! 0\colon\gamma\ell_\gamma(t)\!>\!x\}$, $x\!\ge \!0$, and
  consider the jointly continuous space-time local times $(L(x,t),x\!\in\!\mathbb{R},t\!\ge\! 0)$ of $B$ with inverse 
  $\tau(v)\!=\!\inf\{t\!\ge\! 0\colon L(0,t)\!>\!v\}$ at level 0. Then the following families of random variables are independent
  \begin{itemize}\item $Z_0:=(L(x,S(x)),\,x\ge 0)\sim{\tt BESQ}_0(2\alpha)$
    \item $Z_v^\prime:=(L(x,\tau(v))-L(x,S(x)\wedge\tau(v)),\,x\ge 0)\sim{\tt BESQ}_v(-2\alpha)$ for all $v\ge 0$.
  \end{itemize}
  For each $v\ge 0$, the random level $\zeta^\prime(v):=
  \gamma\ell_\gamma(\tau(v))$ is almost surely finite and coincides with the absorption time of $Z_v^\prime$. Conditionally given $\zeta^\prime(v)=a$,
  \begin{itemize}\item the process $Z_{0,v}:=(L(x,S(x)\wedge\tau(v)),\,x\ge 0)$ is independent of $Z_v^\prime$ and evolves as ${\tt BESQ}_0(2\alpha)$ on $[0,a]$ and as ${\tt BESQ}(0)$ on $[a,\infty)$. 
  \end{itemize}
\end{proposition}

\begin{corollary}\label{cor1PW18}  In the setting of Proposition \ref{propPW18}, fix $v_0>0$ and $y>0$. Then there is a random $\Delta>0$ 
  such that $\zeta^\prime(v)=\zeta^\prime(v_0)$ and $Z_v^\prime|_{[y,\infty)}=Z_{v_0}^\prime|_{[y,\infty)}$ for all $v\in(v_0-\Delta,v_0+\Delta)$. In 
  particular, $Z_{0,v}=Z_{0,v_0}$ for all $v\in(v_0-\Delta,v_0+\Delta)$.
\end{corollary}
\begin{proof} Consider the Poisson random measure of excursions of $B$ away from 0. Since $\ell_\gamma(\tau(v_0))>0$ almost surely and 
  $\ell_\gamma$ only increases when $B(t)=\gamma\ell_\gamma(t)$, it suffices to show that there is $\Delta>0$ such that the Poisson random measure has
  no excursions of height greater than $y\wedge\gamma\ell_\gamma(\tau(v_0/2))>0$ on an interval $(v_0-\Delta,v_0+\Delta)$. This follows 
  from the properties of Poisson random measures, notably independence properties and finite intensity of large excursions.  
\end{proof}

\begin{corollary}\label{cor0PW18} In the setting of Proposition \ref{propPW18}, let
  \begin{align*}
    \mathbf{G} :=\sum_{x\ge 0\colon g_x \not\equiv 0} \delta(x,g_x)\ \ \text{where}\ \ 
    g_x(y) := L(x \!+\! y,S(x)) - L(x \!+\! y,S(x-)),\ y\ge 0.
  \end{align*}
  Then $\mathbf{G}$ has law ${\tt PRM}({\rm Leb} \otimes 2\alpha\nu_{\tt BESQ}^{(0)})$ and is independent of $Z_v^\prime$, for any $v\!\ge\!0$. More\-over, the rest\-riction $\mathbf{G}_v:=\mathbf{G}|_{[0,\zeta'(v)]}$ is conditionally independent of $Z_v^\prime$ given $\zeta^\prime(v)$.
\end{corollary}
\begin{proof} Proposition \ref{propPW18} records implications about local time processes of the underlying decomposition of $B$ established in 
  \cite[Lemma 2.3]{PW18}. While the positive excursions of skew Brownian motion are well-known to be Brownian excursions 
  \cite{ItoMcKean,Walsh78} and the point measure construction of ${\tt BESQ}(2\alpha)$ from a ${\tt PRM}(2\alpha\nu_{\tt BESQ}^{(0)})$ 
  is also well-known \cite{PitmYor82}, the setting of Proposition \ref{propPW18} is that of a given Brownian motion $B$, and we refer to 
  \cite[Lemma 2.3]{PW18} as it yields the independence claims in this setting, and also provides a direct identification of the rate 
  $\mu_\gamma^+=2\gamma/(1-\gamma)=2\alpha$ of Brownian excursions in the given parametrisation.
\end{proof}

By \cite[Corollary 4.11]{FVAT}, there is a kernel $\kappa_\perp(g,dN)$ that associates with
an excursion $g$ a clade $N$ with total mass evolution $g$, such that
\begin{equation}\label{eq:kappa_perp}
 \mbox{the push-forward of $\nu_{\tt BESQ}^{(0)}$ under $\kappa_\perp$ is $\nu_{\perp\rm cld}^{(\alpha)}$.}
\end{equation}
By \eqref{eq:clade_from_F}, a clade with an initial spindle $f$ of height $a$ can be built from $f$, along with $\mathbf{F}\sim{\tt PRM}({\rm Leb}|_{[0,a]}\otimes\nu_{\perp\rm cld}^{(\alpha)})$. Thus, if we instead begin with $f$ and a point measure of $\mathbb{R}$-valued excursions $G$, as in Corollary \ref{cor0PW18}, then we can obtain a clade by first marking the points of $G$ via $\kappa_\perp$ to obtain a point measure $F$ of clades, and then assembling these into a single clade via \eqref{eq:clade_from_F}, 
but with the convention that, instead of concatenating in order of increasing $y$, we do so in order of decreasing $y$, in order to respect the levels at which excursions arise in Figure \ref{fig:PW18}. Let $\widetilde{\kappa}((f,G),dN)$ denote the kernel that carries out this construction.



\begin{lemma}\label{lem:clade_from_BM}
 Fix $v > 0$. Let $Z_v^\prime$ and $\mathbf{G}_v$ be constructed from a Brownian motion $B$, as in Proposition \ref{propPW18} and Corollary \ref{cor0PW18}. 
 Applying $\widetilde{\kappa}$ to the pair $(Z_v^\prime,\mathbf{G}_v)$ we obtain a clade $\mathbf{N}_v\!\sim\!\mathbf{P}_{\{(0,v)\}}^{\alpha,0}$. 
\end{lemma}

\begin{proof}
 Proposition \ref{propPW18} notes that $Z_v^\prime \sim \BESQ_{v}(-2\alpha)$. Corollary \ref{cor0PW18} then observes that $\mathbf{G}_v$ has conditional law ${\tt PRM}({\rm Leb}|_{[0,\zeta'(v)]}\otimes 2\alpha\nu_{\tt BESQ}^{(0)})$ and is conditionally independent of $Z_v^\prime$, given $\zeta'(v)$. The lemma follows by \eqref{eq:clade_from_F}.
\end{proof}

\begin{corollary}\label{cor2PW18} Consider $K\in\mathcal{K}$ and Brownian motion $W$ with inverse local time $\tau_W$ at level $0$. For each interval 
  $U=(a,b)\in\beta(K)$, consider $B_U:=(W(\tau_W(a)+t),t\ge 0)$. 
  Let $\mathbf{N}_{U}$ denote the clade arising from Lemma \ref{lem:clade_from_BM} applied to $B_U$ with $v\!=\!{\rm Leb}(U)$, so that $\mathbf{N}_{U}\!\sim\!\mathbf{P}_{\{(0,{\rm Leb}(U))\}}^{\alpha,0}$. 
  If we restrict the Poisson random measure of excursions of $W$ to $K$, map each excursion onto its total local time process and mark it by 
  $\kappa_\perp$, the resulting Poisson random measure $\mathbf{F}_K$ is distributed as in \eqref{eq:Fdust}. Furthermore, all $\mathbf{N}_U$, $U\in\beta(K)$, and $\mathbf{F}_K$ are independent. 
\end{corollary}
\begin{proof} Clearly, $B_U$ is a Brownian motion for each $U\in\beta(K)$. Also, the construction of Proposition \ref{propPW18} only depends
  on $B$ restricted to $[0,\tau(v))$. Therefore, the constructions are all independent as $U\in\beta(K)$ varies, and further independent of the 
  restriction to $K$ of the Poisson random measure of excursions of $W$. That the Brownian excursion measure is pushed forward to 
  $\nu_{\tt BESQ}^{(0)}$ was observed by Pitman and Yor \cite{PitmYor82}. The claimed distributions, after applying $\kappa_\perp$ and $\widetilde{\kappa}$, follow from 
  \eqref{eq:kappa_perp} and Lemma \ref{lem:clade_from_BM}.  
\end{proof}

Figure \ref{fig:PW18} shows the case $x=0$ of a more general setting of Burdzy et al. \cite{BurdChen01,BBKM01,BurdKasp04} where multiple skew Brownian motions of the same skewness parameter but starting from different initial states are driven by the same Brownian motion. 
Specifically, for any $x\in\mathbb{R}$, or any countable (but not uncountable) collection of values $x\in\mathbb{R}$, there are unique strong
solutions to $X^x_\gamma(t)=x+B(t)-\gamma\ell^x_\gamma(t)$, which we can again write to decompose 
$B(t)=X^x_\gamma(t)-x-\gamma\ell^x_\gamma(t)$. Instances of this can be found in Figure \ref{fig:PW18} to the right of any zero of Brownian
motion, when $X_\gamma$ is negative. It is natural to consider zeroes that are stopping times, but more relevant for us to consider the start of a Brownian excursion or indeed, the corresponding equation driven by Brownian motion conditioned to stay positive, ${\tt BES}_0(3)$.

\begin{proposition}[Lemma 2.4 and its proof in \cite{BurdKasp04}]\label{prop:BK} Let $\vec{B}\sim{\tt BES}_0(3)$. Then 
  $$\vec{X}^x_\gamma(t)=x+\vec{B}(t)-\gamma\vec{\ell}_\gamma^x(t),\ \ \mbox{where } \vec{\ell}_\gamma^x(t)=\lim_{h\downarrow 0}\frac{1}{2h}\int_0^t\!\!1\{-h\!<\!\vec{X}_\gamma^x(s)\!<\!h\}ds,\ \  t\ge 0,$$
  has a unique strong solution for each $x\in\mathbb{R}$. Furthermore, for any sequence $x_n\downarrow 0$, we have $\vec{\ell}_\gamma^{-x_n}(\infty)\rightarrow 0$ almost surely, as $n\rightarrow\infty$.
\end{proposition}

\begin{corollary}\label{corBK} Let $\widetilde{B}$ be a Brownian excursion distributed according to It\^o's measure conditioned on height exceeding $y$. Then
    $$\widetilde{X}^x_\gamma(t)=x+\widetilde{B}(t)-\gamma\widetilde{\ell}_\gamma^x(t),\ \ \mbox{where } \widetilde{\ell}_\gamma^x(t)=\lim_{h\downarrow 0}\frac{1}{2h}\int_0^t\!\!1\{-h\!<\!\widetilde{X}_\gamma^x(s)\!<\!h\}ds,\ \  t\ge 0,$$
    has a unique strong solution for each $x\in\mathbb{R}$. Furthermore, for any sequence of nonnegative $x_n\rightarrow 0$, there is (random) 
    $N\ge 1$ a.s., such that for all $n\ge N$, the solution $\widetilde{X}^{-x_n}_\gamma$ has precisely one positive excursion during which the
    combined height of $\widetilde{X}^{-x_n}_\gamma$ and $x_n+\gamma\widetilde{\ell}_\gamma^{-x_n}$ exceeds level $y$.
\end{corollary}
\begin{proof} Denote by $\varepsilon>0$ the minimum of $\widetilde{B}$ between the first and last visits of level $y$. Consider the Williams 
  decomposition of $\widetilde{B}$ into two ${\tt BES}_0(3)$. Specifically, $\widetilde{B}$ stopped when it first hits level $y$, can be seen as the part of $\vec{B}\sim{\tt BES}_0(3)$ before reaching level $y$, and Proposition \ref{prop:BK} yields $N$ such that $\gamma\vec{\ell}_\gamma^{-x_n}(\infty)<\varepsilon$ for all $n\ge N$. But then $\vec{\ell}_\gamma^{-x_n}=\widetilde{\ell}_\gamma^{-x_n}$ up to the first time $\widetilde{B}$ hits $y$, and $\widetilde{\ell}_\gamma^{-x_n}$ then stays constant until the last time $\widetilde{B}$ hits $y$, for all $n\ge N$. Hence, the 
  entire set of times where $\widetilde{B}$ is above level $y$ is part of the same excursion interval of $\widetilde{X}_\gamma^{x_n}$, for all $n\ge N$.
\end{proof}

\begin{proposition}\label{prop:gencontinitial} Let $K$, $K_n\in\mathcal{K}$, $n\ge 1$, with $d_H(K_n,K)\rightarrow 0$. Consider 
  $((\beta^y_n,M_n^y),y\,\ge 0)\sim{\tt SSIPE}_{K_n}^*(\alpha,\theta)$, $n\ge 1$, and 
  $((\beta^y,M^y),\,y\ge 0)\sim{\tt SSIPE}_{K}^*(\alpha,\theta)$. Then, for all $y>0$, we have $(\beta_n^y,M_n^y)\rightarrow(\beta^y,M^y)$
  in distribution in the product topology induced by the metric $d_\alpha$ on $\mathcal{I}_\alpha$ and the Euclidean distance on $[0,\infty)$.
\end{proposition}
\begin{proof} W.l.o.g., $\theta=0$. We will couple point-measure-scaffolding pairs $(\mathbf{N}_{K},\mathbf{X}_K)$ and 
  $(\mathbf{N}_{K_n},\mathbf{X}_{K_n})$, $n\ge 1$. We assume for simplicity that $\beta(K)$ consists of infinitely many intervals. The case of
  finitely many intervals can be seen similarly. Denote by $U^{(m)}:=(a^{(m)},b^{(m)})$ the $m$th-largest block of $\beta(K)$, with 
  ties broken from left to right, $m\ge 1$. 
  By Lemma \ref{lm:dH}, there is $n_m\ge 1$ such that for all 
  $n\ge n_m$, we can find $U_n^{(m)}:=(a_n^{(m)},b_n^{(m)})\in\beta(K_n)$ with $(a_n^{(m)},b_n^{(m)})\rightarrow(a^{(m)},b^{(m)})$ 
  as $n\rightarrow\infty$. Furthermore, we can choose $n_m$ so large that all the $m$ largest limiting intervals can be matched in $\beta(K_n)$ 
  so as to repect the left-to-right order of the limiting intervals for all $n\ge n_m$, and so that 
  $|{\rm Leb}(U_n^{(j)})-{\rm Leb}(U^{(j)})|<2^{-m-j}$ for 
  all $j=1,\ldots,m$. 
  
  For each $m\ge 1$ and $n\ge n_m$, this partitions $K$ and $K_n$ into $m+1$ disjoint parts (possibly empty), which we denote by $K^{(m,j)}$ and $K^{(m,j)}_n$, $0\le j\le m$, where the 
  $j^{\text{th}}$ is the part to the right of $U^{(j)}$, respectively $U_n^{(j)}$, $1\le j\le m$, with $K^{(m,0)}$ and $K^{(m,0)}_n$ being the remainder 
  beyond the left-most matched interval. We note that the blocks of $U^{(1)},\ldots,U^{(m)}$ are not indexed in left-to-right order, and therefore neither are the $K^{(m,j)}$. We further suppose that $n_m$ is chosen large enough so that for $n\ge n_m$, the diameters of
  $K^{(m,j)}$ and $K^{(m,j)}_n$ differ by less than $2^{-m-j}$. 
  
  \begin{figure}
   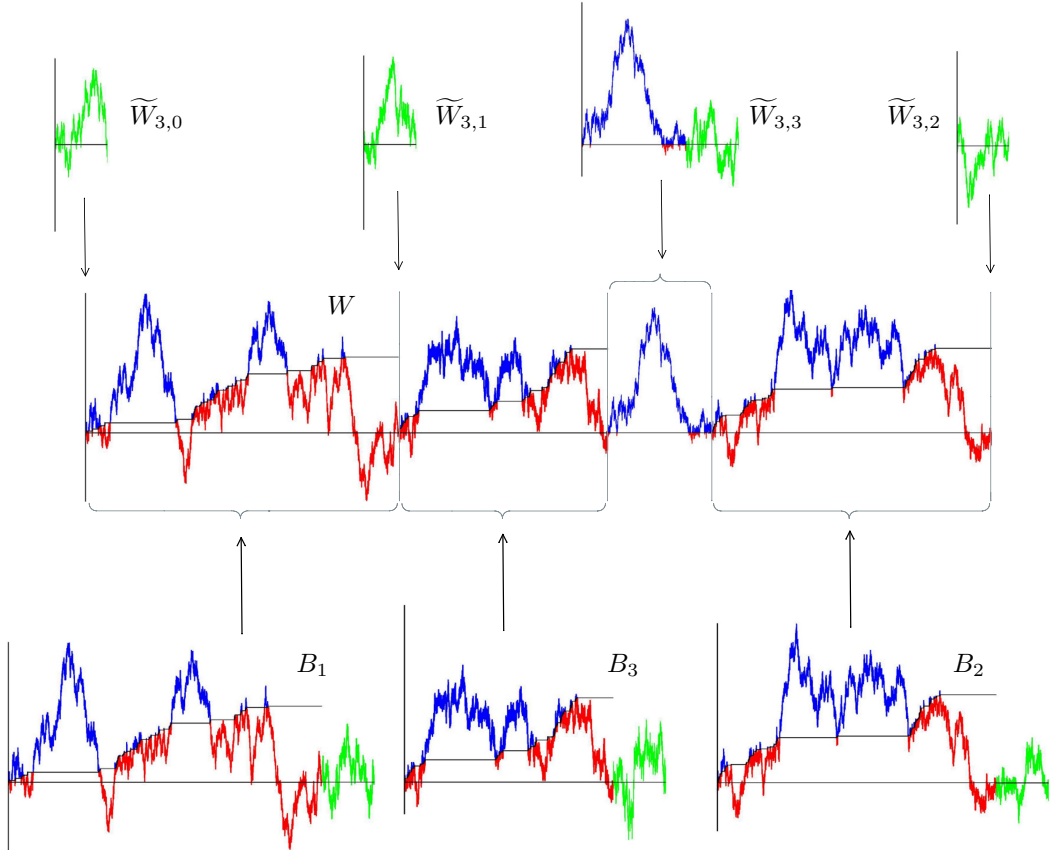
   \caption{The middle panel illustrates the construction of Corollary \ref{cor2PW18}, decomposing Brownian motion $W$ into path segments used to construct $\mathbf{N}_U$, $U\in\beta(K)$, and $\mathbf{F}_K$. Here, 
   $\beta(K)$ consists of three intervals and $K$ has dust between the second and the third. 
   The remainder of this figure 
   shows how this construction is modified for 
   the proof of Proposition \ref{prop:gencontinitial}, in which each path segment is extended by an independent Brownian motion, shown in green.\label{fig:coupling}}
  \end{figure}
  
  Now consider a Brownian motion $W$, the construction of $(\mathbf{N}_K,\mathbf{X}_K)$ in Corollary \ref{cor2PW18} and 
  \eqref{eq:NC}, and an independent
  family $W_j$, $j\in\mathbb{Z}$, of Brownian motions. We denote by $(\mathbf{N}^{(j)},\mathbf{X}^{(j)})$, $j\ge 1$, the clades in $(\mathbf{N}_K,\mathbf{X}_K)$ corresponding to the respective blocks $U^{(j)}$. We will enhance the construction by considering, for each $U^{(j)} = (a^{(j)},b^{(j)})$, $j\ge 1$, the 
  Brownian motion $B_j$ obtained by concatenating $\shiftrestrict{W}{[\tau_W(a^{(j)}),\tau_W(b^{(j)})]}$ with the independent Brownian motion $W_j$. For comparison, in Corollary \ref{cor2PW18}, $B_{U^{(j)}} = \shiftrestrict{W}{[\tau_W(a^{(j)}),\infty)}$. 
  Let $Z^\prime_j$ and $\zeta^\prime_j$ be derived from $B_j$ as in Proposition \ref{propPW18}. Let $\mathbf{G}_{j}$ denote the associated ${\tt PRM}(2\alpha\nu_{\tt BESQ}^{(0)})$, as in Corollary \ref{cor0PW18}, and let $\mathbf{F}_{j}$ denote a ${\tt PRM}(2\alpha\umCladeA)$ obtained by applying $\kappa_\perp$ to the points of $\mathbf{G}_{j}$. 
  Then $\mathbf{N}^{(j)}$ is formed by concatenating the initial spindle $\mathbf{f}^{(j)} = Z^\prime_j$ with the clades of $\restrict{\mathbf{F}_{j}}{[0,\zeta^\prime_j]}$, as in \eqref{eq:clade_from_F}. 
  
  Now fix $m\ge 1$ and $n\ge n_m$. For $1\le j\le m$, we construct $\mathbf{N}_{n}^{(j)} \sim \mathbf{P}^{\alpha,0}_{\{(0,{\rm Leb}(U_n^{(j)}))\}}$ from $B_j$ and $\mathbf{F}_{j}$ in the same manner as above. Let $\mathbf{X}_n^{(j)} := \xi\big(\mathbf{N}_{n}^{(j)}\big)$. To emphasize, $\mathbf{N}^{(j)}$ and $\mathbf{N}_n^{(j)}$ are coupled by: (1) constructing them from the same Brownian motion $B_j$ and (2) using the same $\kappa_\perp$-marked excursions in $\mathbf{F}_j$.
  
  We now address the omitted blocks and dust that make up the segments $K^{(m,j)}$ and $K_n^{(m,j)}$. For each $0\le j\le m$, we denote by $\widetilde{W}_{m,j}$ the part of $W$ relevant for $K^{(m,j)}$, i.e.\  the interval $\big[\tau_W\big(\min\big(K^{(m,j)}\big)-\big),\tau_W\big(\max\big(K^{(m,j)}\big)\big)\big]$, concatenated with $W_{-j}$. We denote by $\big(\widetilde{\mathbf{N}}^{(m,j)},\widetilde{\mathbf{X}}^{(m,j)}\big)$ the part of $(\mathbf{N}_K,\mathbf{X}_K)$ that corresponds to $K^{(m,j)}$; this scaffolding and spindles pair arises from $\widetilde{W}_{m,j}$ and the compact set $K^{(m,j)}$, shifted to start at the origin, via the construction of Corollary \ref{cor2PW18} and \eqref{eq:NC}. We construct $\big(\widetilde{\mathbf{N}}_n^{(m,j)},\widetilde{\mathbf{X}}_n^{(m,j)}\big)$ in the same manner, from $\widetilde{W}_{m,j}$ and the set $K^{(m,j)}_n$ shifted back to the origin. Finally, we concatenate the $2m+1$ point measures and scaffoldings to form $\big(\mathbf{N}_{K_n}^{(m)},\mathbf{X}_{K_n}^{(m)}\big)$.  
  
  Fix $y>0$. For $n\ge n_m$ and $1\le j\le m$, let $\beta^{(j)}(y) := \skewer(y,\mathbf{N}^{(j)},\mathbf{X}^{(j)})$. We correspondingly define $\beta^{(j)}_n(y)$, $\widetilde\beta^{(m,j)}(y)$, and $\widetilde\beta^{(m,j)}_n(y)$ as respective skewers of $\mathbf{N}^{(j)}_n$, $\widetilde{\mathbf{N}}^{(m,j)}$, and $\widetilde{\mathbf{N}}^{(m,j)}_n$. 
 We will show that under our coupling, plus some further coupling in the case ${\rm Leb}(K)>0$, we get $(\beta_n^y,M_n^y)=(\beta^y,M^y)$ for all sufficiently large $n$. It suffices to prove that for all sufficiently large $m$ and all sufficiently large $n$ depending on $m$,
  \begin{enumerate}[label=(A\arabic*),ref=(A\arabic*)]
   \item $\beta^{(j)}(y) = \beta_{n}^{(j)}(y)$ for all $1\le j\le m$, and\label{item:GCI:bigblocks}
   \item $\widetilde\beta^{(m,j)}(y) = \widetilde\beta_{n}^{(m,j)}(y)$ for all $0\le j\le m$.\label{item:GCI:smallblocks}
  \end{enumerate}
  
  First we prove \ref{item:GCI:bigblocks}. Fix $m>0$. Recall that by Corollary \ref{cor1PW18}, there exists some $\Delta>0$ such that for all $j=1,\ldots,m$ and all $n\ge n_{m}$ satisfying $|{\rm Leb}(U_n^{(j)})-{\rm Leb}(U^{(j)})|<\Delta$, we have $\beta_{n}^{(j)}(y)=\beta^{(j)}(y)$ (as both partitions arise from $B_j$ and $\mathbf{F}_j$). By Lemma \ref{lm:dH}, there is some a.s.\ finite $R_m\ge n_m$ sufficiently large so that
  $$|{\rm Leb}(U_n^{(j)})-{\rm Leb}(U^{(j)})|<\Delta\mbox{ for all }n\ge R_m,\ 1\le j\le m.$$
  This proves \ref{item:GCI:bigblocks}.
  
  We now prove \ref{item:GCI:smallblocks} in the special case ${\rm Leb}(K) = 0$. Since $\sum_{j\in\mathbb{Z}}2^{-|j|}<\infty$, there is a (random) $m_1$ such that for all $m\ge m_1$, none of the Brownian motions $W_j$, $j\in\mathbb{Z}$, has excursions above height $y$ before reaching an 
  inverse local time of $2^{-m-|j|}$ at level 0. There is also $m_2\ge m_1$ such that for all $j>m_2$, the clade $\mathbf{N}^{(j)}$ has 
  height less than $y$. Recall that we choose $n_m$ to be sufficiently large so that, for $n\ge n_m$, $|{\rm Leb}(K^{(m,j)}_n) - {\rm Leb}(K^{(m,j)})| < 2^{-m-j}$. Thus, for $m\ge m_2$ and $n\ge n_m$, the segments of the processes $\widetilde{W}_{m,j}$ used to construct the $\widetilde{\mathbf{N}}_n^{(m,j)}$, $0\le j\le m$, never reach height $y$. Thus, $\widetilde\beta^{(j)}(y) = \widetilde\beta_{n}^{(j)}(y)$ as desired, with both equaling $\varnothing$.
  
  Now consider the case ${\rm Leb}(K)>0$. Let $\mathbf{F}_K$ be as in Corollary \ref{cor2PW18}, describing the point measure of clades arising from the dust in $K$. In this case, $\mathbf{F}_K$ may also have clades that exceed level $y$, and the 
  present coupling does not allow us to conclude that $\widetilde\beta^{(m,j)}(y) = \widetilde\beta_{n}^{(m,j)}(y)$ for $n$ sufficiently large. 
  Recall that blocks in these interval partitions correspond to excursions in $\widetilde W^{(m,j)}$ that exceed level $y$. While our bound $n \ge n_{m_1}$ still precludes excursions from the $W_{-j}$ from contributing to $\widetilde\beta_{n}^{(m,j)}(y)$, the same cannot be said of $W$. Indeed, any clades exceeding level $y$ in $\mathbf{F}_K$ would correspond to large excursions of $W$ that are then partitioned off into the processes $\widetilde W^{(m,j)}$. 
  These same excursions are used to construct both partitions, but the application of the kernel $\kappa_\perp$ to obtain clades from such excursions has so far not been coupled appropriately.
  
  Suppose $\mathbf{F}_K$ has $H$ clades of height exceeding $y$, associated with excursions 
  $\widetilde{E}^{(i)}:=(W(\tau_W(v^{(i)}-)+s),\,0\le s\le\tau_W(v^{(i)})-\tau_W(v^{(i)}-))$, $1\le i\le H$. By virtue of having each $v^{(i)}$ belong to $K$, in particular, these excursions of $W$ are partitioned off into the $\widetilde{W}_{m,j}$ (rather than appearing in the $B_{j}$). Fix $m\ge m_2$. For $1\le i\le H$ let $J_i$ denote the index of the component to which $v^{(i)}$ belongs, $v^{(i)} \in K^{(m,J_i)}$. Let $\bar v^{(i)} := v^{(i)} - b^{(J_i)}$, so that the relevant excursion now appears in $\widetilde{W}_{m,J_i}$ at local time $\bar v^{(i)}$. 
  
  For each $i=1,\ldots,H$ and $n\ge 1$, we have two cases: either $\bar v^{(i)} \in K^{(m,J_i)}_n - b_n^{(J_i)}$ or $\bar v^{(i)}\in\big(\bar c_n^{(i)},\bar d_n^{(i)}\big)\in\beta(K^{(m,J_i)}_n - b_n^{(J_i)})$. In the first case, this point lies in dust rather than in a small block, and we can couple the two clades. 
  In this case we adopt the convention $\bar c_n^{(i)} := \bar d_n^{(i)} := \bar v^{(i)}$. 
  In both cases, note that $v^{(i)}$ a.s.\ does not lie on the boundary of any block of $\beta(K)$. 
  Lemma \ref{lm:dH} then implies $\bar d^{(i)}_n - \bar c^{(i)}_n \to 0$ as $n\to\infty$.  
  In the second case, consider the maximal height $h_n^{(i)}$ of the excursions of $\widetilde{W}_{m,J_i}$ in the interval $\big(\tau\big(\bar c_n^{(i)}\big),\tau\big(\bar v^{(i)}\big)\big)$, where $\tau$ denotes inverse local time in $\widetilde{W}_{m,J_i}$. In particular, the associated $\gamma\ell_\gamma$ process as in Figure \ref{fig:PW18} will finish at some 
  $x_n^{(i)}\le h_n^{(i)}$, and since $h_n^{(i)}\rightarrow 0$ a.s., we also have $x_n^{(i)}\rightarrow 0$ a.s.. 
  By Corollary \ref{corBK}, there is $N\ge 1$ such that for each $1\le i\le H$ and $n\ge N$ in the second case, there is a single excursion above this $\gamma\ell_\gamma$ process that contributes to level $y$ within this clade. 
  
  As a consequence, for $1\le i\le H$ and $n\ge N$ in the second case, the initial spindles of the clades associated with $(\bar c_n^{(i)},\bar d_n^{(i)})$ do not 
  reach level $y$. We are applying $\kappa_\perp$ to each excursion above the embedded local time $\gamma\ell_\gamma$, and the single excursion that contributes to level $y$ splits the total mass at level $y$ into ${\tt PDIP}(\alpha,0)$ proportions, by 
  Lemma \ref{lem:min_cld:skewer}\ref{item:MCSk:entrance}. As $y$ is fixed, this can, on the event $\{n\ge N\}$, be perfectly coupled to the ${\tt PDIP}(\alpha,0)$ 
  proportions that $\kappa_\perp$ assigns to the level-$y$ split of the mass contributed by $\widetilde{E}^{(i)}$. This completes the proof of assertion \ref{item:GCI:smallblocks} and thus proves the proposition.
\end{proof}

\subsection{The Feller property and the proof of Theorems \ref{thm:sp} and \ref{thm:dP-IP}}\label{sec:Feller}

Recall that Theorem \ref{thm:sp} claims that the extension of ${\tt SSIPE}(\alpha,\theta)$ to $(\IPHs,d_H\circ C)$ is a Feller process.  

\begin{proof}[Proof of Theorem \ref{thm:sp}] First recall from Corollary \ref{thm:semigroup*} that ${\tt SSIPE}(\alpha,\theta)$ has as its semigroup 
    $(\kappa^{\alpha,\theta}_y,\,y\ge 0)$ as defined in Definition \ref{def:kernel:sp2}. We now check the conditions of a Feller process as given in \cite[p.369]{Kallenberg}. Recall 
    that $(\mathcal{K},d_H)$ and $(\IPHs,d_H\circ C)$ are isometric, so we can work mainly in $(\mathcal{K},d_H)$. 

    To obtain the local compactness of $(\mathcal{K},d_H)$, we first note that it is well-known, e.g.\ \cite[Theorem 7.3.8]{BuraBuraIvan01}, 
    that the compactness of $[0,m]$ entails the compactness of the subspace $\mathcal{K}_{\le m}\subset\mathcal{K}$ of compact subsets of 
    $[0,m]$. But $\mathcal{K}=\bigcup_{m\ge 0}\mathcal{K}_{\le m}$, and every $K\in\mathcal{K}_{\le m}$ has $\mathcal{K}_{\le m+1}$ as a compact 
    neighborhood. This is what it means for $(\mathcal{K},d_H)$ to be locally compact.
  
    Fix $y\!>\!0$. To obtain (F1) of \cite[p.369]{Kallenberg}, we have to show that $K\mapsto\mathbb{E}_K[f(\beta^y,M^y)]$ is continuous 
    and vanishes at infinity, whenever $f\colon\mathcal{K}\rightarrow\mathbb{R}$ is continuous and vanishes at infinity. The continuity follows 
    directly from Proposition \ref{prop:gencontinitial}. Since for every $m\in[0,\infty)$, the set 
    $\mathcal{K}_{\le m}$ is compact, it suffices to show that for any sequence $(K_n,\, n\ge 1)$ in 
    $\mathcal{K}$ whose total masses tend to infinity, we have  $\mathbb{E}_K[f(\beta^y_n,M^y_n)]\rightarrow 0$. This follows because we can 
    couple the ${\tt BESQ}(2\theta)$ total masses so that total masses at level $y$ also tend to infinity.
  
    To obtain (F2), we have to show that $\mathbb{E}_K[f(\beta^y,M^y)]\rightarrow f(K)$ as $y\rightarrow 0$ for all $f$ that are continuous
    and vanish at infinity. But since $f$ is bounded, this follows from the path-continuity of $((\beta^y,M^y),y\ge 0)$ established in 
    Proposition \ref{prop:genIPE:cont0}, by dominated convergence.
\end{proof}

For $K\in\mathcal{K}_1$, we can consider $((\beta^y,M^y),\,y\ge 0)\sim{\tt SSIPE}^*_K(\alpha,\theta)$. By Propositions \ref{prop:sp}(iii) and \ref{prop:F:totalmass} and additivity of ${\tt BESQ}(0)$ and ${\tt BESQ}(2\theta)$, the total mass process $(M^y,\,y\ge0)$ is a ${\tt BESQ}_1(2\theta)$, and $M^y = \|\beta^y\|$ for all $y>0$ by Proposition \ref{prop:entrlaw}. Therefore, we can carry out the time-change of \eqref{eq:tau-beta} and normalize to obtain a de-Poissonized process.  Rather than formalize this as a process on $\IPHs$ with second coordinate identically 1, we project onto the first coordinate and consider it as a process on $\IPHos$, denoting its law by ${\tt PDIPE}_K(\alpha,\theta)$. 

We prove Theorem \ref{thm:dP-IP} by confirming that 
${\tt PDIPE}_K(\alpha,\theta)$ has all of the properties claimed in the theorem: it is a Feller process that enters $\IPAo$ $(d_H\circ C_1)$-continuously and instantly and never leaves, almost surely. 

\begin{proof}[Proof of Theorem \ref{thm:dP-IP}]
 We just note that ${\tt BESQ}_1(2\theta)$ is almost surely path-continuous with the property that the time-change 
  $\tau_{\boldsymbol{\beta}}$ of \eqref{eq:tau-beta} maps $[0,\infty)$ to $[0,\inf\{y\ge 0:$ $\beta^y=\varnothing\})$.
  By Proposition \ref{prop:entrlaw}, we have $(\beta^{y+z},\,z\ge 0)\sim{\tt SSIPE}_{\beta^y}(\alpha,\theta)$ with $\beta^y\in\IPA$ almost surely for any arbitrarily small $y>0$.  By Proposition \ref{prop:cts_imm}, ${\tt SSIPE}_{\beta^y}(\alpha,\theta)$ never leaves $\IPA\subset\mathcal{I}_H$. 
  Hence ${\tt PDIPE}_K(\alpha,\theta)$ enters $\IPAo$ instantly and never leaves, as claimed.
  
  By \cite[Theorem 19.25]{Kallenberg}(iv), the continuity in the initial state holds in path space $\mathbb{D}([0,\infty),\widehat{\mathcal{K}})$,
  where $\widehat{\mathcal{K}}$ is the one-point compactification of $\mathcal{K}$. By \cite[Theorem 2.6]{Helland78}, the time-change part of 
  de-Poissonization preserves the continuity in the initial state. Since sample paths are actually continuous and the time-changed processes 
  never hit $\varnothing$ a.s., normalization to unit mass also preserves (F1) and (F2). Hence, the Feller property is preserved under 
  de-Poissonization.

  The argument of \cite[Proof of Theorem 1.6]{FVAT} is easily transferred from the measure-valued context to the setting of interval partitions to establish 
  ${\tt PDIP}(\alpha,\theta)$ as stationary law of ${\tt PDIPE}(\alpha,\theta)$, using the pseudo-stationarity derived in Section \ref{sec:IPAT:pseudostat}. 
\end{proof}

Denote by ${\tt RANKED}\colon\mathcal{K}_1\rightarrow\overline{\nabla}_\infty$ the map that associates with $K\in\mathcal{K}_1$ the 
decreasing sequence of interval lengths in $\beta(K)$, where
\begin{equation*}
 \overline{\nabla}_\infty := \left\{ (x_i)_{i\ge1}\in [0,\infty)^{\mathbb{N}} \colon \sum\nolimits_{i\ge1} x_i\le 1\right\}
\end{equation*}
is equipped with the metric $d_{\infty}((x_i)_{i\ge1},(y_i)_{i\ge1}) = \sup_{i\ge1}|x_i-y_i|$.

\begin{corollary}\label{cor:RANKED2} Mapping ${\tt PDIPE}_K(\alpha,\theta)$, $K\in\mathcal{K}_1$, under ${\tt RANKED}$ yields a 
  $\overline{\nabla}_\infty$-valued Feller process.
\end{corollary}

\begin{lemma}\label{lm:RANKED} ${\tt RANKED}\colon\mathcal{K}_1\rightarrow\overline{\nabla}_\infty$ is continuous.
\end{lemma}
\begin{proof} Let $d_H(K_n,K)\rightarrow 0$ and $\varepsilon>0$. Then ${\tt RANKED}(K)$ has only finitely many blocks exceeding size 
  $\varepsilon/2$. By Lemma \ref{lm:dH}, we can find $n_0\ge 1$ such that for $n\ge n_0$, they have distinct corresponding blocks
  in $K_n$ with sizes within $\varepsilon/2$ of the limit and that, by compactness of $[0,1]$, no other blocks in $K_n$ have size exceeding 
  $\varepsilon$.
\end{proof}

\begin{proof}[Proof of Corollaries \ref{cor:RANKED2} and \ref{cor:RANKED}] By Lemma \ref{lm:RANKED}, the projected process inherits the path-continuity of 
  ${\tt PDIPE}_K(\alpha,\theta)$. Also, for any $K,K^\prime\!\in\!\mathcal{K}_1$ with ${\tt RANKED}(K)\!=\!{\tt RANKED}(K^\prime)$, we 
  can couple the clade constructions so that the projections to decreasing sequences are path-wise identical. Specifically, we refer to 
  \cite[Theorem 17.41]{Kechris} to relate the intensity measures and hence to couple the Poisson random measures $\mathbf{F}_K$ on $K$ and 
  $\mathbf{F}_{K^\prime}$ on $K^\prime$, which can be any compact subsets of $[0,1]$ with ${\rm Leb}(K)={\rm Leb}(K^\prime)$. 
  Therefore, Dynkin's criterion 
  \cite[Theorem 13.5]{Sharpe} applies and yields that the projected process is a path-continuous Borel right Markov process. In particular, (F2) 
  holds. To check (F1), take any convergent sequence in $\overline{\nabla}_\infty$. Consider canonical representatives in $\mathcal{K}_1$ that 
  have all blocks arranged in decreasing order from left to right. Then $d_\infty$-convergence implies the Hausdorff convergence of the representatives, hence
  (F1) for ${\tt PDIPE}(\alpha,\theta)$ applied to functions of the form $f\circ{\tt RANKED}$ establishes (F1) for the projection.

  This proves Corollary \ref{cor:RANKED2}. Corollary \ref{cor:RANKED} follows since $(\mathcal{K}_1,d_H)$ and $(\overline{I}_{H,1},d_H\circ C_1)$ are isometric.
\end{proof}

In \cite{Paper1-3} we use this corollary to show that this projection is Petrov's two-parameter Poisson--Dirichlet diffusion.

\subsection{Ray--Knight theorem for marked reflected stable L\'evy processes}

In this section we prove the second Ray--Knight theorem stated in Theorem \ref{thm:rayknight}(i). Recall from \eqref{eq:Fdust} that the main ingredient for our construction of
${\tt SSIPE}^*_{[0,z]}(\alpha,0)$ is ${\tt F}_{[0,z]}\sim{\tt PRM}\big({\rm Leb}|_{[0,z]}\otimes\frac{1}{2\alpha}\nu_{\perp\rm cld}^{(\alpha)}\big)$. Then Proposition 
\ref{prop:entrlaw} entails that
\[
\skewerP\left(\Concat_{\text{points }(s,N_s)\text{ of }\mathbf{F}_{[0,z]}}N_s,\Concat_{\text{points }(s,N_s)\text{ of }\mathbf{F}_{[0,z]}}\xi(N_s)\right)\sim{\tt SSIPE}_{[0,z]}^*(\alpha,0).
\]
Theorem \ref{thm:rayknight}(i) claims an alternative construction of ${\tt SSIPE}_{[0,z]}^*(\alpha,0)$. As established in the proof of Theorem \ref{thm:rayknight}(ii) in 
Section \ref{sec:rk1}, the main ingredient is $\mathbf{N}\sim\PRM[\Leb\otimes\mBxcA]$. Specifically, Theorem \ref{thm:rayknight}(i) claims that for
\[
\beta^y := \skewer(y,\bN|_{[0,T_{-z/2\alpha}]},\bX- \underline{\bX}), ,\qquad  y> 0,
\]
we have that $\beta^0_*=(\varnothing,z)$ and $\beta^y_* = \big(\beta^y, \|\beta^y\|\big)$, $y>0$, yields an ${\tt SSIPE}_{[0,z]}^*(\alpha,0)$.

\begin{proof}[Proof of Theorem \ref{thm:rayknight}(i)] As we observed in the proof of Theorem \ref{thm:rayknight}(ii), with reference to \eqref{eq:imm_PPP}, the marked 
  excursions of $\bX=\xi(\bN)$ above its infimum process $\underline{\bX}$ form a $\PRM(\Leb\otimes\umCladeA)$. Stopped at $T_{-z/2\alpha}$, linear scaling maps this to a
  ${\tt PRM}\big({\rm Leb}|_{[0,z]}\otimes\frac{1}{2\alpha}\nu_{\perp\rm cld}^{(\alpha)}\big)$, as required to align with the construction of ${\tt SSIPE}_{[0,z]}^*(\alpha,0)$
  of Proposition \ref{prop:entrlaw}. By Proposition \ref{prop:F:totalmass}, this also entails the claim that the total mass evolves as ${\tt BESQ}_z(0)$.    
\end{proof}

\section{Identification with the Rivera-Lopez--Rizzolo diffusive limit}\label{sec:infgen}

The proof of Theorem \ref{thm:twopargen} is based on several auxiliary results. We begin with an observation that will allow us (inductively) to reduce the identification of the generator to functions $m_\sigma^\circ$ associated with compositions $\sigma$ that have no consecutive 1s. To this end, we denote by $\mathbf{1}_k=(1,\ldots,1)\in\mathcal{C}_k$ the composition of $k$ into $k$ parts.

\begin{lemma}\label{lm:1k} For all $k\ge 2$, we have
  \[
  m^\circ_{\mathbf{1}_k}=\frac{1}{k!}-\sum_{\rho\in\mathcal{C}_k\setminus\{\mathbf{1}_k\}}\left( \prod_{i=1}^{\ell(\rho)}\frac{1}{\rho_i!}\right)m_\rho^\circ.
  \]
\end{lemma}
\begin{proof} This is just a rearrangement of the observation that for $\beta\in\mathcal{I}_{H,1}$
  \begin{align*}
  1=\left(\sum_{V\in\beta}{\rm Leb}(V)\right)^k&=\sum_{V_1,\ldots,V_k\in\beta}\prod_{j=1}^k{\rm Leb}(V_j)\\
    &=\sum_{\rho\in\mathcal{C}_k}\sum_{\overset{\scriptstyle U_1,\ldots,U_{\ell(\rho)}\in\beta}{\rm strictly\,increasing}} {k\choose \rho_1\;\cdots\;\rho_{\ell(\rho)}}\prod_{i=1}^{\ell(\rho)} \left({\rm Leb}(U_i)\right)^{\rho_i}.\\
  &=\sum_{\rho\in\mathcal{C}_k}{k\choose \rho_1\;\cdots\;\rho_{\ell(\rho)}}m_\rho^\circ(\beta).
  \end{align*}
  noting that the multinomial coefficient equals $k!$ when $\rho=\mathbf{1}_k$.
\end{proof}

Denote by 
\[
\widetilde{\mathcal{C}}_n=\left\{\sigma\in\mathcal{C}_n\colon\sigma_{r-1}+\sigma_r\ge 3\mbox{ for all }2\le r\le\ell(\sigma)\right\}
\]
the set of compositions of $n\ge 0$ with no two consecutive 1s, and set $\widetilde{\mathcal{C}}=\bigcup_{n\ge 0}\widetilde{\mathcal{C}}_n$. For $\sigma\in\widetilde{\mathcal{C}}$, the
calculation of the generator is similar to generator calculations for the ranked process in \cite[Section 2-3]{Paper1-3}, and we recall some relevant results, beginning with the 
identification of Wright--Fisher processes taking values in the simplex $\Delta_d:=\big\{\mathbf{w}=(w_1,\ldots,w_d)\in[0,1]^d\colon\sum_{i\in[d]}w_i=1\big\}$ with generator
\[
\mathcal{A}_{\tt WF}^{\mathbf{r}}=2\sum_{i\in[d]}w_i\frac{\partial^2}{\partial w_i^2}-2\sum_{i,j\in[d]}w_iw_j\frac{\partial^2}{\partial w_i\partial w_j}-2\sum_{i\in[d]}(r_+w_i-r_i)\frac{\partial}{\partial w_i},
\]
where $\mathbf{r}=(r_1,\ldots,r_d)$ is a vector of real parameters and $r_+=\sum_{i\in[d]}r_i$. Wright--Fisher processes have infinite lifetime if $r_i\ge 0$ for all $i\in[d]$.
Otherwise, the lifetime of the process is the first time the $i^{\rm th}$ coordinate vanishes for an $i\in[d]$ with $r_i<0$. We denote by ${\tt WF}_\mathbf{b}(\mathbf{r})$ the
distribution of a Wright--Fisher process with parameter vector $\mathbf{r}$ and initial state $\mathbf{b}$.
\begin{lemma}[Special case of Lemma 2.2 of \cite{Paper1-3}]\label{lm:WFgen} The domain of $\mathcal{A}_{\tt WF}^{\mathbf{r}}$ includes all functions $g_\mathbf{q}$ for $\mathbf{q}\in\big(\mathbb{N}\cup\{0\}\big)^d$, given by
  \[
  g_\mathbf{q}(\mathbf{w})=\prod_{i\in[d]}w_i^{q_i},\qquad\begin{array}{l}\mbox{if }r_i\ge 0\mbox{ for all }i\in[d]\mbox{ with }q_i=1,\\[0.1cm]
                                                                                                                                            \mbox{and }r_i\in\mathbb{R}\mbox{ for all }i\in[d]\mbox{ with }q_i\in\{0,2,3,\ldots\}.\end{array}
  \]
\end{lemma}
Fix $\beta\in\mathcal{I}_{H,1}$ and let  
$\big(\widebar{\beta}^u,\,u\ge 0\big)\sim{\tt PDIPE}_\beta(\alpha,\theta)$ be obtained by de-Poissonizing the process $\big(\beta^y,\,y\ge 0\big)$ constructed in Definition \ref{def:IPPAT}. This gives access to continuous block evolutions $\big(W_u^{(V)},u\ge 0\big)$ starting from ${\rm Leb}(V)$, $V\in\beta$, and to evolutions $\big(\widebar{W}_u^{(0)},\,u\ge 0\big)$ and 
$\big(\widebar{W}_u^{(V)},\,u\ge 0\big)$, $V\in\beta$, starting from $0$ that capture the aggregate mass of other blocks arising, respectively, from $\cev{\mathbf{F}}_\theta$ and 
$\mathbf{N}_V$, $V\in\beta$. We also write $X_u^{(V)}=W_u^{(V)}+\widebar{W}_u^{(V)}$. More precisely, denote by $\mathbf{f}_V$ the left-most spindle of $\mathbf{N}_V$ and recall the de-Poissonization notation of Definition \ref{def:depoiss}. We then set
\begin{align*}
W_u^{(V)}=\left\|\beta^{\tau_{\boldsymbol{\beta}}(u)}\right\|^{-1}\mathbf{f}_V(\tau_{\boldsymbol{\beta}}(u)),\quad
X_u^{(V)}=&\left\|\beta^{\tau_{\boldsymbol{\beta}}(u)}\right\|^{-1}\Big\|\skewer\big(\tau_{\boldsymbol{\beta}}(u),\mathbf{N}_V\big)\Big\|,\\
\widebar{W}_u^{(V)}=X_u^{(V)}-W_u^{(V)},\quad
\widebar{W}_u^{(0)}=X^{(0)}_u=&\left\|\beta^{\tau_{\boldsymbol{\beta}}(u)}\right\|^{-1}\Big\|\fskewer\big(\tau_{\boldsymbol{\beta}}(u),\cev{\mathbf{F}}_\theta\big)\Big\|.
\end{align*}
The following lemma is a variant of \cite[Proposition 2.6]{Paper1-3}. The important variation for the present requirements is that the setting of \cite{Paper1-3}, focussing entirely on ranked block sizes,  
allowed to only consider functions $g_\mathbf{q}$ with powers $q_i\neq 1$ for all $i\in[d]$. The inclusion of the case $q_i=1$ means the following generalisation is needed.

\begin{lemma}\label{lm:WF} Let $V_1,\ldots,V_k\in\beta$ be a strictly increasing sequence of intervals. For $0\le i\le k+1$, let $\widebar{W}_u^{(V_{i},V_{i+1})}=\widebar{W}_u^{(V_{i})}+\sum_{V\in\beta\colon V_{i}<V<V_{i+1}}X_u^{(V)}$, with the conventions $V_0=0$ and $V_{k+1}=1$. Then
   \[
  \left(\widebar{W}_u^{(0,V_1)},W_u^{(V_1)},\widebar{W}_u^{(V_{1},V_{2})},\ldots, W_u^{(V_k)},\widebar{W}_u^{(V_k,1)}\right)\sim{\tt WF}_{\mathbf{b}}(\mathbf{r}),
   \]
    where $b_{2i}={\rm Leb}(V_i)$, $i\in[k]$, $b_{2i+1}=\sum_{V\in\beta\colon V_i<V<V_{i+1}}{\rm Leb}(V)$, $0\le i\le k$, and $r_1=\theta$, $r_{2i}=-\alpha$, $r_{2i+1}=\alpha$, $i\in[k]$.

  Similarly, $\big(X^{(V_1^\prime)}_u,\ldots,X^{(V_p^\prime)}_u,1-\sum_{i\in[p]}X_u^{(i)}\big)\sim{\tt WF}_{\mathbf{b}^\prime}(\mathbf{r}^\prime)$, for distinct 
  $V_1^\prime,\ldots,V_p^\prime\in\beta\cup\{0\}$, where $(b_i^\prime,r_i^\prime)=(0,\theta)$ if $V_i^\prime=0$ and $(b_i^\prime,r_i^\prime)=({\rm Leb}(V_i^\prime),0)$ if $V_i^\prime\in\beta$, $i\in[p]$.
\end{lemma}
\begin{proof} Like \cite[Proposition 2.6]{Paper1-3}, this can be seen using Pal's \cite{Pal11,Pal13} construction of Wright--Fisher diffusions from independent squared Bessel
  processes. In the present setting, we note from Proposition \ref{prop:sp}\ref{item:SSIP:mass} that $\mathbf{g}_1:=\big(\big\|\fskewer(y,\cev{\mathbf{F}}_\theta)\big\|,\,y\ge 0\big)$ $\sim{\tt BESQ}_{b_1}(2\theta)$ and $\big(\big\|\skewer(y,\mathbf{N}_{V})\big\|,\,y\ge 0)\sim{\tt BESQ}_{{\rm Leb}(V)}(0)$ for all $V\in\beta$. We see from 
  Definition \ref{def:IPPA} that $\mathbf{g}_{2i}:=\mathbf{f}_{V_i}\sim{\tt BESQ}_{b_{2i}}(-2\alpha)$. Furthermore, as has been observed in \cite[Corollary 3.16]{Paper1-2}, the independence of 
  $\mathbf{f}$ and $\mathbf{N}$ in Definition \ref{def:IPPA} entails that $\big(\big\|\skewer\big(y,\mathbf{N}_{V_i}|_{(0,\infty)\times\mathcal{E}},\xi(\mathbf{N}_{V_i})\big)\big\|,\,0\le y\le \zeta(\mathbf{f}_{V_i})\big)$ is distributed like an independent ${\tt BESQ}_0(2\alpha)$ process, which is then stopped at $\zeta(\mathbf{f}_{V_i})$. We extend
  our probability space to support the full ${\tt BESQ}_0(2\alpha)$ by extending it independently beyond $\zeta(\mathbf{f}_{V_i})$, for each $i\in[k]$.
  By the additivity of ${\tt BESQ}$-processes, adding the independent ${\tt BESQ}_0(2\alpha)$ and ${\tt BESQ}_{{\rm Leb}(V)}(0)$, $V_{i}<V<V_{i+1}$, yields
  a process $\mathbf{g}_{2i+1}\sim{\tt BESQ}_{b_{2i+1}}(2\alpha)$, for each $i\in[k]$, so that for $0\le y\le \zeta(\mathbf{f}_{V_i})$,
  \[
  \mathbf{g}_{2i+1}(y)=\big\|\skewer\big(y,\mathbf{N}_{V_i}|_{(0,\infty)\times\mathcal{E}},\xi(\mathbf{N}_{V_i})\big)\big\|+\sum_{V\colon V_i<V<V_{i+1}}\big\|\skewer(y,\mathbf{N}_{V})\big\|.
  \]
  By construction, $\mathbf{g}_i$, $i\in[2k+1]$, are independent ${\tt BESQ}_{b_i}(2r_i)$,
  $i\in[2k+1]$. In this setting, let $\mathbf{g}=\sum_{i\in[2k+1]}\mathbf{g}_i$. As noted in \cite[Section 2.3]{Paper1-3} in the relevant generality, by reference to \cite[Proposition 12]{Pal11} and \cite[Theorem 4]{Pal13} for non-negative and non-positive parameters, respectively, this entails
  \[
  \Big(\Big(\big(\mathbf{g}(\tau_\mathbf{g}(u))\big)^{-1}\mathbf{g}_i(\tau_\mathbf{g}(u)),\,i\in[2k+1]\Big),\,0\le u<T\Big)\sim{\tt WF}_\mathbf{b}(\mathbf{r}),
  \]
  where $\tau_\mathbf{g}(u)=\inf\big\{z\ge 0\colon\int_0^z(\mathbf{g}(y))^{-1}dy>u\big\}$ and $T$ is the first time in the time-changed process that an even component 
  vanishes, i.e.\ such that $\tau_\mathbf{g}(T)=\min\big\{\zeta(\mathbf{f}_{V_i}),\,i\in[k])\big\}$. This completes the proof since, by construction, 
  $\mathbf{g}(y)=\|\beta^y\|$, $0\le y\le\min\{\zeta(\mathbf{f}_{V_i}),\,i\in[k]\}$. The second claim is proved similarly.
\end{proof}

\begin{proposition}\label{prop:nocons} Let $\beta\in\mathcal{I}_{H,1}$, $(\widebar{\beta}^u,\,u\ge 0)\sim{\tt PDIPE}_\beta(\alpha,\theta)$ and $\sigma\in\widetilde{\mathcal{C}}$. Then 
  \begin{equation}\label{eq:genclaim}
  \lim_{u\downarrow 0}\frac{\mathbb{E}[m_\sigma^\circ(\widebar{\beta}^u)]-m_\sigma^\circ(\beta)}{u}=2\mathcal{A}_{\alpha,\theta}m_\sigma^\circ(\beta).
  \end{equation}
  Moreover, this convergence holds in $\mathbf{L}^2({\tt PDIP}(\alpha,\theta))$.
\end{proposition}
\begin{proof} This proof is a refinement of the proof of \cite[Proposition 3.3]{Paper1-3}. Fix $\beta\in\mathcal{I}_{H,1}$ and $\sigma\in\widetilde{\mathcal{C}}_n$ with $\ell(\sigma)=\ell$ and $k:=\#\{j\in[\ell]\colon\sigma_j\ge 2\}$. Then we can 
  bound the expectation in \eqref{eq:genclaim} below by only considering blocks of $\widebar{\beta}^u$ that arise from $\mathbf{f}_U$, $U\in\beta$, for powers 
  $\sigma_j\ge 2$, but all blocks for power $\sigma_j=1$, in the following sense, 
  \begin{align}\label{eq:lowerb}
  &\mathbb{E}[m_\sigma^\circ(\widebar{\beta}^u)]
  =\mathbb{E}\Bigg[\!\!\sum_{\overset{\scriptstyle\widetilde{U}_j\in\widebar{\beta}^u,\,j\in[\ell]}{\rm strictly\,increasing}}\!\prod_{j\in[\ell]}\!({\rm Leb}(\widetilde{U}_j))^{\sigma_j}\!\Bigg]\!\\
  &=\mathbb{E}\Bigg[\sum_{\overset{\scriptstyle\widetilde{U}_j\in\widebar{\beta}^u,\,j\in[\ell]\colon\sigma_j\ge 2}{\rm strictly\,increasing}}\prod_{j\in[\ell]\colon\sigma_j\ge 2}\!({\rm Leb}(\widetilde{U}_j))^{\sigma_j}\prod_{j\in[\ell]\colon\sigma_j=1}\Bigg(\sum_{\overset{\scriptstyle\widetilde{U}_j\in\widebar{\beta}^u}{\widetilde{U}_{j-1}<\widetilde{U}_j<\widetilde{U}_{j+1}}}{\rm Leb}(\widetilde{U}_j)\Bigg)\!\Bigg]\!\nonumber\\
  &\ge\!\sum_{\overset{\scriptstyle U_j\in\beta,\,j\in[\ell]\colon\sigma_j\ge 2}{\rm strictly\,increasing}}\!\!\!\mathbb{E}\Bigg[\prod_{j\in[\ell]}\!(Z_u^{(j)})^{\sigma_j}\!\Bigg]\!,\nonumber
  \end{align}
  where in the notation of Lemma \ref{lm:WF} with $\{V_i,\,i\in[k]\}=\{U_j,\,j\in[\ell]\colon\sigma_j\ge 2\}$, we set $Z^{(j)}_u=W^{(U_j)}_u$ if $\sigma_j=2$ and 
  $Z^{(j)}_u=\widebar{W}_u^{(U_{j-1},U_{j+1})}$ if $\sigma_j=1$, using conventions $U_0=0$ and $U_{\ell+1}=1$. Indeed, $\sigma\in\widetilde{\mathcal{C}}$ entails that for $j$ with 
  $\sigma_j=1$ the intervals $U_{j-1}$ and $U_{j+1}$ are adjacent members of $(V_i,\,i\in[k])$. 
We write $a_j=Z^{(j)}_0$, $j\in[\ell]$, and
  note $(a_j,\,j\in[\ell])$ is made up of some but in general not all entries of $(b_i,\,i\in[2k+1])$, and $g_\sigma(\mathbf{a})=g_\mathbf{q}(\mathbf{b})$ for a vector $\mathbf{q}$ that has a zero component
  inserted into $\sigma$ wherever the first or two adjacent or the last entry of $\sigma$ are at least 2. 

  We continue with a fixed selection of strictly increasing $U_j\in\beta$, $j\in[k]\colon\sigma_j\ge 2$ and note that the Wright--Fisher parameters in Lemma \ref{lm:WF} are
  such that any components $Z^{(j)}_u$ when $\sigma_j=1$ have parameter $\theta\ge 0$ if $j=1$ and $\alpha\in(0,1)$ if $j\ge 2$. Hence, we can apply Lemma \ref{lm:WFgen}
  and find
  \begin{align}
  &\lim_{u\downarrow 0}\frac{\mathbb{E}\left[\prod_{j=1}^{\ell}(Z_u^{(j)})^{\sigma_j}\right]-\prod_{j=1}^{\ell}a_j^{\sigma_j}}{u}\label{eq:WFgencalc}\\
  &=2\sum_{j\in[\ell]\colon\sigma_j\ge 2}\sigma_j(\sigma_j\!-\!1)g_{\sigma-\square_j}(\mathbf{a})-2\sum_{i,j\in[\ell]\colon i\neq j}\sigma_i\sigma_jg_{\sigma}(\mathbf{a})
           -2\sum_{j\in[\ell]}\sigma_j(\sigma_j\!-\!1)g_{\sigma}(\mathbf{a})\nonumber\\
  &\quad     -2\theta\sum_{j\in[\ell]}g_\sigma(\mathbf{a})-2\alpha\sum_{j\in[\ell]\colon\sigma_j\ge 2}\sigma_jg_{\sigma-\square_j}(\mathbf{a})
       +2\sum_{j\in[\ell]\colon\sigma_j=1}\eta_jg_{\sigma\ominus\square_j}(\mathbf{a})\nonumber\\
  &=-2|\sigma|(|\sigma|\!-\!1\!+\!\theta)g_\sigma(\mathbf{a})+2\sum_{j\colon\sigma_j\ge 2}\sigma_j(\sigma_j\!-\!1\!-\!\alpha)g_{\sigma-\square_j}(\mathbf{a})
       +2\sum_{j\colon\sigma_j=1}\eta_jg_{\sigma\ominus\square_j}(\mathbf{a}).\nonumber
  \end{align}
  To conclude that the RHS of \eqref{eq:genclaim} is a lower bound for the LHS, we want to sum over $(U_j)$ as on the RHS of \eqref{eq:lowerb}. Indeed, subject to checking the 
  conditions of the dominated convergence theorem, this yields the claimed lower bound because \eqref{eq:lowerb} is an equality for $u=0$ and for each $j\in[\ell]$, summing 
  $g_{\sigma-\square_j}(\mathbf{a})$ or $g_{\sigma\ominus\square_j}(\mathbf{a})$, respectively, over all $\mathbf{a}$ that arise from summing over $(U_j)$ similarly yields
  $m_{\sigma-\square_j}^\circ(\beta)$ and $m_{\sigma\ominus\square_j}^\circ(\beta)$, as required. The domination condition can be checked as in 
  \cite[Lemmas 3.5 and 3.8]{Paper1-3} using It\^o's formula and dropping negative terms to first show that $\mathbb{E}\big[\prod_{j\in[\ell]\colon\sigma_j\ge 2}Z_s^{(j)}\big]\le\prod_{j\in[\ell]\colon \sigma_j\ge 2}a_j$ for all $s\ge 0$, then similarly for all $u>0$
  \begin{equation}\label{eq:dom}
  \left|\frac{\mathbb{E}\big[\prod_{j\in[\ell]}(Z_u^{(j)})^{\sigma_j}\big]-\prod_{j\in[\ell]}a_j^{\sigma_j}}{u}\right|\le 2|\sigma|\big(2|\sigma|+2+\alpha+\theta\big)\prod_{j\in[\ell]\colon\sigma_j\ge 2}a_j.
  \end{equation}
  As the final product is summable, this completes the proof of the lower bound.

  For the upper bound, we denote by $\beta_n$ the set of $n\ge 0$ longest intervals of $\beta$, breaking any ties by choosing from left to right, to be definite. For each $n$, we
  distinguish the main contribution from $\mathbf{f}_U$, $U\in\beta_n$, for powers $\sigma_j\ge 2$, still allowing all blocks for power $\sigma_j=1$, and the remainder that consists
  of several parts that cover all other choices of $\widetilde{U}_j\in\widebar{\beta}^u$, $j\in[\ell]$. The argument for the lower bound deals with the main contribution and 
  will yield the RHS of \eqref{eq:genclaim}, 
  and we show that the remainder vanishes, as $n\rightarrow\infty$. Specifically, we bound $\mathbb{E}[m_\sigma^\circ(\widebar{\beta}^u)]$ above by
  \begin{align}&\sum_{\overset{\scriptstyle (U_j,\,j\in[\ell]\colon\sigma_j\ge 2)\in\beta_n^k}{\rm strictly\,increasing}}\!\mathbb{E}\Bigg[\prod_{j\in[\ell]}\!(Z_u^{(j)})^{\sigma_j}\!\Bigg]
      +\sum_{\overset{\scriptstyle (U_j,\,j\in[\ell]\colon\sigma_j\ge 2)\in\beta^k\setminus\beta_n^k}{\rm strictly\,increasing}}\!\mathbb{E}\Bigg[\prod_{j\in[\ell]}\!(Z_u^{(j)})^{\sigma_j}\!\Bigg]\nonumber\\             
      &\qquad\quad+\sum_{r\in[\ell]\colon\sigma_r\ge 2}\;\sum_{U_j\in\beta\cup\{0\},\,j\in[\ell]}\mathbb{E}\Bigg[\big(\widebar{W}_u^{(U_r)}\big)^{\sigma_r}\prod_{j\in[\ell]\setminus\{r\}}\big(X_u^{(U_j)}\big)^{\sigma_j}\Bigg].\label{eq:upperbound}
  \end{align}
  Note that the remainder term in the second line of \eqref{eq:upperbound} covers the choice of blocks where, in the notation of \eqref{eq:lowerb}, $\widetilde{U}_r$ is a block of $\widebar{\beta}^u$
  that is not arising from $\mathbf{f}_U$, $U\in\beta$, but from a block included in $\widebar{W}_u^{(U_r)}$ for some $U_r\in\beta\cup\{0\}$. Note also that
  for $j\neq r$, we allow all choices of blocks without constraining the order. 
  Indeed, the asymptotics as $u\downarrow 0$ for the first term, a finite sum, with each member compensated and divided by $u$ as in \eqref{eq:WFgencalc}, were established in the calculation of \eqref{eq:WFgencalc}, and then letting 
  $n\rightarrow\infty$ yields the RHS of \eqref{eq:genclaim}. In the compensated second term divided by $u$, we use \eqref{eq:dom} to upper bound this part of the remainder by 
  \begin{equation}\label{eq:upterm2}
  2|\sigma|\big(2|\sigma|+2+\alpha+\theta\big)\sum_{\overset{\scriptstyle (U_j,\,j\in[\ell]\colon\sigma_j\ge 2)\in\beta^k\setminus\beta_n^k}{\rm strictly\,increasing}}\;\prod_{j\in[\ell]\colon\sigma_j\ge 2}{\rm Leb}(U_j),
  \end{equation}
  and this vanishes as $n\rightarrow\infty$. Finally, the last term of \eqref{eq:upperbound} is easily bounded above by 
  \begin{equation}\label{eq:upterm3}
  k\sum_{U\in\beta_n\cup\{0\}}\mathbb{E}\Big[\big(\widebar{W}_u^{(U)}\big)^2\Big]+k\sum_{U\in\beta\setminus\beta_n}\mathbb{E}\Big[\big(\widebar{W}_u^{(U)}\big)^2\Big].
  \end{equation}
  For this term, when divided by $u$, \cite[Lemma 3.6]{Paper1-3} ensures that the first, finite, sum still tends to 0 as $u\downarrow 0$, and the second sum is bounded by 
  $4(2+\theta)\sum_{U\in\beta\setminus\beta_n}{\rm Leb}(U)$ and hence vanishes as $n\rightarrow\infty$. This shows that the LHS of \eqref{eq:genclaim} is also bounded 
  above by the RHS.

  Finally, we note that the dominations identified in \eqref{eq:upterm3}--\eqref{eq:upterm2}, in the case $n=0$, further yield
  \[
  \sup_{\overset{\scriptstyle u\in(0,1]}{\beta\in\mathcal{I}_{H,1}}}\left|\frac{\mathbb{E}[m_\sigma^\circ(\widebar{\beta}^u)]-m_\sigma^\circ(\beta)}{u}\right|\le 
  2|\sigma|(2|\sigma|+2+\alpha+\theta)+4k(2+\theta)+\sup_{u\in(0,1]}\frac{\mathbb{E}[(\overline{W}_u^{(0)})^2]}{u},
  \]
  which is finite, again by \cite[Lemma 3.6]{Paper1-3}, and the dominated convergence theorem establishes the ${\bf L}^2$-claim.
\end{proof}

For general $\sigma\in\mathcal{C}$, we now adapt the proof of Proposition \ref{prop:nocons} to establish an auxiliary result, in which consecutive 1s are included, but not treated
in the way needed for $m^\circ_\sigma$. Fix $\beta\in\mathcal{I}_{H,1}$ and $\sigma\in\mathcal{C}$ with $\ell=\ell(\sigma)$.
Then the composition $\sigma$ is of the form $\mathbf{1}_{\ell_0}$ followed by $\sigma_j\mathbf{1}_{\ell_j}$ for each $j\in[\ell]$ with $\#\sigma_j\ge 2$, where
$\ell_j\in\mathbb{N}\cup\{0\}$ counts the length of the run of 1s in $\sigma$ to the right of $\sigma_j$. Let 
\[
m_\sigma^*(\beta)=\sum_{\overset{\scriptstyle U_j\in\beta,\,j\in[\ell]\colon\sigma_j\ge 2}{\rm strictly\,increasing}}
\left\|\beta_0\right\|^{\ell_0}\prod_{j\in[\ell]\colon\sigma_j\ge 2}\left(\big({\rm Leb}(U_j)\big)^{\sigma_j}\left\|\beta_j\right\|^{\ell_j}\right),
\]
where we abuse notation and write $(\beta_j,\,j\in[\ell]\colon\sigma_j\ge 2)$ for the decomposition of $\beta$ at the blocks $(U_j)$ so that
\begin{equation}\label{eq:betadecomp}
\beta=\beta_0\concat\Concat_{j\in[\ell]\colon\sigma_j\ge 2}\left(U_j\concat\beta_j\right).
\end{equation}
Then $m^*_\sigma=m^\circ_\sigma$ if $\sigma\in\widetilde{\mathcal{C}}$. 

\begin{corollary}\label{lm:mstar} Let $\beta\in\mathcal{I}_{H,1}$, $(\widebar{\beta}^u,\,u\ge 0)\sim{\tt PDIPE}_\beta(\alpha,\theta)$ and $\sigma\in\mathcal{C}$. Then
  \begin{equation}\label{eq:lm:genclaim}
  \lim_{u\downarrow 0}\frac{\mathbb{E}[m_\sigma^*(\widebar{\beta}^u)]-m_\sigma^*(\beta)}{u}=2\mathcal{A}_{\alpha,\theta}m_\sigma^*(\beta),
  \end{equation}
  where
  \begin{align}
  \mathcal{A}_{\alpha,\theta}m_\sigma^*=&-|\sigma|(|\sigma|-1+\theta)m_\sigma^*\ +\ \ell_0(\ell_0-1+\theta)m^*_{\sigma\ominus\square_1}\label{eq:mstargen}\\
  &+\sum_{j\colon\sigma_j\ge 2}\left(\sigma_j(\sigma_j-1-\alpha)m_{\sigma-\square_j}^*+\ell_j(\ell_j-1+\alpha)m_{\sigma\ominus\square_{j+1}}^*\right).\nonumber
  \end{align}
  Furthermore, the convergence in \eqref{eq:lm:genclaim} holds in ${\bf L}^2({\tt PDIP}(\alpha,\theta))$.
\end{corollary}  
\begin{proof} To simplify notation, we write $A=\{j\in[\ell]\colon\sigma_j\ge 2\}$ and $A_0=\{0\}\cup A$. We first check the claimed form \eqref{eq:mstargen} of the generator when applied to
  $m_\sigma^*$ and note that as in Lemma \ref{lm:1k}, we can write for $j\in A_0$ and $\beta_j$ as in \eqref{eq:betadecomp}
  \[
  \left\|\beta_j\right\|^{\ell_j}=\sum_{\rho^{(j)}\in\mathcal{C}_{\ell_j}}{\ell_j\choose\rho_1^{(j)}\;\cdots\;\rho_{\ell(\rho^{(j)})}^{(j)}}  
  \sum_{\overset{\scriptstyle V_1,\ldots,V_{\ell(\rho^{(j)})}\in\beta_j}{\rm strictly\,increasing}}\prod_{i=1}^{\ell(\rho^{(j)})}({\rm Leb}(V_i))^{\rho_i^{(j)}}
  \]
  Writing shorthand ${\ell_j\choose\rho^{(j)}}$ for the multinomial coefficient, this entails 
  \begin{equation}\label{eq:mstar}
  m_\sigma^*=\sum_{(\rho^{(r)},\,r\in A_0)\in\prod_{r\in A_0}\mathcal{C}_{\ell_r}}\left(\prod_{r\in A_0} {\ell_r\choose\rho^{(r)}}\right)
                       m^\circ_{\rho^{(0)}\concat\Concat_{r\in A}(\sigma_r\concat\rho^{(r)})},
  \end{equation}
  where the symbols $\concat$ and $\Concat$ are used here to denote the concatenation of entries and compositions, i.e.\ forming longer compositions by listing all parts in a
  single vector of length $\ell(\rho^{(0)})+\sum_{r\in A}(1+\ell(\rho^{(r)}))$. In particular, $m_\sigma^*$ is in the domain of $\mathcal{A}_{\alpha,\theta}$. Furthermore, 
  applying the generator to this linear combination easily yields the claimed $m_\sigma^*$- and $m_{\sigma-\square_j}^*$-terms, $j\in A$, of the RHS of \eqref{eq:mstargen}. To identify the other terms, we note that
  $m_{\sigma\ominus\square_{j+1}}^*$ for $j\in A_0$ can be written as in \eqref{eq:mstar}, but with $\ell_j$ replaced by $\ell_{j}-1$. There are $2\ell(\widetilde{\rho}^{(j)})+1$ ways to
  obtain a given $\widetilde{\rho}^{(j)}\in\mathcal{C}_{\ell_j-1}$ from members $\rho^{(j)}\in\mathcal{C}_{\ell_j}$: the $i^{\rm th}$ part could have been one greater in
  $\rho^{(j)}\in\mathcal{C}_{\ell_j}$ than in $\widetilde{\rho}^{(j)}\in\mathcal{C}_{\ell_j-1}$, for any one $i\in[\ell(\widetilde{\rho}^{(j)})]$, or there could have been an 
  additional part in $\rho^{(j)}\in\mathcal{C}_{\ell_j}$, inserted into $\widetilde{\rho}^{(j)}\in\mathcal{C}_{\ell_j-1}$ in one of $\ell(\widetilde{\rho}^{(j)})+1$ possible places. 
  In these two cases, we observe
  \[
  {\ell_j\choose\rho^{(j)}}\rho_i^{(j)}=\ell_j{\ell_j-1\choose\rho^{(j)}-\square_i}\quad\mbox{respectively}\quad 
  {\ell_j\choose\rho^{(j)}}=\ell_j{\ell_j-1\choose\rho^{(j)}\ominus\square_i}
  \]
  Collecting the coefficients from the respective terms in $\cA_{\alpha,\theta}m^\circ_{\rho^{(0)}\concat\Concat_{r\in A}\sigma_r\concat\rho^{(r)}}$, we obtain a coefficient
  of $m^*_{\sigma\ominus\square_{j+1}}$ of
  \[
  \ell_j\left(\sum_{i\in[\ell(\widetilde{\rho}^{(j)})]}(\widetilde{\rho}_i^{(j)}-\alpha)+\eta_j+\ell(\widetilde{\rho}^{(j)})\alpha\right)=\ell_j(\ell_j-1+\eta_j), 
  \] 
  as claimed in \eqref{eq:mstargen}. For the main claim \eqref{eq:lm:genclaim}, we use the notation of the proof of Proposition \ref{prop:nocons} in that we similarly set 
  $Z_u^{(j)}=W_u^{(U_j)}$ for $j\in A$, but we now use the same $Z_u^{(r)}$ for each run of consecutive $\sigma_r=1$. To be precise, we write $Z_u^{(j+r)}=\widebar{Z}^{(j)}_u=\widebar{W}_u^{(U_j,U_{j+\ell_j+1})}$, $r\in[\ell_j]$, $j\in A_0$, with the convention that $U_0=0$ and $U_{\ell+1}=1$. We also write $a_j={\rm Leb}(U_j)$ 
  for $j\in A$ and $a_{j+r}=\|\beta_j\|$ for all $r\in[\ell_j]$ and all $j\in A_0$.
  Then 
 \eqref{eq:WFgencalc} generalises to
   \begin{align*}
  &\lim_{u\downarrow 0}\frac{\mathbb{E}\left[\prod_{j=1}^{\ell}(Z_u^{(j)})^{\sigma_j}\right]-\prod_{j=1}^{\ell}a_j^{\sigma_j}}{u}\\
  &=\lim_{u\downarrow 0}\frac{\mathbb{E}\left[\big(\widebar{Z}_u^{(0)}\big)^{\ell_0}\prod_{j\colon\sigma_j\ge 2}\left(\big(Z_u^{(j)}\big)^{\sigma_j}\big(\widebar{Z}_u^{(j)}\big)^{\ell_j}\right)\right]-\|\beta_0\|^{\ell_0}\prod_{j\colon\sigma_j\ge 2}\left(a_j^{\sigma_j}\|\beta_j\|^{\ell_j}\right)}{u}
 \\
&=-2|\sigma|(|\sigma|\!-\!1\!+\!\theta)g_\sigma(\mathbf{a})+2\sum_{j\in A}\sigma_j(\sigma_j\!-\!1\!-\!\alpha)g_{\sigma-\square_j}(\mathbf{a})\\
 &\qquad\qquad\qquad\qquad \qquad\ \ \      +2\sum_{j\in A_0}\ell_j(\ell_j\!-\!1\!+\!\eta_j)g_{\sigma\ominus\square_{j+1}}(\mathbf{a}).\nonumber
  \end{align*}
  The remainder of the proof of Proposition \ref{prop:nocons} applies with no further changes.
\end{proof}

\begin{proof}[Proof of Theorem \ref{thm:twopargen}]
  Based on Proposition \ref{prop:nocons} and Corollary \ref{lm:mstar}, this will follow by linearity and general theory. Indeed, we will show that the linear span of 
  $\{m_\sigma^*,\,\sigma\in\mathcal{C}\}$ is the same as the linear span of $\{m_\sigma^\circ,\,\sigma\in\mathcal{C}\}$. Specifically, one inclusion follows from \eqref{eq:mstar}.
  For the other inclusion, we take $\sigma\in\mathcal{C}$ and inductively replace all runs of 1s in $\sigma$ using Lemma \ref{lm:1k}. 
More precisely, we will set up an induction to
  show a stronger statement that allows us to handle one run of 1s at a time. Consider $\sigma$ of the form 
  $\sigma=\sigma^*\concat\mathbf{1}_k\concat\sigma^\circ$ where $\sigma^*$ does not end with a 1 and $\sigma^\circ$ does not begin with a 1, but where we do allow 
  degenerate cases where $\sigma^*=\emptyset$ and/or $k=0$ and/or $\sigma^\circ=\emptyset$. For all such $(\sigma^*,k,\sigma^\circ)$, we define
  \begin{align*}
  m_{\sigma^*,\mathbf{1}_k\concat\sigma^\circ}^{*,\circ}(\beta)=\sum_{\overset{\scriptstyle U_j^*\in\beta,\,j\in[\ell^*]\colon\sigma_j^*\ge 2,}{\overset{\scriptstyle U_j^\circ\in\beta,\,j\in[k+\ell^\circ],}{\rm all\,strictly\,increasing}}}\|\beta_0^*\|^{\ell_0^*}&\prod_{j\in[\ell^*]\colon\sigma_j^*\ge 2}\left(\left({\rm Leb}(U_j^*)\right)^{\sigma_j^*}\|\beta_j^*\|^{\ell_j^*}\right)\\[-0.7cm]
    &\times\left(\prod_{j=1}^{k}{\rm Leb}(U_j^\circ)\right)\prod_{j=1}^{\ell^\circ}\left({\rm Leb}(U_{k+j}^\circ)\right)^{\sigma_j^\circ}
  \end{align*}
  where notation around \eqref{eq:betadecomp} has been superscripted by $^*$ or $^\circ$ in the natural way, so that, in particular 
  $(\beta_j^*,j\in[\ell^*]\colon\sigma_j^*\ge 2)$ is (the first part of) the decomposition of $\beta$ at the blocks $(U_j^*)$ so that 
  \begin{equation}\label{eq:betasplit}
  \beta=\beta_0^*\concat\Concat_{j\in[\ell^*]\colon\sigma_j^*\ge 2}(U_j^*\concat\beta_j^*)\concat\Concat_{j\in[k+\ell^\circ]\colon\sigma_j^\circ\ge 2}(U_j^\circ\concat\beta_j^\circ)
  \end{equation}
  for some $(\beta_j^\circ)$ that do not feature in the above formula. We also note that for $j=\ell^*$, we have $\ell_j^*=0$, so the formula does not depend on the right-most 
  $\beta_j^*$ either, and we will also denote it by $\beta_0^\circ$. With this notation, we claim that for any $(\sigma^*,\mathbf{1}_k\concat\sigma^\circ)$ of this form, 
  $m_{\sigma^*,\mathbf{1}_k\concat\sigma^\circ}^{*,\circ}$ is in the linear span of $\{m_\sigma^*,\,\sigma\in\mathcal{C}\}$. We prove this by strong induction on the 
  number $r$ of 1s in $\mathbf{1}_k\concat\sigma^\circ$ that are adjacent to at least one other 1. We call these 1s ``adjacent 1s'' for simplicity. If $r=0$, then 
  $m_{\sigma^*,\mathbf{1}_k\concat\sigma^\circ}^{*,\circ}=m_\sigma^*$ for $\sigma=\sigma^*\concat\mathbf{1}_k\concat\sigma^\circ$. The case $r=1$ is void. Assuming 
  that the induction hypothesis holds up to and including a given $r\ge 1$, we consider $(\sigma^*,\mathbf{1}_k\concat\sigma^\circ)$ such that 
  $\mathbf{1}_k\concat\sigma^\circ$ has $r+1$ adjacent 1s. If $k=0$ or $k=1$, we can include part of $\mathbf{1}_k\concat\sigma^\circ$ in $\sigma^*$ without changing
  $m_{\sigma^*,\mathbf{1}_k\concat\sigma^\circ}^{*,\circ}$. Therefore, we may assume without loss of generality that $k\ge 2$. Then, using the notation introduced in and below \eqref{eq:betasplit}, 
  \begin{align*}
  m_{\sigma^*,\mathbf{1}_k\concat\sigma^\circ}^{*,\circ}(\beta)=\sum_{\overset{\scriptstyle U_j^*\in\beta,\,j\in[\ell^*]\colon\sigma_j^*\ge 2,}{\overset{\scriptstyle U_j^\circ\in\beta,\,j\in[k+\ell^\circ]\setminus[k],}{\rm all\,strictly\,increasing}}}\|\beta_0^*\|^{\ell_0^*}&\prod_{j\in[\ell^*]\colon\sigma_j^*\ge 2}\left(\left({\rm Leb}(U_j^*)\right)^{\sigma_j^*}\|\beta_j^*\|^{\ell_j^*}\right)\\[-0.8cm]
        &\times\|\beta_0^\circ\|^km_{\mathbf{1}_k}^\circ\left(\|\beta_0^\circ\|^{-1}\beta_0^\circ\right)\prod_{j=1}^{\ell^\circ}\left({\rm Leb}(U_{k+j}^\circ)\right)^{\sigma_j^\circ}\!.
  \end{align*}
  We apply Lemma \ref{lm:1k} to $m_{\mathbf{1}_{k}}^\circ\big(\|\beta_0^\circ\|^{-1}\beta_0^\circ\big)$. The resulting linear combination contains only terms of the same form and we
  can write 
  \begin{equation}\label{eq:lincomb}
  m_{\sigma^*,\mathbf{1}_k\concat\sigma^\circ}^{*,\circ}=\frac{1}{k!}m_{\sigma^*\concat\mathbf{1}_k,\sigma^\circ}^{*,\circ}-\sum_{\rho\in\mathcal{C}_k\setminus\{\mathbf{1}_k\}}\left(\prod_{i=1}^{\ell(\rho)}\frac{1}{\rho_i!}\right)m_{\sigma^*,\rho\concat\sigma^\circ}^{*,\circ}.
  \end{equation}
  As $\mathbf{1}_k$ has $k\ge 2$ adjacent 1s, the compositions $\sigma^\circ$ and $\rho\concat\sigma^\circ$ for $\rho\in\mathcal{C}_k\setminus\{\mathbf{1}_k\}$ have
  at most $r$ adjacent 1s. By the induction hypothesis, all terms in \eqref{eq:lincomb} can be written as linear combinations of members of $\{m_{\sigma^\prime}^*,\,\sigma^\prime\in\mathcal{C}\}$, and this completes the induction step. 

  In particular, $m^\circ_\sigma=m^{*,\circ}_{\emptyset,\sigma}$ is in the span of $\{m_{\sigma^\prime}^*,\sigma^\prime\in\mathcal{C}\}$ for all $\sigma\in\mathcal{C}$. 
  Since $\cA_{\alpha,\theta}$ agrees with the generator of ${\tt PDIPE}(\alpha,\theta)$ on 
  $\{m_\sigma^\circ,\,\sigma\in\mathcal{C}\}$, this completes the proof, as in \cite[Theorem 1.2]{Paper1-3}.
\end{proof}

\section*{Acknowledgements}

We thank Soumik Pal for his contributions in early discussions of this project.

\bibliographystyle{abbrv}
\bibliography{AldousDiffusion4}
\end{document}